\documentclass[11pt]{article}
\usepackage[includehead, includefoot, heightrounded, left=0.8in, top=1.6cm, bottom=2.4cm, right=0.8in]{geometry}
\usepackage{tgpagella}

\usepackage{amsmath, amssymb, amsthm}
\usepackage{dsfont}
\usepackage{bbm} 
\usepackage[svgnames, dvipsnames, usenames]{xcolor}
\usepackage{graphicx} 
\usepackage{tikz, tikz-cd, pgfplots}
\pgfplotsset{compat=newest}
\usepackage{enumerate}

\newcommand{\EE}{\mathbb E}
\newcommand{\NN}{\mathbb N}
\newcommand{\RR}{\mathbb R}
\newcommand{\PP}{\mathbb P}
\newcommand{\vocab}[1]{\textbf{\boldmath #1}}
\newcommand{\eps}{\varepsilon}
\newcommand*{\circled}[2][]{\tikz[baseline=(C.base)]{
    \node[inner sep=0pt] (C) {\vphantom{1g}#2};
    \node[draw, circle, inner sep=0.1pt, yshift=0.3pt] 
        at (C.center) {\vphantom{1g}};}}
\newcommand{\encircle}[1]{\,\circled #1\,}

\newcommand*{\circledBig}[2][]{\tikz[baseline=(C.base)]{
    \node[inner sep=0pt] (C) {\vphantom{1g}#2};
    \node[draw, circle, inner sep=5pt, yshift=0.3pt] 
        at (C.center) {\vphantom{1g}};}}
\newcommand{\encircleBig}[1]{\,\circledBig #1\,}

\newcommand{\mc}[1]{\mathcal{#1}}

\DeclareMathOperator{\freq}{freq}
\DeclareMathOperator{\WC}{WC}
\DeclareMathOperator{\pWC}{pWC}
\DeclareMathOperator{\pat}{pat}
\DeclareMathOperator{\Tree}{Tree}

\usepackage[pdftitle={High-dimensional permutons: theory and applications}, pdfauthor={Jacopo Borga and Andrew Lin},
colorlinks=true,linkcolor=NavyBlue,urlcolor=RoyalBlue,citecolor=PineGreen,bookmarks=true,bookmarksopen=true,bookmarksopenlevel=2,unicode=true,linktocpage]{hyperref}
\usepackage[capitalize, nameinlink]{cleveref}

\newtheorem{theorem}{Theorem}[section]
\newtheorem{corollary}[theorem]{Corollary}
\newtheorem{proposition}[theorem]{Proposition}
\newtheorem{lemma}[theorem]{Lemma}
\newtheorem{example}[theorem]{Example}
\newtheorem{definition}[theorem]{Definition}
\newtheorem{remark}[theorem]{Remark}

\title{High-dimensional permutons: theory and applications}
\author{
\begin{tabular}{c} Jacopo Borga\thanks{jborga@mit.edu}\\[-3pt]\small MIT \end{tabular}
\begin{tabular}{c} Andrew Lin\thanks{lindrew@stanford.edu}\\[-3pt]\small Stanford University \end{tabular} 
}
\date{}

\begin{document} 

\maketitle

\begin{abstract}
Permutons, which are probability measures on the unit square $[0, 1]^2$ with uniform marginals, are the natural scaling limits for sequences of (random) permutations. 

We introduce a $d$-dimensional generalization of these measures for all $d \ge 2$, which we call \emph{$d$-dimensional permutons}, and extend -- from the two-dimensional setting -- the theory to prove convergence of sequences of (random) $d$-dimensional permutations to (random) $d$-dimensional permutons. 

Building on this new theory, we determine the random high-dimensional permuton limits for two natural families of high-dimensional permutations. First, we determine the $3$-dimensional permuton limit for Schnyder wood permutations, which bijectively encode planar triangulations decorated by triples of spanning trees known as Schnyder woods. Second, we identify the $d$-dimensional permuton limit for $d$-separable permutations, a pattern-avoiding class of $d$-dimensional permutations generalizing ordinary separable permutations. 

Both high-dimensional permuton limits are random and connected to previously studied universal 2-dimensional permutons, such as the Brownian separable permutons and the skew Brownian permutons, and share interesting connections with objects arising from random geometry, including the continuum random tree, Schramm--Loewner evolutions, and Liouville quantum gravity surfaces.
\end{abstract}

\tableofcontents

\begin{figure}[h!]
    \vspace{1 cm}
	\includegraphics[width=0.38\textwidth]{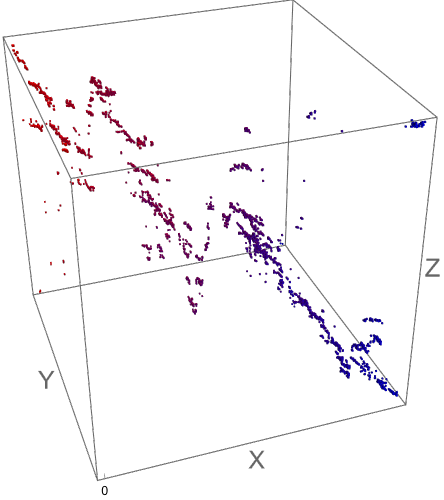}
	\hspace{5mm}
	\includegraphics[width=0.38\textwidth]{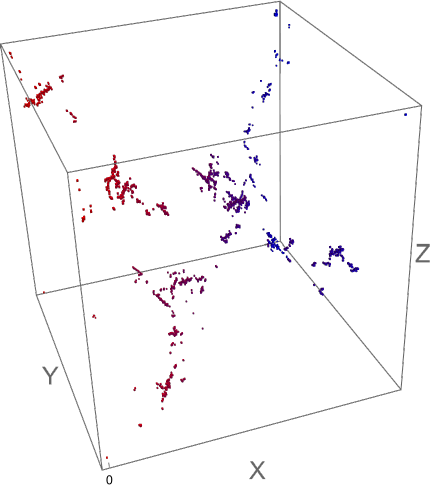}
	\includegraphics[width=0.22\textwidth]{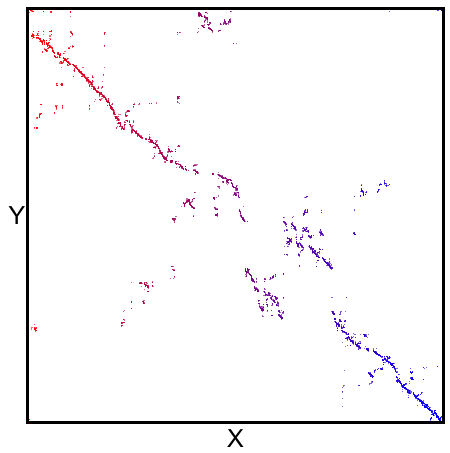}	
	\includegraphics[width=0.22\textwidth]{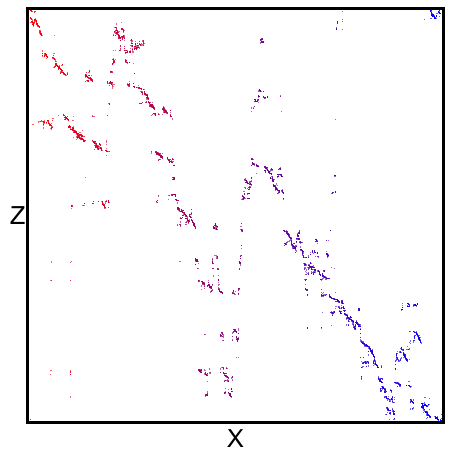}
	\hspace{0.1\textwidth}
	\includegraphics[width=0.22\textwidth]{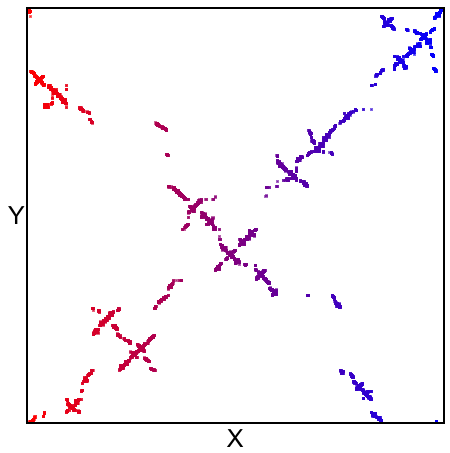}	
	\includegraphics[width=0.22\textwidth]{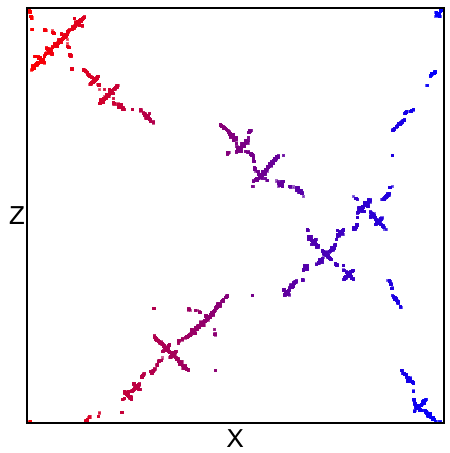}
	\caption{\label{fig:3dpermutons} Simulations for two 3-dimensional permutons with their respective 2-dimensional marginal permutons.  \textbf{Left:} The 3-dimensional permuton associated with a permutation of size 10000 sampled from the Schnyder wood permuton of \cref{mainresult3}.   \textbf{Right:} The 3-dimensional permuton associated with a permutation of size 10000 sampled from the Brownian separable $3$-permuton of \cref{mainresult4}. Animated simulations can be found at \href{http://www.jacopoborga.com/2024/12/27/high-dimensional-permutons-the-schnyder-wood-and-brownian-separable-d-permuton/}{this webpage}.}
\end{figure}

\section{Introduction} 

Over the last decade, there has been a growing interest in understanding scaling limits of random permutations uniformly sampled from permutation classes or under various non-uniform measures. Such limits, called \emph{permutons}, arise naturally in pattern avoidance, statistics, random geometry, and other areas of combinatorics and probability, and their description provides a way to study the global behavior of large permutations. The theory of permutons has been developed in previous work~\cite{hoppen2012limits,Bassino_2020}, and we now recall the main definitions. 

For any permutation $\sigma = (\sigma(1), \cdots, \sigma(n))$ of size $n$, we may associate to it a probability measure on the unit square $[0, 1]^2$ given by
\begin{equation}\label{eq:permut-perm}
    \mu_{\sigma}(d\vec{x}) = n \cdot \mathds{1}\left\{\sigma(\lceil nx_1 \rceil) = \lceil nx_2 \rceil \right\} d\vec{x},
\end{equation}
where $\vec{x} = (x_1, x_2)\in [0, 1]^2$ and $d\vec{x}$ denotes the Lebesgue measure on $[0, 1]^2$. Such a probability measure assigns a mass $\frac{1}{n}$ to a $\frac{1}{n} \times \frac{1}{n}$ square corresponding to the rescaled point $(i, \sigma(i))$, and any such measure has uniform marginals. An example is shown in the left-hand side of  \cref{fig:permuton}.

\begin{figure}[h]
\begin{minipage}{.32\textwidth}
\centerline{\includegraphics[width=4.8cm]{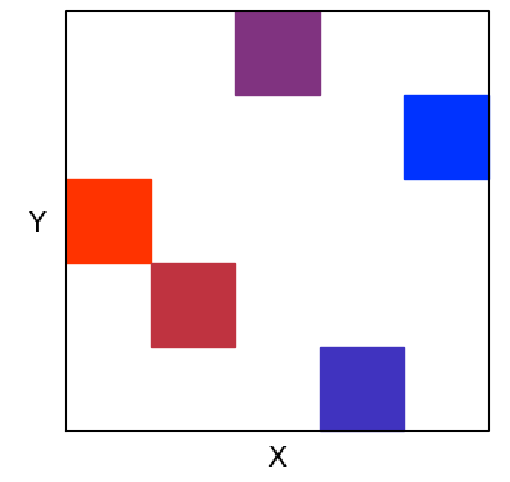}}
\end{minipage}%
\begin{minipage}{.32\textwidth}
\centerline{\includegraphics[width=0.92\textwidth]{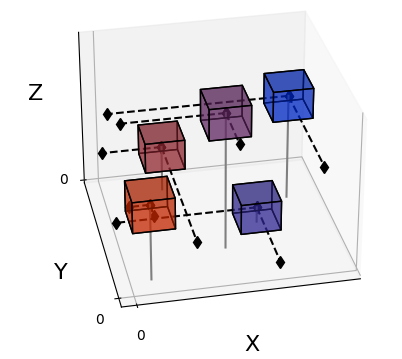}}
\end{minipage}%
\begin{minipage}{.32\textwidth}
\centerline{\includegraphics[width=\textwidth]{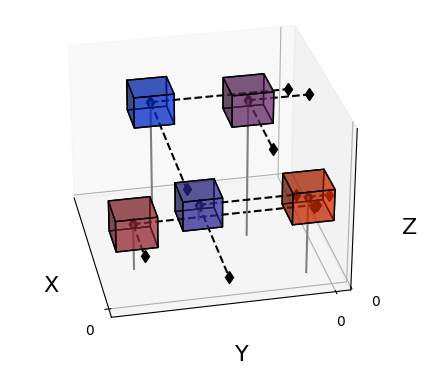}}
\end{minipage}
\caption{In all images, the colors are meant only for visual aid (with boxes colored from red to blue as the $x$-coordinate ranges from $0$ to $1$). \textbf{Left:} The (2-dimensional) permuton associated to the permutation $\sigma = (3, 2, 5, 1, 4)$ (as defined in \cref{eq:permut-perm}). Each shaded square (of side length $\frac{1}{5}$) is uniformly assigned a total probability mass of $\frac{1}{5}$. \textbf{Middle-right:} The 3-dimensional permuton associated to the $3$-dimensional permutation $\sigma$ of size $5$ (as introduced in \cref{eq:defn-d-perm}) defined by $\sigma(1) = (1, 3), \sigma(2) = (5, 2), \sigma(3) = (2, 5), \sigma(4) = (3, 1), \sigma(5) = (4, 4)$, shown from two different angles. In shorthand, following \cref{eq:perm-shorthand}, this permutation may also be written as $((1,5,2,3,4), (3,2,5,1,4))$. Notice that the left $2$-permuton is a marginal of the right $3$-permuton (specifically the projection onto the coordinates $(X, Z)$). \label{fig:permuton}}
\end{figure}

We say that a Borel probability measure $\mu$ on $[0, 1]^2$ is a \vocab{permuton} if it has uniform marginals, that is, 
\begin{equation*}
    \mu([0, 1] \times [x, y]) = \mu([x, y] \times [0, 1]) = y-x\quad\text{ for all }\quad 0 \le x \le y \le 1.
\end{equation*} 
Additionally, we say that $\mu$ is the \vocab{permuton limit} of a sequence of permutons $\mu_n$ if we have convergence
\[
    \int_{[0, 1]^2} f d\mu_n \to \int_{[0, 1]^2} f d\mu
\]
for all (bounded) continuous functions $f: [0, 1]^2\to \mathbb R$. Accordingly, we thus say that $\mu$ is the \vocab{permuton limit} of a sequence of permutations $\sigma_n$ if $\mu_{\sigma_n} \to \mu$.

Various works establishing convergence to explicit (deterministic and random) permutons exist in the literature. For instance, \cite{Starr_2009} describes the deterministic permuton limit for Mallows permutations, and \cite{squareperms,almostsquareperms} describe the random permuton limit for square and almost-square permutations (a random measure on $[0, 1]^2$ whose source of randomness is a single random parameter $z \in [0, 1]$). The first remarkable examples of fractal, canonical and universal random permuton limits, called \vocab{Brownian separable permutons}, were discovered and studied in the series of works \cite{separablelimit,Bassino_2020,maazoun2020brownian,Borga_2020,bassino2022scaling}. The Brownian separable permutons are very related to the continuum random tree (CRT) of Aldous \cite{aldouscrt} and describe the permuton limits for various pattern-avoiding permutation classes. More recently, \cite{borga2023skew} constructs the \vocab{skew Brownian permutons}, a two-parameter family of fractal, canonical and universal permutons which generalize the Brownian separable permutons. These permutons are closely linked with universal objects studied in random geometry, such as Liouville quantum gravity surfaces (LQGs) and Schramm--Loewner evolution curves (SLEs); see for instance \cite{borga2022permutons,borga2023baxter}. Skew Brownian permutons describe the permuton limits for various generalized pattern-avoiding permutation classes, and we elaborate more on their construction in \cref{backgroundsection}.

The general theory of permutons is by now well-developed, and the usual strategy for proving convergence is by studying pattern frequency~\cite{hoppen2012limits,Bassino_2020}, \emph{a.k.a.}\ pattern occurrences, as we now explain. For a permutation $\sigma$ of size $n$, a permutation $\tau$ of size $k \le n$, and a subset $I = \{i_1, \cdots, i_k\}$ of $[n]:=\{1, \cdots, n\}$ of size $k$, we say that $\tau$ is the \vocab{pattern} of $\sigma$ on $I$ if the values $\sigma(i_1), \cdots, \sigma(i_k)$ are in the same relative order as $\tau(1), \cdots, \tau(k)$, and we define the \vocab{pattern frequency}
\begin{equation}\label{2dpattern}
    \freq(\tau, \sigma) = \frac{1}{\binom{n}{k}} \#\left\{I \subset [n] \,:\, \tau = \pat_I(\sigma)\right\}
\end{equation}
to be the fraction of all $k$-subsets of $[n]$ which yield the pattern $\tau$. For a sequence of random permutations $\sigma_n$, it has been proven in \cite{Bassino_2020}, building on results of \cite{hoppen2012limits}, that convergence of $\mu_{\sigma_n} \to \mu$ is equivalent to the statement that for all patterns $\tau$, the numbers $\EE[\freq(\tau, \sigma_n)]$ converge to constants $c_\tau$. We stress the remarkable (and rather surprising) fact that only convergence in \emph{expectation} of pattern frequencies is needed to prove permuton convergence. Thus, the construction of a permuton limit for permutations often reduces to computing or estimating certain enumerative combinatorial quantities.

\medskip

The first goal of our paper is to extend this theory to higher dimensional permutations and permutons. Viewing a permutation as a map $\sigma: [n] \to [n]$, a \vocab{$d$-dimensional permutation} is analogously a map $[n] \to [n]^{d-1}$ which is a permutation restricted to each coordinate. Much like an ordinary permutation corresponds to a measure on $[0, 1]^2$ (recall \cref{eq:permut-perm}), a $d$-dimensional permutation corresponds to a measure on $[0, 1]^d$. We thus introduce $d$-dimensional permutons and establish an analogous condition for convergence of high-dimensional permutations to these limiting objects -- see \cref{sect:main 1} for further details.

Our second goal is to describe the permuton limit of certain natural high-dimensional permutations uniformly sampled from two different families. The first family consists of certain 3-dimensional permutations encoding \vocab{Schnyder wood triangulations} \cite{schnyderoriginal}, which are a well-studied family of planar maps decorated by three spanning trees; see for instance \cite{schnyderoriginal, MATHESON1996565, NAGAMOCHI2004223, felsner2008schnyder}. In particular, Schnyder woods are known to converge in the peanosphere sense to SLE-decorated LQG surfaces \cite{li2022schnyder}. The second family consists of $d$-dimensional generalizations of separable permutations, which are in bijection with guillotine partitions~\cite{asinowski2008separable}. In both cases, we show that the limiting permutons can be described in terms of natural objects arising in random geometry, such as the CRT, SLEs, and LQGs. This extends evidence of universality shown in the two-dimensional case by the Brownian separable permutons and skew Brownian permutons -- see \cref{sect:main2} for further details.

\medskip

In the remainder of this introduction, we describe our results more precisely and provide the necessary background to state them rigorously.

\subsection{Characterization of high-dimensional permuton convergence}\label{sect:main 1}

We begin by stating a precise definition for $d$-dimensional permutations.\footnote{Other notions of high-dimensional permutations have also been considered in the literature. See the end of \cref{sec:conc} for more details.}

\begin{definition}
Let $d \ge 2, n \ge 1$ be integers. A \vocab{$d$-dimensional permutation of size $n$} (or \vocab{$d$-permutation} for short) is a function $\sigma: [n] \to [n]^{d-1}$ such that the restriction of $\sigma$ to each coordinate is an ordinary permutation. We let $\mc{S}_{d,n}$ denote the set of $d$-dimensional permutations of size $n$.
\end{definition}

\noindent Following \cite{bonichon2022baxter}, for all $1 \le j \le d-1$, we denote the $j$--th coordinate of $\sigma(i)$ by $\sigma(i)^{(j)}$ -- in other words, we write 
\[
    \sigma(i) = \left(\sigma(i)^{(1)}, \sigma(i)^{(2)}, \cdots, \sigma(i)^{(d-1)}\right).
\]
Moreover, we use the shorthand 
\begin{equation}\label{eq:perm-shorthand}
    \sigma = \left(\sigma^{(1)}, \cdots, \sigma^{(d-1)}\right),
\end{equation}
where $\sigma^{(j)}$ is the $n$-tuple $(\sigma(1)^{(j)}, \cdots, \sigma(n)^{(j)})$, to write our $d$-permutations more concisely. We also denote the size of a permutation $\sigma$ by $|\sigma|$.

Much like ordinary $2$-dimensional permutations may be viewed as $n$ points $(i, \sigma(i))$ in an $n \times n$ grid, $d$-dimensional permutations may be viewed as $n$ points $(i,\sigma(i)^{(1)}, \sigma(i)^{(2)}, \cdots, \sigma(i)^{(d-1)})$ in the $d$-dimensional grid $[n]^d$. From this point of view, the next definition provides the natural candidate limiting objects for describing the permuton limits of high-dimensional permutations.

\begin{definition}
Let $d \ge 2$ be an integer. A \vocab{$d$-dimensional permuton} (or \vocab{$d$-permuton} for short) is a Borel probability measure on $[0, 1]^d$ whose $d$ $1$-dimensional marginals are each uniform on $[0, 1]$.
\end{definition}

In statistics, $d$-dimensional permutons are called \emph{copulae} and are typically studied from a rather different perspective compared to the probabilistic one.

For each $d$-dimensional permutation $\sigma_n$ of size $n$, we can associate to it the $d$-permuton 
\begin{equation}\label{eq:defn-d-perm}
    \mu_{\sigma_n}(d\vec{x}) = n^{d-1} \cdot \mathds{1}\left\{\sigma(\lceil nx_1 \rceil) = (\lceil nx_2 \rceil, \cdots, \lceil nx_d \rceil)\right\} d\vec{x},
\end{equation}
where $d\vec{x}$ is Lebesgue measure on $[0, 1]^d$. Graphically, we may describe $\mu_{\sigma_n}$ as shading in boxes of side length $\frac{1}{n}$ corresponding to the values of ${\sigma_n}$. The middle-right-hand side of \cref{fig:permuton} shows the permuton $\mu_{\sigma}$ corresponding to a particular $3$-permutation of size $5$ (where each box has total uniform mass $\frac{1}{5}$).

We wish to describe conditions under which a sequence of such permutons $\mu_{\sigma_n}$ converges in the weak topology to a limiting permuton $\mu$ as $n \to \infty$, meaning that for every (bounded) continuous function $f: [0, 1]^d \to \RR$, we have 
\[
    \int_{[0, 1]^d} f \, d\mu_{\sigma_n} \to \int_{[0, 1]^d} f \, d\mu.
\]
Earlier in the introduction, we described that in the 2-dimensional case this convergence is encoded by convergence of pattern frequencies, and we provided the definition in \cref{2dpattern}. The definitions for general $d$-dimensional permutations and permutons are as follows.

\begin{definition}\label{defnpattern}
Let $\tau \in \mc{S}_{d, k}$ and $\sigma \in \mc{S}_{d, n}$ for integers $1 \le k \le n$. For a subset $I \in \binom{[n]}{k}$ (meaning that $I$ is some subset of $\{1, \cdots, n\}$ of size $k$), we say that \vocab{$\tau$ is the pattern of $\sigma$ on $I$}, denoted $\tau = \pat_I(\sigma)$, if the values of $\sigma$ restricted to $I$ are in the same relative order as $\tau$. That is, if $I = \{i_1 < i_2 < \cdots < i_k\}$, then $\tau$ satisfies
\[
    \tau(a)^{(j)} < \tau(b)^{(j)} \iff \sigma(i_a)^{(j)} < \sigma(i_b)^{(j)}
\]
for all $a, b \in [k]$ and $1 \le j \le d-1$. Furthermore, we define
\[
    \text{occ}(\tau, \sigma) = \#\left\{I \in \binom{[n]}{k}: \tau = \pat_I(\sigma)\right\}\quad \text{ and } \quad \freq(\tau, \sigma) = \frac{\text{occ}(\tau, \sigma)}{\binom{n}{k}}.
\]
We also define $\text{freq}(\tau, \sigma) = 0$ if $|\tau| > |\sigma|$.
\end{definition}
Pattern frequency for permutons is similarly defined by reading off the relative order from $k$ iid sampled points from $\mu$.

\begin{definition}\label{defnsamplefrompermuton}
For a $d$-dimensional permuton $\mu$, let $\vec{x}_1, \cdots, \vec{x}_k$ be $k$ iid points sampled from $\mu$, and let 
\[
    P_\mu[k] = \text{Perm}(\vec{x}_1, \cdots, \vec{x}_k)
\] 
be the unique $d$-dimensional permutation of size $k$ in the same relative order as the points $\vec{x}_i$. (In this notation, we suppress the dependence on the points $\vec{x}_i$.) In other words, let $f:[k] \to [k]$ be such that for all $i$, $\vec{x}_i$ has the $f(i)$--th biggest first coordinate out of all points (so $f$ is a permutation almost surely). Then $P_\mu[k]$ is the unique (random) element of $\mc{S}_{d, k}$ such that
\[
    P_\mu[k](f(a))^{(j)} < P_\mu[k](f(b))^{(j)} \iff \vec{x}_a^{(j+1)} < \vec{x}_b^{(j+1)}
\]
for all $1 \le j \le d-1$ and all $a, b \in [k]$. For any $\tau \in \mc{S}_{d, k}$, we then define \vocab{the frequency of the pattern $\tau$ in the permuton $\mu$} as
\[
    \freq(\tau, \mu) = \int_{([0, 1]^d)^k} \mathds{1}\{P_\mu[k] = \tau\} \; \mu(d\vec{x}_1) \cdots \mu(d\vec{x}_k).
\]
\end{definition}

Our first main result is a generalization of \cite[Theorem 2.5]{Bassino_2020} to $d$-dimensional permutations and permutons.

\begin{theorem}\label{tfaemain}
Let $\sigma_n$ be a random $d$-dimensional permutation of size $n$ for all $n \in \NN$. Then the following are equivalent:
\begin{enumerate}[(1)]
    \item $\mu_{\sigma_n}$ converges weakly to a (possibly random) $d$-dimensional permuton $\mu$.
    \item The vector $(\freq(\tau, \sigma_n))_{\tau}$, indexed by $d$-dimensional permutations $\tau$ of all sizes, converges in distribution in the product topology to some random vector $(v_\tau)_\tau$.
    \item For all $d$-dimensional permutations $\tau$, $\EE[\freq(\tau, \sigma_n)]$ converges to some constant $c_\tau$.
    \item For each $k$, the random $d$-dimensional permutation $\pat_{I_{n,k}}(\sigma_n)$, where $I_{n, k}$ is a uniform subset of $[n]$ of size $k$ independent of $\sigma_n$, converges in distribution to a random $d$-dimensional permutation $\rho_k$. In other words, for all permutations $\tau$ of size $k$,
    \[
        \lim_{n \to \infty} \PP\left(\pat_{I_{n,k}}(\sigma_n) = \tau\right) = \PP(\rho_k =\tau).
    \]
\end{enumerate}
When these conditions hold, we have for all $d$-dimensional permutations $\tau$ of size $k$ that 
\[\EE[\text{freq}(\tau, \mu)] = \EE[v_\tau] = c_\tau = \PP(\rho_k = \tau),\] 
and furthermore the vector $(v_\tau)_\tau$ is identically distributed as $(\text{freq}(\tau, \mu))_\tau$ and $\rho_k=P_\mu[k]$ in distribution.
\end{theorem}

The proof of \cref{tfaemain} can be found in \cref{theorysection}. In that same section, we will also establish some other results to develop a complete theory for $d$-dimensional permutons.

\subsection{Random high-dimensional permuton limits for Schnyder wood permutations and $d$-separable permutations}\label{sect:main2}

Our next main results apply the general theory developed in \cref{theorysection} to prove convergence of two natural models of random high-dimensional permutations. 

\subsubsection{The 3-dimensional Schnyder wood permuton}

We first introduce Schnyder wood triangulations, following the conventions established in \cite{li2022schnyder}, and define the corresponding $3$-dimensional Schnyder wood permutations. We then motivate our interest in their study.

\begin{definition}
Let $M$ be a simple plane triangulation with an unbounded outer triangular face, and call the vertices and edges not adjacent to this unbounded face the \vocab{internal vertices} and \vocab{internal edges} of $M$, respectively. Equip $M$ with a \vocab{$3$-orientation}, meaning that all edges are directed so that each internal vertex has outdegree $3$. A \vocab{Schnyder wood} is a coloring of the internal edges of $M$ by the colors \{blue, green, red\} satisfying \vocab{Schnyder's rule}, which specifies that the edges adjacent to an internal vertex must be (in clockwise order) incoming blue, outgoing red, incoming green, outgoing blue, incoming red, outgoing green, and that each vertex has one outgoing edge of each color. We say that a Schnyder wood triangulation has \vocab{size} $n$ if there are $n$ internal vertices. 
\end{definition}

See \cref{fig:sampleschnyder} for an example. In any Schnyder wood triangulation, the blue edges form a tree rooted at one of the outer vertices, which we call the \vocab{blue root}; this tree is a spanning tree on the internal vertices and the blue root. We define the \vocab{red root} and \vocab{green root} similarly. Furthermore, Schnyder's rule forces all edges of a given color to be directed toward the root (more details can be found in the introduction of \cite{li2022schnyder}). This means that a Schnyder wood triangulation can be thought of as a triangulation decorated by three rooted trees, each of which contains all $n$ internal vertices.

\begin{definition}\label{schnyderpermutationdefn}
Let $M$ be a Schnyder wood triangulation of size $n$. The 3-dimensional \vocab{Schnyder wood permutation} 
\[
    \sigma_M = \left(\sigma_M^g, \sigma_M^r\right)
\]
associated with $M$ is the $3$-permutation $\sigma: [n] \to [n]^2$ defined as follows. Consider the depth-first clockwise traversals of the three trees of $M$. Assign each of the $n$ internal vertices a blue, red, and green label, which are the positions (from $1$ to $n$, ascending) in which the vertex appears in the corresponding traversals. Then $\sigma_M^g(i)$ (resp.\ $\sigma_M^r(i)$) is the green (resp.\ red) label of the vertex with blue label $i$.
\end{definition}

\cref{fig:sampleschnyder} below shows a Schnyder wood triangulation $M$ of size $10$, along with its Schnyder wood permutation $\sigma_M$. For example, $\sigma_M(3) = (5, 9)$, because the vertex visited third by the blue tree (that is, with blue label $3$) is visited fifth by the green tree and ninth by the red tree.

\begin{remark}\label{inverseofschnyder}
For ease of reading diagrams and following arguments in this paper, we note that $(\sigma_M^g)^{-1}$ -- that is, the inverse permutation of the first coordinate of $\sigma_M$ -- is thus the sequence of blue labels visited by the green tree's clockwise traversal, and similarly $(\sigma_M^r)^{-1}$ is the sequence of blue labels visited by the red tree's clockwise traversal. For example in \cref{fig:sampleschnyder}, $(\sigma_M^g)^{-1} = (10, 6, 1, 5, 3, 4, 9, 8, 2, 7)$ and $(\sigma_M^r)^{-1} = (8, 7, 2, 10, 9, 4, 6, 5, 3, 1)$.
\end{remark}

\begin{figure}[h]
\centerline{\includegraphics[width=14cm]{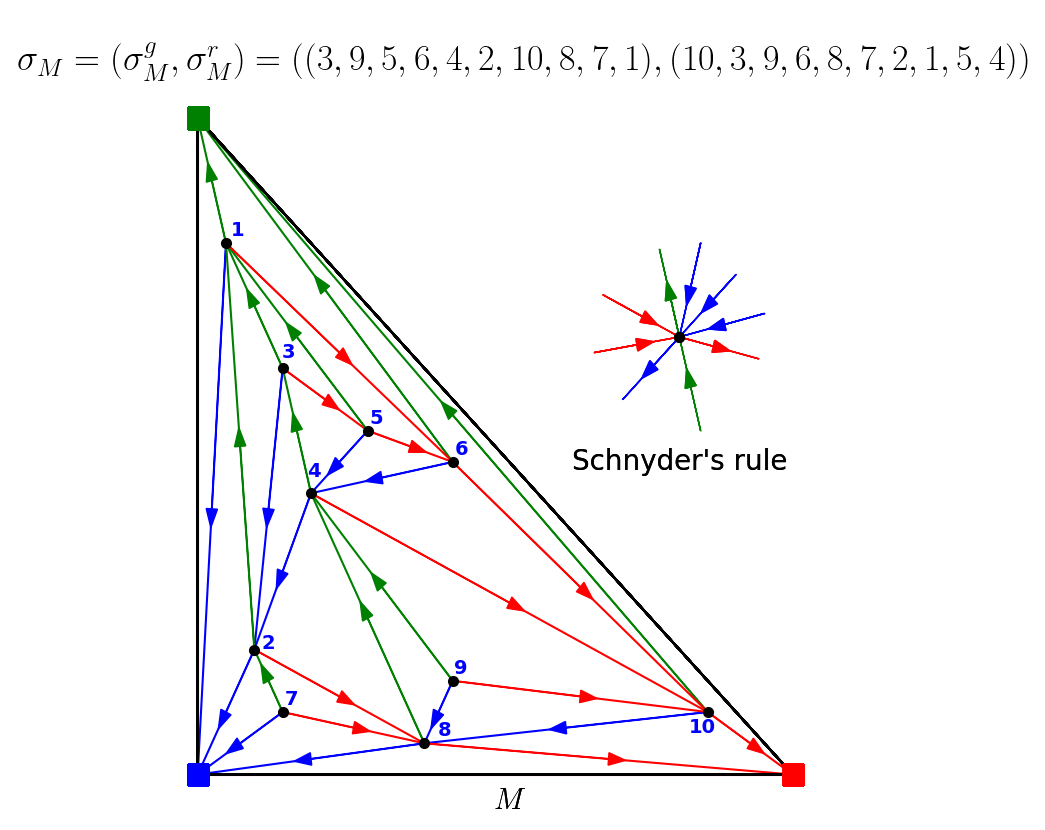}}
\caption{A Schnyder wood triangulation of size 10 with its corresponding Schnyder wood permutation on top. The roots are colored in their respective colors. We include only the blue labels for clarity. This Schnyder wood will be used as a running example also for several constructions in \cref{schnydersection}.}\label{fig:sampleschnyder}
\end{figure}

\medskip

Schnyder wood triangulations were first introduced in \cite{schnyderoriginal} as a way to embed planar graphs in the plane using straight line segments and with integer coordinates. Later applications of Schnyder wood triangulations include constructing dominating sets \cite{MATHESON1996565} and determining maximal planarity \cite{NAGAMOCHI2004223}. In more recent work, \cite{li2022schnyder} shows that Schnyder wood triangulations viewed as decorated planar maps converge to an LQG surface decorated with SLE curves in the peanosphere sense; specifically, the exploration curves of the three trees converge to three SLEs coupled in the imaginary geometry sense with pairwise angles all equal to $\frac{2\pi}{3}$.

As shown in \cref{schnyderwoodbijectperm}, the map $M \mapsto \sigma_M$ is a bijection between Schnyder wood triangulations of size $n$ and Schnyder wood permutations of size $n$. One of the motivations for studying Schnyder wood permutations is to explore how this connection between Schnyder wood triangulations and permutations at the discrete level extends to the continuum level; as explained later in  \cref{marginaldetermines,SLES-LQG} and the preceding discussion, the continuum connection is more subtle than one would expect.

Hence, we are interested in understanding the permuton limit for uniform large Schnyder wood permutations. The left-hand side of \cref{fig:largeexamples} illustrates an example -- notably, the limiting shape does not appear to be a deterministic measure on $[0, 1]^3$, and the support exhibits fractal behavior. Indeed, the limiting permuton for Schnyder wood permutations will be random, but some additional notation is needed to describe it precisely. We will do so in more detail in \cref{backgroundsection} and provide only the main features here. Specifically, we will introduce the skew Brownian permutons, a two-parameter family of universal limiting permutons introduced in \cite{borga2023skew} and further studied in \cite{borga2021permuton}, and then explain how they are used to construct the permuton limit of Schnyder wood permutations. 

\begin{figure}[!h]
\centering
\includegraphics[width=0.45\textwidth]{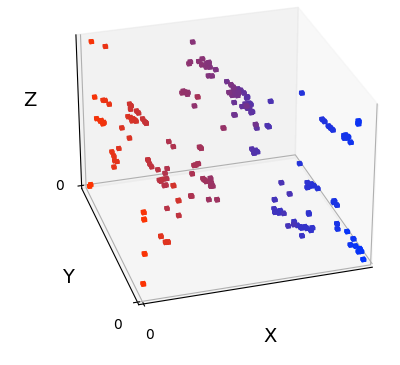}
\hspace{0.05\textwidth}
\includegraphics[width=0.45\textwidth]{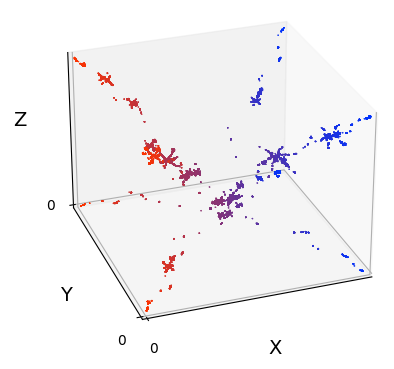}
\hspace{0.05\textwidth}
\centerline{\includegraphics[width=\textwidth]{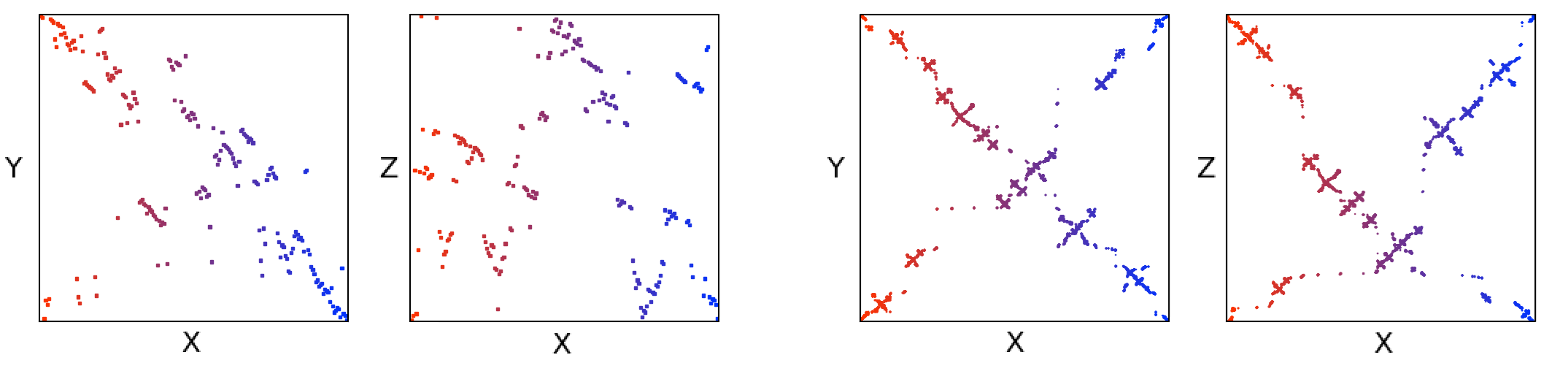}}
\caption{The permutons associated to a uniformly sampled Schnyder wood permutation of size $200$ (left) and a uniformly sampled $3$-separable permutation of size $40077$ (right). Much like in \cref{fig:permuton}, the color indicates the $x$-coordinate of the box (ranging from red to blue as $x$ ranges from $0$ to $1$). The two-dimensional marginals (projecting down onto the first and second coordinates and onto the first and third coordinates) are shown below their corresponding permutons. We highlight that the main difference with \cref{fig:3dpermutons} is that here the two simulations are for uniform Schnyder wood and 3-separable \emph{permutations}, respectively, while in \cref{fig:3dpermutons} the simulations are obtained by sampling permutations from the corresponding \emph{limiting permutons}, as discovered in \cref{mainresult3} and \cref{mainresult4}.}
\label{fig:largeexamples}
\end{figure}

\medskip

Let $\rho \in (-1,1)$ and $q \in [0, 1]$ be parameters, and let $W_\rho=(X_\rho, Y_\rho)$ be a two-dimensional \vocab{Brownian excursion of correlation $\rho$ in the first quadrant} on the time interval $[0, 1]$, which is a two-dimensional Brownian motion of correlation $\rho$ conditioned to stay in $\RR_{\ge 0}^2$ and start (at time 0) and end (at time 1) at the origin $(0, 0)$. Using a specific $q$-dependent system of stochastic differential equations\footnote{The specific stochastic differential equations used here are detailed in \cref{skewbrowniansdedefn} but are not necessary for understanding the statement of our main result in \cref{mainresult3}.} driven by $W_\rho$ and indexed by $u \in [0, 1]$, we may define a set of (coupled) one-dimensional stochastic processes $\left(Z_{\rho, q}^{(u)}(t)\right)_{t\in[0,1]}$. These processes may then be used to define the following random function on $[0, 1]$:
\begin{equation}\label{skewbrowniansdedefn2}
\phi_{\rho, q}(t) = \text{Leb}\left(\left\{x \in [0, t): Z_{\rho, q}^{(x)}(t) < 0\right\} \cup \left\{x \in [t, 1]: Z_{\rho, q}^{(t)}(x) \ge 0\right\}\right).
\end{equation}
We then define the \vocab{skew Brownian permuton $\mu_{\rho, q}$ driven by $W_\rho$ and of skewness $q$} to be the (random) measure on $[0, 1]^2$ satisfying, for any measurable set $A\subseteq[0,1]^2$,
\[
    \mu_{\rho, q}(A) = (\text{Id}, \phi_{\rho, q})_\ast \text{Leb}(A) = \text{Leb}\left(\left\{t \in [0, 1]: (t, \phi_{\rho, q}(t)) \in A\right\}\right).
\]
The interpretation of $\mu_{\rho, q}$ will become clearer in the subsequent sections of the paper, but intuitively, $Z_{\rho, q}^{(u)}$ is a collection of paths on $[0,1]$ indexed by $u\in[0, 1]$, and $\phi_{\rho, q}(t)$ encodes what proportion of these paths lie below the path $Z_{\rho, q}^{(t)}$. $\mu_{\rho, q}$ is then the ``continuous permutation'' corresponding to this encoding.

In previous work \cite{borga2022scaling, borga2021permuton}, the permuton limits for uniform Baxter, strong-Baxter, and semi-Baxter permutations have all been determined to be various elements of this universality class. Specifically, the limits for the three classes of permutations are approximately $\mu_{-0.5, 0.5}$, $\mu_{-0.22, 0.3}$, and $\mu_{-0.81, 0.5}$, respectively. Our next main result describes the permuton limit of Schnyder wood permutations in terms of $\mu_{\rho, q}$ for a fourth choice of these parameters. 

\begin{theorem}\label{mainresult3}
Let $\sigma_{n}$ be a uniform 3-dimensional Schnyder wood permutation of size $n$. Let
\[\rho = -\frac{\sqrt{2}}{2}\qquad \text{and} \qquad q = \frac{1}{1 + \sqrt{2}}.\]
Then $\mu_{\sigma_{n}}$ converges in distribution to the \vocab{Schnyder wood permuton} $\mu_S$ defined as follows. Let $W_\rho$ be a two-dimensional Brownian excursion of correlation $\rho$ and $W'_\rho$ be its time-reversal. Let $\mu_{\rho, q}^g$ and $\mu_{\rho, q}^r$ be (coupled) skew Brownian permutons of parameter $(\rho, q)$, driven by $W'_\rho$ and $W_\rho$ respectively, with associated random functions $\phi^g_{\rho,q}$ and $\phi^r_{\rho,q}$ as in \cref{skewbrowniansdedefn2}. Then for each measurable subset $A \subseteq [0, 1]^3$, the Schnyder wood permuton assigns to it the measure
\begin{equation}\label{schnyderpermutonexplicit}
    \mu_S(A) = \text{Leb}\left(\left\{t \in [0, 1]: \left(t \,,\,  \phi^g_{\rho,q}(1 - t) \, , \,
    1 - \phi^r_{\rho,q}(t)\right) \in A\right\}\right).
\end{equation}
In particular, the two 2-dimensional marginals $\mu_S^{g}$ and $\mu_S^{r}$ of $\mu_S$, defined for all measurable sets $B, C \subseteq [0,1]$ by
\begin{equation*}
\mu_S^{g}(B\times C)=\mu_S(B \times C \times [0,1]) \qquad \text{ and } \qquad \mu_S^{r}(B\times C) = \mu_S(B \times [0,1] \times C), 
\end{equation*}
satisfy
\begin{equation}\label{eq:weifvwebfowep}
\mu_S^{g}=f_\ast(\mu_{\rho, q}^g) \quad \text{ and } \quad \mu_S^{r}=h_\ast(\mu_{\rho, q}^r), 
\end{equation}
where the maps $f, h: [0, 1]^2 \to [0, 1]^2$ are defined by $f(x, y) = (1-x, y)$ and $h(x, y) = (x, 1-y)$.
\end{theorem}

For simplicity, we will often write
\[
    \mu_S = \left(\mu_S^{g}\,,\,\mu_S^{r}\right).
\]
As shown in \cref{fig:schnyder_size3}, there are Schnyder wood triangulations $M$ (of the same size) with the same marginal $\sigma_M^g$ but different marginal $\sigma_M^r$. This immediately implies that, at the discrete level, $\sigma_M^g$ (or $\sigma_M^r$) alone does not determine $\sigma_M$. Surprisingly, the situation is quite different in the continuum limit.

\begin{figure}[ht]
\centerline{\includegraphics[width=16cm]{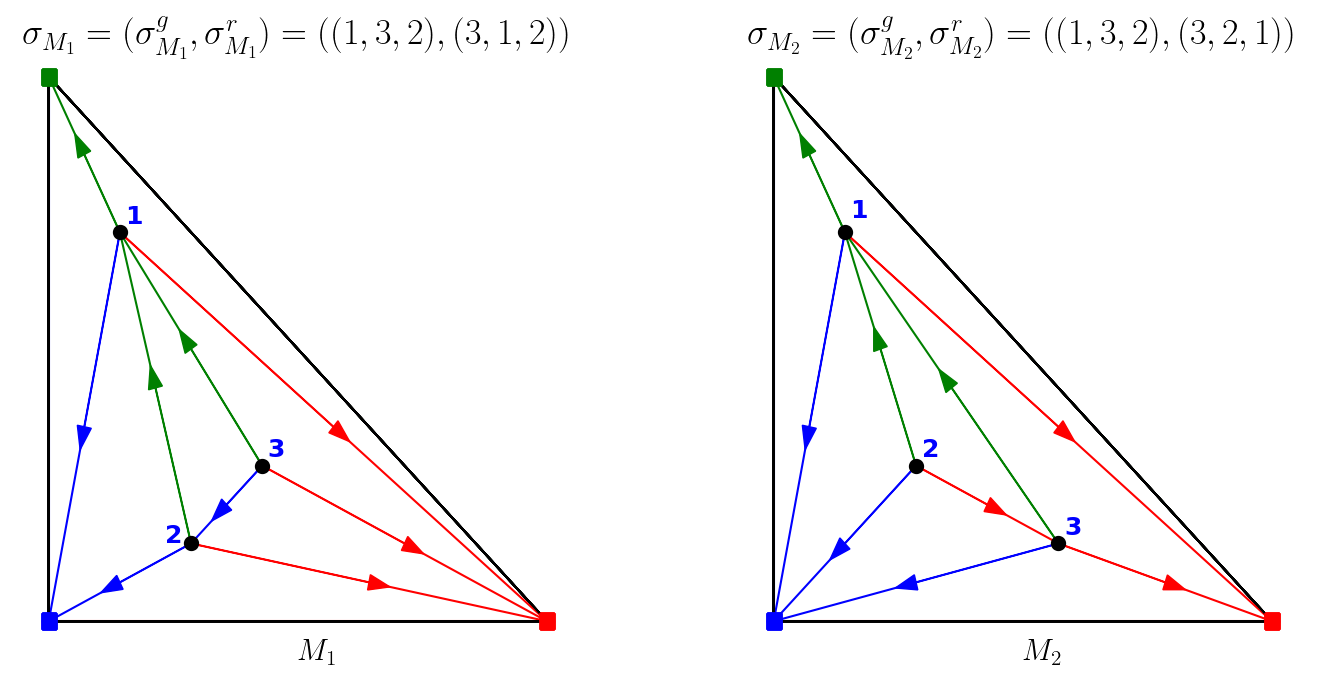}}
\caption{\label{fig:schnyder_size3}Two Schnyder woods $M_1$ and $M_2$ of size $3$, where $\sigma_{M_1}^g = \sigma_{M_2}^g$ but $\sigma_{M_1}^r \ne \sigma_{M_2}^r$.}
\end{figure}

\begin{proposition}\label{marginaldetermines}
The Schnyder wood permuton $\mu_S = \left(\mu_S^{g}\,,\,\mu_S^{r}\right)$ is determined by the marginal $2$-permuton $\mu_S^{g}$ (or by the marginal $2$-permuton $\mu_S^{r}$). That is, there exists a deterministic measurable function $F$ such that $\mu_S=F(\mu_S^{g})$.
\end{proposition}

The proof of \cref{marginaldetermines} can be found at the end of \cref{permutonsubsection}, and it follows combining \cref{mainresult3} and \cite[Theorem 1.1]{borgagwynne2024}.

\medskip

For the reader familiar with the construction of skew Brownian permutons in terms of SLEs and LQG surfaces~\cite{borga2023skew}, we point out that the coupling of the two marginals of $\mu_S = \left(\mu_S^g \, , \, \mu_S^r \right)$ can be also geometrically described as follows. Let 
\[
    \kappa=16, \quad \gamma=1,\quad \text{and} \quad \chi=\frac{\sqrt \kappa}{2}- \frac{2}{\sqrt \kappa}= 3/2.
\]
Let $\hat{h}$ be a whole-plane Gaussian free field (viewed modulo additive multiples of $2\pi\chi$), and let $(\eta^b,\eta^r,\eta^g)$ be three whole-plane space-filling SLE$_\kappa$ counter-flow lines of $\hat{h}$ from $\infty$ to $\infty$ of angles $\left(0,\frac{2\pi}{3},\frac{4\pi}{3}\right)$, respectively.

Let $h$ be a random generalized function on $\mathbb{C}$, independent of $\hat{h}$, corresponding to a singly marked unit-area $\gamma$-Liouville quantum sphere $(\mathbb C, h, \infty)$, and let $\mu_h$ be its associated measure of the $\gamma$-LQG area. We parametrize each of $\eta^b$, $\eta^r$, and $\eta^g$ by $\mu_h$-mass, i.e., 
\begin{equation*}
	\mu_h(\eta^b([0,t])) = \mu_h(\eta^r([0,t])= \mu_h(\eta^g([0,t]) = t,\qquad\text{for all } t\in[0,1].
\end{equation*} 
Let $\psi^{b,\circ}:[0,1]\to[0,1]$ for $\circ\in\{g,r\}$ be Lebesgue measurable functions such that
\begin{equation}\label{eq:wefivfvweofu}
    \eta^b(t)=\eta^{\circ}\left(\psi^{b,\circ}(t)\right),\qquad \text{for all $t\in[0,1]$.}
\end{equation}
The $2$-permuton $\mu^{b,\circ}$ associated with $(\mu_h,\eta^b,\eta^\circ)$ for $\circ\in\{r,g\}$ is defined by
\begin{equation}\label{eq:permutonfromcoupledcurves}
    \mu^{b,\circ}(A) = \text{Leb}\left(\left\{t \in [0, 1]: (t, \psi^{b,\circ}(t)) \in A\right\}\right).
\end{equation}
\begin{proposition}\label{SLES-LQG}
The Schnyder wood permuton $\mu_S=(\mu_S^{g}\,,\,\mu_S^{r})$ of \cref{mainresult3} satisfies
\begin{equation}\label{eq:wefivoweuiqbfqwop}
    \mu_S \stackrel{d}{=} \text{Leb}\left(\left\{t \in [0, 1]: \left(t,  \psi^{b,g}(t), \psi^{b,r}(t)\right) \in A\right\}\right).
\end{equation}
Moreover, $\mu_S^{g}\stackrel{d}{=}\mu^{b,g}$ and $\mu_S^{r}\stackrel{d}{=}\mu^{b,r}$.
\end{proposition}

The proof of \cref{SLES-LQG} can be found at the end of \cref{permutonsubsection}. In words, this result justifies that $\mu_S=(\mu_S^g,\mu_S^r)$ can be obtained by comparing the space-filling SLEs $\eta^b$ and $\eta^g$ of angle $0$ and $\frac{4\pi}{3}$ (giving $\mu_S^g$)  and the space-filling SLEs $\eta^b$ and $\eta^r$ of angle $0$ and $\frac{2\pi}{3}$ (giving $\mu_S^r$). Moreover, in the proof of \cref{marginaldetermines}, we will show as a corollary of \cite[Theorem 1.1]{borgagwynne2024} that $\mu_S^g$ (or equivalently $\mu_S^r$) determines these three curves together with the LQG-sphere $(\mathbb C, h, \infty)$ (up to conjugation).

\subsubsection{The $d$-dimensional Brownian separable permuton}

Our other class of high-dimensional permutations considered in this paper are $d$-separable permutations. These permutations were introduced in \cite{asinowski2008separable} and also described in \cite{bonichon2022baxter} as a generalization of the well-studied (2-dimensional) \vocab{separable permutations}. 

Before we give the definition, we describe the main features of the 2-dimensional case, and we point the reader to \cite{separablelimit} for more details. A permutation $\sigma \in \mc{S}_{2, n}$ is \vocab{separable} if it avoids the patterns $(2, 4, 1, 3)$ and $(3, 1, 4, 2)$. Such permutations have various nice properties. For example, they are exactly the permutations that can be sorted by series of pop-stacks \cite{popstacks}, they are the smallest nontrivial substition-closed permutation class \cite{Bassino_2020}, and they are in bijection with slicing floorplans \cite{floorplans} and with Schr\"oder trees \cite{separablelimit}. This last bijection is due to the fact that separable permutations are exactly those which can be repeatedly split into two contiguous parts, such that all entries in one part are larger than all entries in the other. In other words, separable permutations are those that can be built by what are called the ``direct sum'' and ``skew sum'' permutation operations.

We now give a precise definition of separable permutations in any dimension:

\begin{definition}\label{dseparablepermutationdefn}
Let $d \ge 2$ be an integer. Given two $d$-permutations $\sigma_1, \sigma_2$ of size $n_1, n_2$ and a \vocab{sign sequence} $s = (s_1, \cdots, s_{d-1}) \in \{\pm 1\}^{d-1}$ of length $(d-1)$, we define the \vocab{block sum permutation} $\sigma_1 \encircle{s} \sigma_2$ by placing $\sigma_2$ ``above'' $\sigma_1$ in the $j$--th coordinate if $s_j = 1$ and ``below'' otherwise:
\[
    (\sigma_1 \encircle{s} \sigma_2)^{(j)} = \begin{cases} 
    (\sigma_1^{(j)} \oplus \sigma_2^{(j)}):= (\sigma_1(1)^{(j)}, \cdots, \sigma_1(n_1)^{(j)}, \sigma_2(1)^{(j)} + n_1, \cdots, \sigma_2(n_2)^{(j)} + n_1), & \text{if }
    s_j = +1, \\ 
    (\sigma_1^{(j)} \ominus \sigma_2^{(j)}):=(\sigma_1(1)^{(j)} + n_2, \cdots, \sigma_1(n_1)^{(j)} + n_2, \sigma_2(1)^{(j)}, \cdots, \sigma_2(n_2)^{(j)}), & \text{if }s_j = -1. \end{cases}
\]
A \vocab{$d$-separable permutation} is a $d$-permutation which can be obtained from trivial size-$1$ permutations via block sums.
\end{definition}

An example of a non-separable $3$-permutation is $\sigma = ((1,3,2), (2,1,3))$; indeed, we cannot decompose this permutation as a block sum of sizes $2$ and $1$ (resp.\ $1$ and $2$) because of the first (resp.\ second) coordinate.

\begin{remark}\label{seppatternavoidance}
Just like for ordinary separable permutations, \cite{asinowski2008separable} provides a description of $d$-separable permutations via (lower-dimensional) pattern avoidance.  \cite[Theorem 4.1]{bonichon2022baxter} restates this result in notation similar to our paper, and they consider pattern occurrence in a slightly more general setting, allowing for patterns of lower dimension to occur. Specifically, a $d$-permutation is separable if and only if the following is true: all of its $3$-dimensional marginals avoid the pattern $((1, 3, 2), (2, 1, 3))$ and all of its symmetries (that is, its images under symmetries of the $3$-dimensional cube), and all of its $2$-dimensional marginals avoid the patterns $(2, 4, 1, 3)$ and $(3, 1, 4, 2)$. However, the characterization in terms of block sums is more useful for us. 
\end{remark}

To state the permuton limit for uniform $d$-separable permutations, it is useful to consider the degenerate case $\mu_{1, q}$ of the skew Brownian permutons $\mu_{\rho, q}$. In \cite{borga2023skew}, it has been proven that $\mu_{1, 1-p}$ agrees with the \vocab{Brownian separable permuton} $\mu^B_p$ introduced in \cite{Bassino_2020}, which is described as follows. Let $e(t)$ be a one-dimensional Brownian excursion on $[0, 1]$, and let $(s(\ell))$ be an iid sequence of signs which are each $+1$ with probability $p$ and $-1$ with probability $1-p$, indexed by the (countable) local minima of $e$. This construction yields a total order $<_{e, p}$ on all but a null set of $[0, 1]$, in which $x <_{e, p} y$ if the minimum on $[x, y]$ is labeled $+1$ and $y <_{e, p} x$ otherwise. This allows us to define the random function
\begin{equation*}
    \psi_{e, p}(t) = \text{Leb}\left(\left\{x \in [0, 1]: x <_{e, p} t\right\}\right)
\end{equation*}
and the corresponding \vocab{Brownian separable permuton} of parameter $p$ as
\[
    \mu^B_p(A) = (\text{Id}, \psi_{e, p})_\ast \text{Leb}(A) = \text{Leb}\left(\left\{t \in [0, 1]: (t, \psi_{e, p}(t)) \in A\right\}\right).
\]
For a more precise construction, see the discussion before \cref{brownian-separable-d-permuton} in \cref{skewbrownianpermutonsection}.

In \cite{Borga_2020}, building on \cite{separablelimit}, $\mu^B_p$ has been shown to be the permuton limit for uniform permutations from many substitution-closed classes (with the value of $p$ depending on properties of the particular class) -- in particular, uniform separable permutations converge to $\mu^B_{1/2}$. 

To state our result, we generalize this construction. For $p_1, \cdots, p_{d-1}\in[0,1]$, the \vocab{Brownian separable $d$-permuton} $$\mu^B_{p_1, \cdots, p_{d-1}}$$ 
is defined analogously to $\mu^B_p$, but each local minimum of a single Brownian excursion is now labeled by a sequence of $(d-1)$ independent signs that are $+1$ with probabilities $p_1, \cdots, p_{d-1}$ respectively, and the permuton now encodes the ``continuous permutation'' on $[0, 1]$ induced by the orders in each of the $(d-1)$ coordinates. A more precise definition can be found in \cref{brownian-separable-d-permuton}. Our final main result shows that $d$-separable permutations also converge to the Brownian separable $d$-permuton for a particularly simple choice of parameters.

\begin{theorem}\label{mainresult4}
Let $\sigma_n$ be a uniform $d$-separable permutation of size $n$. Then $\mu_{\sigma_n}$ converges in distribution to the Brownian separable $d$-permuton $\mu^B_{1/2, \cdots, 1/2}$.
\end{theorem}

In contrast to the Schnyder wood permuton $\mu_S$ -- recall \cref{marginaldetermines} -- we additionally show (\cref{prop:non-triviality}) that the Brownian separable $d$-permuton is not determined by its lower-dimensional marginals, so it is an ``honest'' $d$-dimensional random permuton (with random nontrivially-correlated lower-dimensional marginals). Moreover, as detailed in \cref{separabledpermuton_remark} below, we believe that other natural models of high-dimensional permutations will also converge to the Brownian separable $d$-permuton, possibly with different parameters $(p_1, \cdots, p_{d-1})$.

\subsection{Outline, key ideas for the proofs, and final remarks}\label{sec:conc}

We now briefly describe the key ideas for the proofs of each of our main results. In each case, we generalize the established techniques for $2$-permutations and $2$-permutons; in particular, we expect that the ideas for proving $d$-permuton convergence developed here are applicable in other combinatorial and probabilistic settings where random $d$-permutations may arise.

\medskip

The main goal of \cref{theorysection} is to prove that weak convergence $\mu_n \to \mu$ of $d$-dimensional permutons can be equivalently formulated as convergence of pattern frequencies $\EE[\freq(\tau, \sigma_n)]$, as detailed in \cref{tfaemain}. The essential lemma required for this is that a permuton $\mu$ can be well-approximated by sampling a large number $N$ of iid points from $\mu$, reading off the relative order of those $N$ points as a $d$-dimensional permutation $\sigma_N$ (in the sense of \cref{defnsamplefrompermuton}), and then considering $\mu_{\sigma_N}$ as an approximation of $\mu$. These approximation results are verified first in the deterministic case (\cref{muvsmuk}) and then the random case (\cref{muvsmuk2}); the crux of the argument is then that encoding all pattern frequencies of a fixed size into a permuton yields a good approximation for the candidate limiting permuton itself.

\medskip 

The purpose of \cref{backgroundsection} is to give a more thorough introduction to the skew Brownian permutons and the Brownian separable permutons, providing rigorous definitions. We will also introduce the notion of \vocab{coalescent-walk processes}, originally introduced in \cite{borga2022scaling}, which will play a fundamental role in \cref{schnydersection}.

\medskip
    
Next, \cref{schnydersection} is dedicated to describing the permuton limit for Schnyder wood permutations. This limit is found by constructing a pair of coalescent-walk processes which recover the $3$-permutation that we aim to study. In short, a coalescent-walk process is a collection of non-crossing random walks that ``stick together'' once they intersect; illustrations of coalescent-walk processes can be seen in \cref{fig:samplecoalescent} and \cref{fig:sampleschnyder2}, and a formal definition is provided in \cref{coalescentintrosection}. We prove that Schnyder wood permutations bijectively encode Schnyder wood triangulations (\cref{schnyderwoodbijectperm}), which are themselves in bijection with certain random walk excursions with specified increments (\cref{randomwalkisuniformschnyder1}). While a uniform rescaled such walk converges to a Brownian excursion, recovering the permutation in the limit requires us to extract more information than just the Brownian excursion itself, and thus we must study the scaling limits for the full coalescent-walk processes instead. This is the subject matter of \cref{walklimits}.

There are two primary differences between our work and previous constructions of this type, as in \cite{borga2021permuton,borga2022scaling}. First of all, working with a $3$-permutation instead of a $2$-permutation means that we must construct a coalescent-walk process for each of $\sigma_M^g$ and $\sigma_M^r$. The surprise here (as detailed in \cref{mainresult1}) is that these coalescent-walk processes are in fact driven by the same walk (except forward in one case and backward in the other), making the bijections much simpler to prove and the dependence between the marginals very strong. Second, our coalescent-walk processes only recover the 3-permutation of interest if we specify specific starting points from which to grow our walk in the coalescent-walk process (as in \cref{coalescentdefn}), rather than beginning a walk at every integer point as in previous cases. Thus, we must take more care to study the distribution of the starting points and show that they are asymptotically evenly spaced out to leading order (as done in \cref{schnyderlocationtail}).

\begin{remark}\label{semibaxterrmk}
It can be proved (though we omit the proof in this paper for brevity) that for any Schnyder wood permutation $\sigma_M$, $\sigma_M^r$ is always a \vocab{semi-Baxter} permutation, meaning it avoids the vincular pattern $2-41-3$, and $\sigma_M^g$ is always the inverse of such a permutation. However, because these marginals are not uniform over all possibilities, the parameters $(\rho, q) = (-\frac{\sqrt{2}}{2}, \frac{1}{1+\sqrt{2}})$ that specify the Schnyder wood permuton limit do not agree with the parameters for the permuton limit of uniform semi-Baxter permutations (proved in \cite{borga2021permuton} to be $(\rho, q) = (-\frac{1 + \sqrt{5}}{4}, \frac{1}{2})$). It would be interesting to understand if other natural weighted models of pattern-avoiding permutations exhibit similar behavior.
\end{remark}

\begin{remark}\label{inversionrmk}
While explicit formulas for the distribution of pattern frequencies in $\mu_{\rho, q}$ are not generally known, \cite[Proposition 1.14]{borga2023baxter} shows that the expected frequency of inversions is a function of the angle $\theta$ between the space-filling SLEs in the alternate construction of the skew Brownian permuton (see \cref{lqgdescription2}). While the general relation between $\theta$ and $\rho, q$ is nonexplicit, it is known in our particular case thanks to \cref{rem:angle}. In particular, $\theta\left(\rho=-\frac{\sqrt 2}{2},q=\frac{1}{1+\sqrt{2}}\right)=\frac{2\pi}{3}$. Plugging this angle into \cite[Proposition 1.14]{borga2023baxter}, we have that 
\[
    \EE\left[\freq((2, 1), \mu_{\rho, q}^g\right] = \EE\left[\freq((2, 1), \mu_{\rho, q}^r\right] = \frac{1}{3},
\]
and therefore after applying the reflections in \cref{eq:weifvwebfowep} we have
\[
    \EE\left[\freq((2, 1), \mu^g_S\right] = \EE\left[\freq((2, 1), \mu^r_S\right] = \frac{2}{3}.
\]
Simulations for this result can be found at \href{http://www.jacopoborga.com/2024/12/27/high-dimensional-permutons-the-schnyder-wood-and-brownian-separable-d-permuton/}{this webpage}. It would be interesting to determine any further information about the distribution of pattern frequencies, such as the expectation of any $3$-pattern frequency in the full permuton $\mu_S$.
\end{remark}

\medskip

Finally, \cref{dseparablesection} identifies the permuton limit for $d$-separable permutations for all $d \ge 2$. Here, we make use of bijections (\cref{dsep-bijection} and \cref{dsep-bijection2}) between $d$-separable permutations and trees with labeled internal vertices, in which the shape and labeling of the tree encode the order and signs with which block sums are performed to construct the permutation. This generalizes the bijection for (ordinary) separable permutations. We then argue that such trees can be sampled by instead sampling conditioned Galton-Watson trees (\cref{separablegaltonwatson}), for which limit theorems have been established and are sufficient for characterizing pattern frequency.

\begin{remark}\label{separabledpermuton_remark}
We conjecture that the Brownian separable $d$-permuton, much like the Brownian separable permuton, is a universal limit that describes other natural classes of $d$-dimensional permutations. In particular, the strategy of representing such permutations via ``packed trees'' (whose vertices are now more complicated gadgets encoding permutation patterns) in \cite{Borga_2020} may generalize to higher dimensions, and the parameters $p_i$ may be computable in terms of properties of the $d$-dimensional permutation class.
\end{remark}

We close this introduction by noting that other generalizations of permutations to higher dimensions have been considered in other works. For example, bijections on the vertices of higher-dimensional graphs such as the torus $\mathbb{Z}^2/(n\mathbb{Z} \times m\mathbb{Z})$ \cite{hammondhelmuth} and integer lattice $\mathbb{Z}^d$ \cite{elboimsly} have been studied, primarily analyzing cycle structure for permutations biased towards the identity map. Drawing more parallels to our work, \cite{latinons} studies scaling limits of Latin squares, which are higher-dimensional permutations in the sense that fixing a particular value yields an ordinary permutation. While this latter work also establishes convergence results based on pattern density and approximation by finite objects, the objects we study in this paper generalize permutations more naturally to any dimension $d$. In particular, the explicit examples of scaling limits we construct have direct connections to existing universality classes of ordinary permutons.

\paragraph{Acknowledgments.} J.B.\ was partially supported by the NSF under Grant No.\ DMS-2441646, and A.L.\ was partially supported by the NSF under Grant No.\ DGE-2146755.

\section{High-dimensional permuton theory}\label{theorysection}

In this section, we prove \cref{tfaemain} and establish other $d$-dimensional generalizations of results in (ordinary) permuton theory. The case $d = 2$ was first considered for deterministic permutations in \cite{hoppen2012limits} and then random ones in \cite{Bassino_2020}, and we will take an analogous path for general $d$.

Throughout, we work with various probability measures; to make the source of randomness clear, we may write $\PP^\mu$ or $\EE^\mu$ when computing probabilities or expectations with respect to a permuton $\mu$.

\subsection{The theory for the deterministic setting}

Recall that a $d$-dimensional permuton is a Borel probability measure on $[0, 1]^d$ with uniform $1$-dimensional marginals. First, we consider the case where $\sigma_n$ is some deterministic $d$-dimensional permutation for each $n$, meaning that $\mu_{\sigma_n}$ is some deterministic measure on $[0, 1]^d$. We wish to describe the permuton limit of these permutations $\sigma_n$. 

To specify a probability measure $\mu$ on $[0, 1]^d$, we can equivalently specify the cumulative distribution function
\[
    F(\vec{x}) = \mu\left([0, x_1] \times \cdots \times [0, x_d]\right)
\]
for all $\vec{x} \in [0, 1]^d$. Then requiring uniform $1$-dimensional marginals corresponds to the requirement that $F(1, \cdots, 1, x, 1, \cdots, 1) = x$ for all $x \in [0, 1]$ in each coordinate, and the probability of lying within a box $\prod_{i=1}^d [x_i, y_i]$ is then a sum and difference of $2^d$ values of $F(\vec{x})$ via inclusion-exclusion. Thus if we have two measures $\mu_1, \mu_2$ on $[0, 1]^d$ with distribution functions $F_1$ and $F_2$, we have
\begin{equation}\label{twod}
\mu_1\left(\prod_{i=1}^d [x_i, y_i]\right) - \mu_2 \left(\prod_{i=1}^d [x_i, y_i]\right) \le 2^d ||F_1 - F_2||_\infty,
\end{equation}
and we also have
\begin{equation}\label{triineq}
    |F(\vec{y}) - F(\vec{x})| \le \sum_{i=1}^d \mu\left([0, 1] \times \cdots \times [0, 1] \times [x_i, y_i] \times \cdots [0, 1] \times \cdots \times [0, 1]\right) = \sum_{i=1}^d |y_i - x_i|.
\end{equation}

\begin{proposition}\label{weakinf}
Let $\mu_n, \mu$ be $d$-dimensional permutons with distribution functions $F_n, F$ respectively. Then $\mu_n$ converges weakly to $\mu$ if and only if $||F_n - F||_{\infty} \to 0$ (that is, $F_n(\vec{x}) - F(\vec{x}) \to 0$ uniformly on $[0, 1]^d$).
\end{proposition}

\begin{proof}
For one direction, suppose $||F_n - F||_{\infty} \to 0$. Fix $\eps > 0$ and a bounded continuous function $g: [0, 1]^d \to \RR$. By uniform continuity, there is some positive integer $k$ such that breaking up $[0, 1]^d$ into $k^d$ subcubes of equal size (with all vertices at $\frac{1}{k}$-lattice points), we have $|g(\vec{y}) - g(\vec{x})| < \frac{\eps}{2}$ whenever $\vec{x}, \vec{y}$ are within the same cube $C$. Then we have

\begin{align*}
    \int_{[0, 1]^d} g \, d\mu_n - \int_{[0, 1]^d} g \, d\mu &\le \sum_{\text{cubes }C}\int_C g \, d\mu_n - \int_C g \, d\mu \\
    &\le \sum_{\text{cubes }C}\left(\mu_n(C) \max_C g  - \mu(C) \min_C g \right) \\
    &= \sum_{\text{cubes }C}\left(\mu_n(C) \max_C g  - \mu(C) \max_C g + \mu(C) \max_C g - \mu(C) \min_C g \right) \\
    &\le \sum_{\text{cubes }C}\left((\mu_n(C) - \mu(C))||g||_{\infty} + \mu(C) \frac{\eps}{2} \right) \\
    &\le 2^d k^d ||F_n - F||_\infty ||g||_\infty + \frac{\eps}{2},
\end{align*}
where in the last step we use \cref{twod} and that there are $k^d$ cubes. Since $||F_n - F||_\infty$ tends to zero by assumption, this means that for $n$ large enough this whole quantity is less than $\eps$. The same argument works with the roles of $\mu$ and $\mu_n$ swapped, so we have shown weak convergence. 

\medskip

For the other direction, again fix $\eps > 0$. Since $\mu$ has uniform marginals, the box $\prod_i [0, x_i]$ is a continuity set of $\mu$ for any $\vec{x} \in [0, 1]^d$ (meaning that its boundary has $\mu$-measure zero). Thus for each $\vec{x} \in [0, 1]^d$, $F_n(\vec{x})$ converges to $F(\vec{x})$ as $n \to \infty$ (by the portmanteau theorem). Let $k = \lceil \frac{4d}{\eps} \rceil$, and choose $n$ large enough so that $|F_n(\vec{y}) - F(\vec{y})| < \frac{\eps}{2}$ for all (finitely many) $\frac{1}{k}$-lattice points $\vec{y}$. Then for any $\vec{x}$, let $\vec{y}$ be the $\frac{1}{k}$-lattice point $(\frac{1}{k} \lfloor kx_1 \rfloor, \cdots, \frac{1}{k} \lfloor kx_d \rfloor)$. By \cref{triineq}, we have that
\begin{align*}
    |F_n(\vec{x}) - F(\vec{x})| &\le |F_n(\vec{x}) - F_n(\vec{y})| + |F_n(\vec{y}) - F(\vec{y})| + |F(\vec{y}) - F(\vec{x})|\\
    &\le \sum_{i=1}^d |y_i - x_i| + \frac{\eps}{2} + \sum_{i=1}^d |y_i - x_i|,
\end{align*}
and this last quantity is at most $\frac{2d}{k} + \frac{\eps}{2} < \eps$ by assumption. Thus we have uniform convergence of $F_n$ to $F$ as desired.
\end{proof}

We may also endow the space of such permutons with a metric (primarily to be consistent with previous literature on permutons):

\begin{definition}
Let $\mu_1, \mu_2$ be two $d$-dimensional permutons. The \vocab{box distance} $d_{\square}(\mu_1, \mu_2)$ is given by
\[
    d_{\square}(\mu_1, \mu_2) = \sup_{\substack{\vec{x}, \vec{y} \in [0, 1]^d \\ x_i < y_i \,\, \forall i}} \left| \mu_1\left(\prod_{i=1}^{d} [x_i, y_i]\right) - \mu_2\left(\prod_{i=1}^{d} [x_i, y_i]\right) \right|.
\]
\end{definition}
Letting $F_1, F_2$ be the distribution functions of $\mu_1, \mu_2$, we have by \cref{twod} that
\[
    ||F_1 - F_2||_{\infty} \le d_{\square}(\mu_1, \mu_2) \le 2^d ||F_1 - F_2||_{\infty}.
\]
In particular, $d_{\square}(\mu_n, \mu) \to 0$ if and only if $F_n \to F$ uniformly; in such a situation we write that $\mu_n \to \mu$.

Recall the definition of pattern frequency for a permuton from \cref{defnsamplefrompermuton}. In the deterministic case, $\text{freq}(\tau, \mu)$ is just a constant. We have the following useful estimate relating pattern frequencies in permutations and their corresponding permutons (which explains why we use the same notation for both):

\begin{lemma}\label{permapprox}
Let $\tau \in \mc{S}_{d, k}$ and $\sigma \in \mc{S}_{d, n}$ for $1 \le k \le n$. Then 
\[
    \left|\freq(\tau, \sigma) - \freq(\tau, \mu_\sigma)\right| \le \frac{1}{n} \binom{k}{2}.
\]
\end{lemma}

\begin{proof}
View $\freq(\tau, \sigma)$ as the probability that a uniform random subset of the $k$ points of $\sigma$ are in the right relative order $\tau$, and view $\freq(\tau, \mu_\sigma)$ as the probability that $k$ random points sampled from $\mu_\sigma$ are in the right relative order. Recall that $\mu_\sigma$ assigns a mass of $\frac{1}{n}$ to each of $n$ boxes of side length $\frac{1}{n}$. By a union bound, the probability that these $k$ points are not selected from distinct boxes is at most $\frac{1}{n} \binom{k}{2}$, since any two points have a probability $\frac{1}{n^2}$ of both being sampled from any given box (and there are $n$ possible boxes and $\binom{k}{2}$ pairs of points). But conditioned on the event that all $k$ points are chosen from distinct boxes, we can sample from $\mu_\sigma$ by choosing a uniform subset of $k$ of the $n$ boxes and then uniformly choosing a point inside; this couples the sampling of $\mu_\sigma$ to that of $\sigma$. Thus the frequencies can only differ by at most $\frac{1}{n} \binom{k}{2}$, as desired.
\end{proof}

The main result of this subsection is the following characterization of convergence:

\begin{proposition}\label{permconti}
Let $\mu_n$ be a sequence of $d$-dimensional permutons. Then $\mu_n \to \mu$ if and only if for every $k$ and every $\tau \in \mc{S}_{d, k}$, we have $\freq(\tau, \mu_n) \to \freq(\tau, \mu)$. 
\end{proposition}

In particular, if $\sigma_n$ is a sequence of $d$-dimensional permutations with $|\sigma_n| \to \infty$ and $\mu_n = \mu_{\sigma_n}$, then $\freq(\tau, \sigma_n) \to \freq(\tau, \mu_n)$ by \cref{permapprox} and thus this statement also applies to permutation frequency when we consider a sequence of permutations $\sigma_n$.

\begin{remark}\label{eventuallyconstant}
If $|\sigma_n|$ does not go to infinity but $\freq$ converges, then the sequence must be constant; indeed, let $a = \liminf_{n \to \infty} |\sigma_n|$. Then summing $\freq(\tau, \sigma_n)$ over all $\tau \in \mc{S}_{d, a + 1}$ yields $0$ if $|\sigma_n| = a$ and $1$ otherwise, so convergence of $\freq$ can only occur if $|\sigma_n| = a$ eventually (that is, our permutations are eventually of some given size). Then since there are finitely many $d$-dimensional permutations of size $a$, convergence of frequency only holds if $\sigma_n$ is eventually constant.  
\end{remark}

We will prove \cref{permconti} by first showing that we can approximate a permuton $\mu$ in box distance by (the permuton associated to) the random permutation $P_{\mu}[k]$ introduced in \cref{defnsamplefrompermuton}.

\begin{proposition}\label{muvsmuk}
Let $\mu$ be a $d$-dimensional permuton. For all sufficiently large $k$, we have
\[
    \PP\left(d_\square(\mu, \mu_{P_\mu[k]}) \le d2^{d+2} k^{-1/4}\right) \ge 1 - \exp(-\sqrt{k}).
\]
\end{proposition}

In particular, this means that $\mu_{P_\mu[k]}$ converges to $\mu$ in probability as $k \to \infty$.

\begin{proof}
Let $F$ be the distribution function for $\mu$, and for each $k$, let $F_k$ be the distribution function for $\mu_{P_\mu[k]}$. We will estimate $||F - F_k||_{\infty}$ and use that $d_{\square}(\mu_1, \mu_2) \le 2^d ||F_1 - F_2||_{\infty}$ to conclude. 

First, we only consider the $\frac{1}{k}$-lattice points $\vec{a} = \frac{1}{k}(a_1, \cdots, a_d)$ with all $a_i \in [k]$; let this set of points be denoted $L$. Recall that $\mu_{P_\mu[k]}$ assigns a mass of $\frac{1}{k}$ to $k$ different (random) $\frac{1}{k}$-lattice boxes in $[0, 1]^d$. Thus, $kF_k(\vec{a})$ counts how many of the $k$ points sampled for $P_\mu[k]$ satisfy the following condition: for all $i$, the $i$--th coordinate of the point is at most the $a_i$--th largest out of all $k$ points sampled.

We claim that we have the bounds
\begin{equation}\label{hoeffdingbound}
    \PP\left(F_k(\vec{a}) - F(\vec{a}) > 2dk^{-1/4}\right) < (d+1) \exp(-2\sqrt{k}), \quad \PP\left(F_k(\vec{a}) - F(\vec{a}) < 2dk^{-1/4}\right) < (d+1) \exp(-2\sqrt{k}).
\end{equation}
We'll just prove the first bound; the second one is shown in basically the same way. The event $\{F_k(\vec{a}) - F(\vec{a}) > 2dk^{-1/4}\}$ only occurs if one of the following cases occurs:
\begin{enumerate}
    \item For some $1 \le i \le d$, the $a_i$--th largest coordinate in the $i$--th direction is bigger than $\frac{a_i}{k} + k^{-1/4}$. 
    \item The $a_i$--th largest coordinate in the $i$--th direction is at most $\frac{a_i}{k} + k^{-1/4}$ for all $k$ (call this event $E$), but we still have $F_k(\vec{a}) - F(\vec{a}) > 2dk^{-1/4}$.
\end{enumerate}
For case (1), since $\mu$ is a permuton, each of the $k$ points has a uniform $i$--th coordinate. So the number of points $N_1$ whose $i$--th coordinate is less than $\frac{a_i}{k} + k^{-1/4}$ is binomial with parameters $(k, \frac{a_i}{k} + k^{-1/4})$. By Hoeffding's theorem, the probability that a $\text{Bin}(k, p)$ random variable is $k\eps$ smaller (also larger) than its mean is at most $\exp(-2k\eps^2)$. Thus the event in case (1) occurs with probability 
\[
    \PP(N_1 < a_i) = \PP(N_1 - (a_i + k^{3/4}) < -k \cdot k^{-1/4}) \le \exp(-2\sqrt{k})
\]
for each $i$ and thus an overall probability of less than $d\exp(-2\sqrt{k})$. 

Next, we can bound the probability for case (2) from above by 
\[
    \PP\left(E \cap \left\{F_k(\vec{a}) - F\left(\vec{a} + k^{-1/4}\vec{1}\right) > dk^{-1/4}\right\}\right),
\]
since by \cref{triineq} we know that $F(\vec{a} + k^{-1/4}\vec{1}) \le F(\vec{a}) + dk^{-1/4}$. But if $E$ occurs, then $kF_k(\vec{a})$ is at most the number of sampled points for $P_\mu[k]$ whose $i$--th coordinate is at most $\frac{a_i}{k} + k^{-1/4}$ for all $i$, which is a binomial random variable $N_2$ with parameters $\left(k, F\left(\vec{a} + k^{-1/4}\vec{1}\right)\right)$. Thus again by Hoeffding's theorem we can bound this from above by 
\[
    \PP\left(\frac{N_2}{k} - F\left(\vec{a} + k^{-1/4}\vec{1}\right) > dk^{-1/4}\right) = \PP(N_2 - kF\left(\vec{a} + k^{-1/4}\vec{1}\right) > k \cdot dk^{-1/4}) \le \exp(-2d^2\sqrt{k}) < \exp(-2\sqrt{k}).
\]
A union bound over these cases yields the result of \cref{hoeffdingbound}. 

\medskip

Now applying \cref{hoeffdingbound} to all such points $\vec{a} \in L$ (of which there are $k^d$ in total), and doing a union bound, we see that
\begin{equation}\label{approxlatticebound}
    \PP\left(\sup_{\vec{a} \in L} \left|F_k(\vec{a}) - F(\vec{a})\right|  > 2dk^{-1/4} \right) \le 2(d+1)k^d\exp(-2\sqrt{k}).
\end{equation}
To finish the proof of the proposition, now take some arbitrary $\vec{x} \in [0, 1]^d$ and let $\vec{y} = (\frac{1}{k} \lceil kx_1 \rceil, \cdots, \frac{1}{k} \lceil kx_d \rceil) \in L$ be its lattice point approximation. Then we have by \cref{triineq} that
\begin{align*}
    \left|F_k(\vec{x}) - F(\vec{x})\right| &\le \left|F_k(\vec{x}) - F_k(\vec{y})\right| + \left|F_k(\vec{y}) - F(\vec{y})\right| + \left|F(\vec{y}) - F(\vec{x})\right| \\
    &\le \frac{2d}{k} + \left|F_k(\vec{y}) - F(\vec{y})\right| \\
    &\le 2dk^{-1/4} + \left|F_k(\vec{y}) - F(\vec{y})\right|,
\end{align*}
so combining this with \cref{approxlatticebound} shows that
\[
    \PP\left(\sup_{\vec{x} \in [0, 1]^d} \left|F_k(\vec{x}) - F(\vec{x})\right| > 4dk^{-1/4} \right) \le 2(d+1)k^d\exp(-2\sqrt{k}).
\]
Now choose $m$ so that $2(d+1)k^d \le \exp(\sqrt{k})$ for all $k \ge m$ to find that 
\[
    \PP(||F_k - F||_\infty > 4dk^{-1/4}) \le \exp(-\sqrt{k})
\]
for all sufficiently large $k$. Finally, the fact that $d_\square(\mu, \mu_{P_\mu[k]}) \le 2^d ||F_k - F||_\infty$ yields the result.
\end{proof}

We will now show \cref{permconti} by making use of this approximation of $\mu$ by $P_\mu[k]$.

\begin{proof}[Proof of \cref{permconti}]
Suppose $\mu_n \to \mu$, and fix some integer $k$ and $\tau \in \mc{S}_{d, k}$. Let $\vec{x}_{1, n}, \cdots, \vec{x}_{k, n}$ be iid elements of $[0, 1]^d$ from $\mu_n$; then $(\vec{x}_{1, n}, \cdots, \vec{x}_{k, n})$ converges weakly to $(\vec{x}_1, \cdots, \vec{x}_k)$ iid sampled from $\mu$. Let $E_\tau$ be the set of points $(\vec{y}_1, \cdots, \vec{y}_k)$ such that the relative order of the coordinates match up with $\tau$ (that is, $\text{Perm}(\vec{y}_1, \cdots, \vec{y}_k) = \tau$). This is a subset of $[0, 1]^{kd}$ independent of $n$ with boundary of measure zero (since all coordinates have uniform marginals by definition, and lying on the boundary of $E_\tau$ requires some two of the $kd$ coordinates to be equal, which is a probability-zero event). Thus 
\[
    \PP\left(\left(\vec{x}_{1, n}, \cdots, \vec{x}_{k, n}\right) \in E_\tau\right) \to \PP\left(\left(\vec{x}_1, \cdots, \vec{x}_k\right) \in E_\tau\right),
\]
and these probabilities on the left and right side are exactly the definitions of $\freq(\tau, \mu_n)$ and $\freq(\tau, \mu)$ respectively. Thus we do have convergence of frequency, proving the forward direction.

On the other hand, suppose $\freq(\tau, \mu_n) \to \freq(\tau, \mu)$ for all $\tau$. Let $f: [0, 1]^d \to \RR$ be a bounded continuous function, and suppose for the sake of contradiction that $\EE^{\mu_n}[f]$ does not converge to $\EE^{\mu}[f]$. The sequence $\EE^{\mu_n}[f]$ is bounded (since $f$ is bounded), so there must be some subsequence $\EE^{\mu_{n(k)}}[f]$ converging to $x \ne \EE^{\mu}[f]$. Since $[0, 1]^d$ is compact, Prohorov's theorem (\cite[Theorem 5.1]{billing}) implies that the space of $d$-dimensional permutons is compact. Thus some further subsequence of the $\mu_{n(k)}$s, say $\mu_{n(k(\ell))}$, must converge weakly to some $\mu'$. Since all $\mu_{n(k(\ell))}$ have uniform marginals, so does $\mu'$. By the proof of the forward direction, we have $\freq(\tau, \mu_{n(k(\ell))}) \to \freq(\tau, \mu')$ for all $\tau$. But we also know that $\freq(\tau, \mu_{n(k(\ell))}) \to \freq(\tau, \mu)$ by assumption. This means that for every $k$, $\mu_{P_\mu[k]}$ and $\mu_{P_{\mu'}[k]}$ agree, since those permutons are fully determined by the frequencies of all $d$-dimensional permutations of size $k$. By \cref{muvsmuk} and the triangle inequality, that means $d_{\square}(\mu, \mu') = 0$; this can only happen if $\mu = \mu'$. Thus $\EE^{\mu_{n(k(\ell))}}[f] \to \EE^{\mu'}[f] = \EE^{\mu}[f]$, which is a contradiction with $\EE^{\mu_{n(k)}}[f] \to x$. Therefore we do have convergence of expectation and thus weak convergence. 
\end{proof}

Next, we show that all permutons can indeed be approximated by an appropriate sequence of permutations:

\begin{proposition}\label{findpermutationsequence}
Let $\mu$ be a $d$-dimensional permuton. Then there is a sequence of $d$-dimensional permutations $\sigma_n$ such that $\mu_{\sigma_n} \to \mu$.
\end{proposition}
\begin{proof}
Let $(\vec{x}_1, \vec{x}_2, \cdots)$ be an infinite sequence of iid elements of $[0, 1]^d$ sampled from $\mu$, and let $\sigma_n = \text{Perm}(\vec{x}_1, \vec{x}_2, \cdots, \vec{x}_n)$ (which is a $d$-dimensional permutation almost surely). By \cref{muvsmuk}, for any $\eps > 0$, we must have $d_{\square}(\mu, \mu_{\sigma_n}) \le \eps$ eventually almost surely, which is equivalent to $\mu_{\sigma_n} \to \mu$ weakly, as desired.
\end{proof}

Finally, we check completeness of the space of $d$-dimensional permutons:

\begin{proposition}
A sequence of $d$-dimensional permutons $\mu_n$ converges if and only if it is Cauchy with respect to $d_\square$.
\end{proposition}
\begin{proof}
If $\mu_n$ converges to some $\mu$, then $d_{\square}(\mu_n, \mu) \to 0$, so the sequence is Cauchy. For the other direction, let $F_n$ be the distribution function for $\mu_n$ for each $n$. By \cref{twod}, the distribution functions are Cauchy with respect to the sup norm, so $F_n$ converges to some function $F: [0, 1]^d \to \RR$. To show that $F$ is the distribution function of a $d$-dimensional permuton, take a subsequence $\mu_{n(k)}$ that converges weakly to some $\mu'$; as in the proof of \cref{permconti}, $\mu'$ again has uniform $1$-dimensional marginals and thus is a $d$-permuton with some distribution function $F'$. By \cref{weakinf} we know that $||F_{n(k)} - F'||_{\infty} \to 0$, so in fact $F = F'$ and we do have convergence to an actual $d$-dimensional permuton, as desired. 
\end{proof}

\subsection{The theory for the random setting}

We now discuss sequences of random permutations and their limiting random permutons $\mu$. We can think of points sampled from $\mu$ as first sampling $\mu$ from some distribution and then sampling iid points $\vec{x}_1, \cdots, \vec{x}_k$ conditionally on $\mu$. The bound from \cref{muvsmuk} still holds even if $\mu$ is random:

\begin{proposition}\label{muvsmuk2}
Let $\mu$ be a $d$-dimensional random permuton. For all sufficiently large $k$, we have 
\[
    \PP^{\mu}\left(d_\square(\mu, \mu_{P_\mu[k]}) \le d2^{d+2} k^{-1/4}\right) \ge 1 - \exp(-\sqrt{k}).
\]
\end{proposition}
\begin{proof}
Consider the formula for $\freq(\tau, \mu)$ in \cref{defnpattern}, and note that this quantity is now a random variable since $\mu$ is random. We may enlarge the probability space under consideration to $(\mu, \vec{x}_1, \cdots, \vec{x}_k)$, where we first sample $\mu$ and then (conditionally) sample the $\vec{x}_i$s iid from $\mu$. This distribution is characterized by specifying that for any measurable functional $H$ on the space, we have
\[
    \EE^{\mu, \vec{x}_1, \cdots, \vec{x}_k}\left[H(\mu, \vec{x}_1, \cdots, \vec{x}_k)\right] = \EE^{\mu}\left[ \int_{([0, 1]^d)^k} H(\mu, \vec{x}_1, \cdots, \vec{x}_k) \mu(d\vec{x}_1) \cdots \mu(d\vec{x}_k)\right].
\]
We plug in the functional $H(\nu, \vec{x}_1, \cdots, \vec{x}_k) = \mathds{1}\left\{d_\square(\nu, \mu_{P_\nu[k]}) \le d2^{d+2} k^{-1/4}\right\}$ into the equation above (switching notation to avoid overloading the use of $\mu$). Then \cref{muvsmuk} shows that the quantity inside the expectation on the right is bounded from below by $1 - \exp(-\sqrt{k})$, and the left-hand side is exactly the probability in the proposition statement. Thus the result follows.
\end{proof}

With this bound, we can now state and prove the main characterization of $d$-dimensional permuton convergence, showing that weak convergence is equivalent to convergence of all pattern frequencies:

\begin{proof}[Proof of \cref{tfaemain}]
First assuming (1), we have that the vector
\[
    (\freq(\tau_1, \mu_{\sigma_n}), \cdots, \freq(\tau_m, \mu_{\sigma_n})) \rightarrow (\freq(\tau_1, \mu), \cdots, \freq(\tau_m, \mu))
\] 
converges in distribution for any finitely many $d$-dimensional permutations $\tau_1, \cdots, \tau_m$; indeed, the mapping $x \mapsto (\freq(\tau_1, x), \cdots, \freq(\tau_m, x))$ is continuous by \cref{permconti}. Since the random variables $\freq(\tau_i, \mu_{\sigma_n})$ and $\freq(\tau_i, \sigma_n)$ differ by $O(1/n)$ by \cref{permapprox}, $(\freq(\tau_1, \sigma_n), \cdots, \freq(\tau_m, \sigma_n))$ converges to $(\freq(\tau_1, \mu), \cdots, \freq(\tau_m, \mu))$ in distribution as well, which is exactly what (2) states. 

(2) implies (3) is clear -- convergence in distribution implies convergence in expectation since $\freq$ is bounded by $1$. 

(3) also immediately implies (4), because the frequency of any pattern $\tau$ in $\sigma_n$ is by definition exactly the probability that $\tau$ is the pattern of $\sigma_n$ on a uniform subset $I_{n,k}$. (And the constants $c_\tau$ will indeed sum to $1$, since $\sum_{\tau \in \mc{S}_{d, k}} \EE[\freq(\tau, \sigma_n)] = 1$ for any $n \ge k$ and there are only finitely many $d$-dimensional permutations of size $k$ in the sum.)

The hard implication is the last one, that is, (4) implies (1). First notice that for fixed $k$, we have 
\[
    \PP(P_{\mu_{\sigma_n}}[k] = \tau) = \EE^{\sigma_n}[\freq(\tau, \mu_{\sigma_n})],
\] 
and again by \cref{permapprox} the right-hand side has the same limit as $\EE^{\sigma_n}[\freq(\tau, \sigma_n)]$ as $n \to \infty$. But this latter expectation is exactly the probability that $\pat_{I_{n,k}}(\sigma_n) = \tau$. Since we assume this probability converges, $P_{\mu_{\sigma_n}}[k]$ converges to $\rho_k$ in distribution and thus $\mu_{P_{\mu_{\sigma_n}}[k]}$ converges to $\mu_{\rho_k}$ in distribution as well (since all objects here are made up of $\frac{1}{k}$-side length squares, meaning there are only finitely many possible values they can take on).

From here, we can prove weak convergence by letting $f$ be an arbitrary bounded continuous function from the space of $d$-dimensional permutons to $\RR$. First notice that
\begin{align*}
    \left|\EE[f(\mu_{\sigma_n})] - \EE[f(\mu_{P_{\mu_{\sigma_n}}[k]})]\right|
    \le &\,\EE\left[\left|f(\mu_{\sigma_n}) - f(\mu_{P_{\mu_{\sigma_n}}[k]})\right|; d_\square(\mu_{\sigma_n}, \mu_{P_{\mu_{\sigma_n}}[k]}) \le d2^{d+2}k^{-1/4}\right] \\ 
    &+  \EE\left[\left|f(\mu_{\sigma_n}) - f(\mu_{P_{\mu_{\sigma_n}}[k]})\right|; d_\square(\mu_{\sigma_n}, \mu_{P_{\mu_{\sigma_n}}[k]}) > d2^{d+2}k^{-1/4}\right].
\end{align*} 
(Here the expectations on the right-hand side are taken with respect to the ``enlarged probability spaces'' as in the proof of \cref{muvsmuk2}.) By \cref{muvsmuk2}, the second term is at most $2\sup|f| \exp(-\sqrt{k})$, and the first term is at most $\omega(d2^{d+2}k^{-1/4})$, where $\omega$ is the modulus of continuity for $f$. Thus we have
\[
    \left|\EE[f(\mu_{\sigma_n})] - \EE[f(\mu_{P_{\mu_{\sigma_n}}[k]})]\right| \le \omega(d2^{d+2}k^{-1/4}) + 2\sup|f| \exp(-\sqrt{k}).
\]
Recalling that the space of $d$-dimensional permutons is compact, consider any convergent subsequence of $\mu_{\sigma_n}$ converging to some $\mu'$. Since $\mu_{P_{\mu_{\sigma_n}}[k]}$ converges to $\mu_{\rho_k}$ as $n \to \infty$, we also have
\[
    \left|\EE[f(\mu')] - \EE[f(\mu_{\rho_k})]\right| \le \omega(d2^{d+2}k^{-1/4}) + 2\sup|f| \exp(-\sqrt{k}).
\]
Taking $k \to \infty$, both terms on the right-hand side go to zero because $f$ is a bounded continuous function on a compact space. Thus the whole right-hand side goes to zero, implying that $\EE[f(\mu_{\rho_k})]$ converges to $\EE[f(\mu')]$. Since our choice of subsequence does not depend on $f$, this means $\mu_{\rho_k}$ converges weakly to $\mu'$. But then this means that any convergent subsequence of $\mu_{\sigma_n}$ converges to $\mu'$, so the whole sequence must converge (because any subsequence has a further convergent subsequence), proving that $\mu_{\sigma_n}$ has a limit as desired. 

Finally, assuming these conditions hold, the equality of various quantities follows from the proofs above. Indeed, the proof that (1) implies (2) shows that $(v_\tau)_\tau$ is distributed as $(\freq(\tau, \mu))_\tau$ and hence that $\EE[v_\tau] = \EE[\freq(\tau, \mu)]$ for any permutation $\tau$ of an arbitrary size $k$, and convergence in expectation (from (2) implies (3)) means this quantity is also $c_\tau$. And this quantity is also the probability mass $\PP(\rho_k = \tau)$, since (as described in the proof of (3) implies (4)) this is exactly the limit of the expected pattern frequencies $\EE[\freq(\tau, \sigma_n)]$. To conclude, we indeed have $P_\mu[k] \stackrel{d}{=} \rho_k$ -- that is, $\PP(P_{\mu}[k] = \tau) = \PP(\rho_k = \tau)$ for all $\tau$ of size $k$ -- because the former quantity is exactly $\EE[\text{freq}(\tau, \mu)]$, which we've just shown is equal to $\PP(\rho_k = \tau)$.
\end{proof}

We conclude this section with results that show that the space of permutons is the ``correct'' candidate space of limits:

\begin{proposition}\label{randunique}
If two random permutons $\mu, \mu'$ satisfy 
\[
    \PP^{\mu,\vec{x}_1, \cdots, \vec{x}_k}\left(P_\mu[k] = \tau\right) = \PP^{\mu',\vec{x}_1', \cdots, \vec{x}_k'}\left(P_{\mu'}[k] = \tau\right)
\]
for all sufficiently large $k$ and all $\tau \in \mc{S}_{d, k}$, then $\mu = \mu'$ in distribution.
\end{proposition}
\begin{proof}
Let $\phi$ be any bounded continuous function from the space of $d$-dimensional permutons to $\RR$. Then for all $k$ we have (letting $\vec{\mathbf{x}}$ denote $(\vec{x}_1, \cdots, \vec{x}_k)$ iid from $\mu$ and $\vec{\mathbf{x}}'$ similarly sampled from $\mu$')
\begin{align*}
    &\EE^{\mu}[\phi(\mu)]- \EE^{\mu'}[\phi(\mu')] \\
    &= \left(\EE^{\mu}[\phi(\mu)] - \EE^{\mu, \vec{\mathbf{x}}}[\phi(\mu_{P_\mu[k]})]\right) +  \left(\EE^{\mu, \vec{\mathbf{x}}}[\phi(\mu_{P_\mu[k]})] - \EE^{\mu', \vec{\mathbf{x}}'}[\phi(\mu_{P_{\mu'}[k]})]\right) + \left(\EE^{\mu', \vec{\mathbf{x}}'}[\phi(\mu_{P_{\mu'}[k]})] - \EE^{\mu'}[\phi(\mu')]\right) \\
    &= \EE^{\mu, \vec{\mathbf{x}}}\left[\phi(\mu) - \phi(\mu_{P_\mu[k]})\right] +  \left(\EE^{\mu, \vec{\mathbf{x}}}[\phi(\mu_{P_\mu[k]})] - \EE^{\mu', \vec{\mathbf{x}}'}[\phi(\mu_{P_{\mu'}[k]})]\right) + \EE^{\mu', \vec{\mathbf{x}}'}\left[\phi(\mu_{P_{\mu'}[k]}) - \phi(\mu')\right] .
\end{align*}
Taking $k \to \infty$, the middle term is eventually zero by assumption, and by \cref{muvsmuk} the first and last expectations also go to zero (using the same modulus of continuity argument as in the end of the proof of \cref{tfaemain}). Thus $\mu$ and $\mu'$ must have the same expectation for $\phi$, meaning they are equal in distribution.
\end{proof}

Finally, we show a strengthening of \cref{findpermutationsequence} now for random permutons:

\begin{proposition}
Let $\mu$ be a random $d$-dimensional permuton. Then there is a sequence of random $d$-dimensional permutations $\sigma_n$ such that $\mu_{\sigma_n} \to \mu$. Furthermore, for any \vocab{consistent family} of $d$-dimensional permutations (meaning that $|\sigma_n| = n$ and a uniformly random $k$-subset of $\sigma_n$ is distributed as $\sigma_k$ for all $n \ge k$), we have convergence to some unique permuton.
\end{proposition}
\begin{proof}
For the first statement, let $\sigma_n = P_{\mu}[n]$ (notice in particular that for any $n \ge k$, we may pick $n$ points at uniform to form $P_\mu[n]$ and a uniformly random subset of them to form $P_\mu[k]$, so in fact these $\sigma_n$ form a consistent family). Then \cref{muvsmuk2} shows that $d_\square(\mu, \mu_{\sigma_n}) \to 0$ almost surely by Borel-Cantelli, so we have convergence in distribution. On the other hand, if we have a consistent family of permutations, condition (4) of \cref{tfaemain} holds (so that the $\mu_{\sigma_n}$s do converge weakly to some random permuton) and \cref{randunique} yields uniqueness, completing the proof. 
\end{proof}

\section{Background: the skew Brownian permutons and coalescent-walk processes}\label{backgroundsection}

We collect here some preliminary results that will be used later in \cref{schnydersection} and \cref{dseparablesection}.

\subsection{The skew Brownian permutons}\label{skewbrownianpermutonsection}

The purpose of this section is to give a more thorough introduction to the skew Brownian permutons, providing rigorous definitions and references to the known results. For completeness, we restate some of the constructions from the introduction and gather all of the necessary objects in one place.

Let $\rho \in (-1,1)$ and $q \in [0, 1]$ be parameters, and let $W_\rho=(X_\rho, Y_\rho)$ be a two-dimensional \vocab{Brownian excursion of correlation $\rho$ in the first quadrant} on the time interval $[0, 1]$, which is a two-dimensional Brownian motion of correlation $\rho$ conditioned to stay in $\RR_{\ge 0}^2$ and start (at time 0) and end (at time 1) at the origin $(0, 0)$. The details on how to construct such an object can be found for instance in \cite{millersheffield2018}, which cites \cite{shimura} for the original results. Consider the solutions $\left(Z_{\rho, q}^{(u)}(t)\right)_{t\in[0,1]}$ to the following family of stochastic differential equations driven by $(X_\rho, Y_\rho)$ and indexed by $u \in [0, 1]$: 
\begin{equation}\label{skewbrowniansdedefn}
\begin{cases} 
dZ_{\rho, q}^{(u)}(t) = \mathds{1}\left\{Z_{\rho, q}^{(u)}(t) > 0\right\} d Y_\rho(t) - \mathds{1}\left\{Z_{\rho, q}^{(u)}(t) < 0\right\} d X_\rho(t)+ (2q-1) dL^{Z_{\rho, q}^{(u)}}(t), & t \in (u, 1), \\ 
Z_{\rho, q}^{(u)}(t) = 0, & t \in [0, u],
\end{cases}
\end{equation}
where $L^{Z_{\rho, q}^{(u)}}$ is the symmetric local-time process at zero of $Z_{\rho, q}^{(u)}$ -- that is, 
\[
    L^{Z_{\rho, q}^{(u)}} = \lim_{\varepsilon \to 0} \frac{1}{2\varepsilon} \int_0^t \mathds{1}\left\{Z_{\rho, q}^{(u)}(s) \in [-\varepsilon, \varepsilon]\right\} ds.
\]
Intuitively, the sample paths $Z_{\rho, q}^{(u)}(t)$ are started at zero at time $u$, follow $Y_\rho$ (resp.\ $-X_\rho$) above (resp.\ below) the $x$-axis, and have probability $q$ of starting a positive excursion each time they hit $0$. For this SDE, existence and uniqueness of strong solutions is known for $\rho \in (-1, 1)$ (\cite[Theorem 1.7]{borga2023skew}), and furthermore the walks started at $u_1 \ne u_2$ do not cross but do coalesce at a later time (\cite[Proposition 5.2]{borga2022scaling}). This allows us to define the following random function on $[0, 1]$:
\begin{equation}\label{eq:defn-ass-process}
    \phi_{\rho, q}(t) = \text{Leb}\left(\left\{x \in [0, t): Z_{\rho, q}^{(x)}(t) < 0\right\} \cup \left\{x \in [t, 1]: Z_{\rho, q}^{(t)}(x) \ge 0\right\}\right).
\end{equation}

\begin{definition}\label{skewbrownianremark}
With the notation above, the \vocab{skew Brownian permuton $\mu_{\rho, q}$ driven by $(X_\rho, Y_\rho)$ of skewness $q$} is the pushforward of Lebesgue measure on $[0, 1]$ via the mapping $(\text{Id}, \phi_{\rho, q})$. In other words, for any Borel set $A$, we have
\begin{equation}\label{eq:defn-perm}
    \mu_{\rho, q}(A) = (\text{Id}, \phi_{\rho, q})_\ast \text{Leb}(A) = \text{Leb}\left(\left\{t \in [0, 1]: (t, \phi_{\rho, q}(t)) \in A\right\}\right).
\end{equation}

\end{definition}

This measure $\mu_{\rho, q}$ may be interpreted as the continuous-time permutation induced by the collection of paths $Z_{\rho, q}^{(u)}$ for $u \in [0, 1]$. Specifically, almost surely, $Z_{\rho, q}$ induces a (random) total ordering $<_{\rho, q}$ on a subset of $[0, 1]$ of Lebesgue measure 1, such that for $t_1 \le t_2 \in [0, 1]$, we have 
\begin{equation}\label{skewbrownianordering}
    t_1 <_{\rho, q} t_2,\quad\text{ if } \quad Z_{\rho, q}^{(t_1)}(t_2) < 0,\qquad\text{ and }\qquad t_2 <_{\rho, q} t_1, \quad\text{ otherwise.}
\end{equation}
Then $\phi_{\rho, q}(t)$ describes the Lebesgue measure of points smaller than $t$ under this total ordering $<_{\rho, q}$. In particular, as shown in \cite[Lemma 5.5]{borga2022scaling}, almost surely, for almost every $t_1 < t_2 \in[0,1]$,
\begin{equation}\label{rel-order}
    t_1 <_{\rho, q} t_2 \text{ and } \phi_{\rho, q}(t_1)< \phi_{\rho, q}(t_2),\qquad\text{ or}, \qquad t_2 <_{\rho, q} t_1\text{ and } \phi_{\rho, q}(t_2)< \phi_{\rho, q}(t_1).
\end{equation}

Recalling \cref{defnsamplefrompermuton}, we also note for future reference that if $P_{\mu_{\rho, q}}[k]$ is obtained from the $k$ iid points 
\[
    \bigg(\left(v_1,\phi_{\rho, q}(v_1)\right),\dots,\left(v_k,\phi_{\rho, q}(v_k)\right)\bigg),
\]
with $v_1 < \cdots < v_k$ iid uniform random variables relabeled to be ordered (note that these $k$ points have the same law as $k$ iid ordered points sampled from $\mu_{\rho, q}$), then almost surely for all $i$ and $j$,
\begin{equation}\label{eq:patt-from-perm}
    (P_{\mu_{\rho, q}}[k])(i)<(P_{\mu_{\rho, q}}[k])(j)\iff \phi_{\rho, q}(v_i)<\phi_{\rho, q}(v_j).
\end{equation}

We also recall that $\mu_{\rho, q}$ determines $(X_\rho, Y_\rho)$:

\begin{proposition}[{\cite[Theorem 1.1 and Remark 1.4]{borgagwynne2024}}]\label{prop:det-perm}
   Almost surely, the permuton $\mu_{\rho, q}$  determines the two-dimensional Brownian excursion $(X_\rho, Y_\rho)$ of correlation $\rho$, in the sense that there exists a measurable function $G$ such that $(X_\rho, Y_\rho)=G(\mu_{\rho, q})$.
\end{proposition}

\medskip

We next describe a connection between the skew Brownian permuton and certain objects in random geometry. We assume here a certain familiarity of the reader with SLEs and LQG surfaces; for more details, we refer the reader to \cite[Sections 1 and 4]{borga2023skew}. We point out that the next result is not needed for almost all results in our paper, and we only use it for the proof of \cref{SLES-LQG}.

\begin{proposition}[{\cite[Theorem 1.1]{millersheffield2018},\cite[Theorem 1.17, Proposition 4.1]{borga2023skew}}]\label{lqgdescription2}
Let $\gamma \in (0, 2)$, $\kappa = 16/\gamma^2$,  $\chi = \frac{\sqrt \kappa}{2}- \frac{2}{\sqrt \kappa}$, and $\rho = -\cos(\pi \gamma^2/4)$. Let $(\mathbb{C}, h, \infty)$ be a $\gamma$-LQG sphere with quantum area $1$. Let $\hat{h}$ be a whole-plane Gaussian free field (viewed modulo additive multiples of $2\pi\chi$) independent of the LQG sphere. For $\theta \in [0, \pi]$, let $(\eta_0, \eta_\theta)$ be two whole-plane space-filling SLE$_\kappa$ counter-flow lines of $\hat{h}$ of angle  0 and $\theta$, respectively. Parameterize the pair of curves by the $\mu_h$-LQG area measure so that $\eta_0(0) = \eta_0(1) = \eta_\theta(0) = \eta_\theta(1) = \infty$ and $\mu_h(\eta_0([0, t])) = \mu_h(\eta_\theta([0, t])) = t$ for all $t \in [0, 1]$.

Let $\psi_{\gamma, \theta}:[0,1]\to [0,1]$ be a measurable function such that $\eta_0(t)=\eta_\theta(\psi_{\gamma, \theta}(t))$. Then the measure\footnote{In \cite{borga2023skew} (and also \cite{borga2021permuton}), the notion of left and right boundaries for counter-flow lines are swapped compared to this paper. This is why we have here $1-\psi_{\gamma, \theta}$ instead of $\psi_{\gamma, \theta}$.} 
\[
    (\text{Id}, 1-\psi_{\gamma, \theta})_\ast \text{Leb}
\]
is a skew Brownian permuton of parameters $(\rho, q)$ for some (nonexplicit) $q = q_\gamma(\theta)$.

\medskip

Furthermore, define $X_\rho(t)$ and $Y_\rho(t)$ to be the $\nu_h$-LQG length measures for the left and right outer boundaries of $\eta_0([0, t])$. Then, up to time reparameterization, $(X_\rho, Y_\rho)$ has the law of a two-dimensional Brownian excursion of correlation $\rho$, and it determines the curve-decorated quantum surface $((\mathbb{C} \cup \{\infty\}, h, \infty), \eta_0)$. More precisely, $(X_\rho(s), Y_\rho(s))|_{s\in[0,t]}$ determines $((\mathbb{C} \cup \{\infty\}, h, \infty), \eta_0)$ restricted to $\eta_0([0,t])$.

Finally, if we define $Z_{\rho, q}$ to be the solutions to \cref{skewbrowniansdedefn} driven by $(X_\rho, Y_\rho)$ with parameter $q$ and consider the corresponding function $\phi_{\rho, q}$ as in \cref{eq:defn-ass-process}, then almost surely we have that
\[
    1-\psi_{\gamma,\theta}(t) = \phi_{\rho, q}(t),\qquad\text{for almost all $t \in (0, 1)$.}
\]
\end{proposition}

\begin{remark}\label{rem:angle}
    It is proved\footnote{Note that our parameter $q$ is equal to $1-p$, where $p$ is as in \cite[Propositions 4.4 and 5.10]{li2022schnyder}.} in \cite[Propositions 4.4 and 5.10]{li2022schnyder} that $q_{\gamma=1}(\frac{2\pi}{3})=\frac{1}{1+\sqrt{2}}$. Moreover, from \cite[Remark 1.18]{borga2023skew}, we have that $q_\gamma(\theta)+q_\gamma(\pi-\theta)=1$ for all $\theta\in[0,1]$ (and all $\gamma\in(0,2)$). In particular, we have that $q_{\gamma=1}(\frac{\pi}{3})=1-\frac{1}{1+\sqrt{2}}$.
\end{remark}

\subsection{The Brownian separable $d$-permutons}

Next, we precisely describe the degenerate skew Brownian permutons $\mu_{1,q}$ and the relation that is used to define them in an alternate way. For further details, see for instance \cite[Section 1.4]{maazoun2020brownian}. Let $e(t)$ be a one-dimensional Brownian excursion on $[0, 1]$, and let $(s(\ell))$ be an iid sequence of signs which are each $+1$ with probability $p$ and $-1$ with probability $1-p$, indexed by the (countable) local minima of $e$. Much like the SDE in \cref{skewbrowniansdedefn} yields a random order on $[0, 1]$, we may define an order $<_{e, p}$ by the following rule. There is a random set $N$ of measure zero such that for any $x, y \in [0, 1] \setminus N$ (say $x < y$ without loss of generality), $e$ achieves a unique minimum on $[x, y]$ at some point $\ell$. We then say that 
$$x <_{e, p} y,\quad \text{ if }\quad s(\ell) = 1, \quad \text{ and } \quad y <_{e, p} x\quad \text{ otherwise},$$ 
and this yields a total order on $[0, 1] \setminus N$. This thus allows us to define the random function
\begin{equation}\label{biasedbrownianpsi}
    \psi_{e, p}(t) = \text{Leb}\left(\left\{x \in [0, 1]: x <_{e, p} t\right\}\right)
\end{equation}
and the corresponding \vocab{Brownian separable permuton} of parameter $p$
\[
    \mu^B_p(A) = (\text{Id}, \psi_{e, p})_\ast \text{Leb}(A) = \text{Leb}\left(\left\{t \in [0, 1]: (t, \psi_{e, p}(t)) \in A\right\}\right).
\]
With this, we are now prepared to rigorously describe the $d$-dimensional generalization of this construction:

\begin{definition}\label{brownian-separable-d-permuton}
Let $p_1, \cdots, p_{d-1} \in [0, 1]$ be real numbers. The \vocab{Brownian separable $d$-permuton} $\mu^B_{p_1, \cdots, p_{d-1}}$ of parameters $p_1, \cdots, p_{d-1}$ is the $d$-dimensional permuton defined as follows. Let $e(t)$ be a one-dimensional Brownian excursion on $[0, 1]$, and let $(s(\ell))$ be an iid sequence of variables of the form $s_\ell = (s_1(\ell), \cdots, s_{d-1}(\ell))$ indexed by the local minima $\ell$ of $e(t)$, where the $s_i$ are independently $+1$ with probability $p_i$ and $-1$ with probability $1 - p_i$. 

For all $1 \le i \le d-1$, we may then define the relation $<_e^{(i)}$ similarly to $<_{e, p}$ above, setting $x <_e^{(i)} y$ if $s_i(\ell) = 1$ and $y <_e^{(i)} x$ otherwise. Defining $\psi_e^{(i)}(t)$ as in \cref{biasedbrownianpsi} but now with the relation $<_e^{(i)}$ in place of $<_{e, p}$, the Brownian separable $d$-permuton of parameters $p_1, \cdots, p_{d-1}$ is then defined by
\begin{equation}\label{brownian-separable-d-permuton-eqn}
    \mu^B_{p_1, \cdots, p_{d-1}}(A) = (\text{Id}, \psi_e^{(1)}, \cdots, \psi_e^{(d-1)})_\ast \text{Leb}(A)=\text{Leb}\left(\left\{t \in [0, 1]: (t,  \psi_e^{(1)}(t), \cdots, \psi_e^{(d-1)}(t)) \in A\right\}\right).
\end{equation}
\end{definition}

While this construction may appear difficult to work with, several helpful properties of Brownian excursions turn out to make studying pattern frequencies of $\mu^B_{p_1, \cdots, p_{d-1}}$ tractable, as we will see in \cref{prop:pattern-sep-perm}.

\subsection{Coalescent-walk processes}\label{coalescentintrosection}

Coalescent-walk processes, first introduced in \cite{borga2022scaling}, are collections of random walks coupled by a single driving two-dimensional walk. Permuton convergence for uniform Baxter, strong-Baxter, and semi-Baxter permutations have each been previously obtained by constructing bijections of those permutations with certain two-dimensional walks and then using the walks to drive coalescent-walk processes that can also recover the original permutation. In each case, the limiting driving walk converges to a two-dimesnional Brownian excursion, and the discrete coalescent-walk process converges to a continuous coalescent-walk process described by an SDE, which leads to a description of the permuton as in \cref{skewbrownianpermutonsection}.

For the three-dimensional Schnyder wood permutations considered in \cref{schnydersection}, a similar encoding will be described using a pair of two-dimensional walks and thus a pair of coalescent-walk processes. However, we will need a more general definition here than the one introduced in previous works.

\begin{definition}\label{coalescentdefn}
Let $I$ be an interval of $\mathbb{Z}$, and let $J \subseteq I$. A \vocab{coalescent-walk process on $I$ with starting points in $J$} is a collection of one-dimensional random walks $Z = \{(Z_s^{(j)})_{s \ge j, s \in I}\}_{j \in J}$ satisfying the following two conditions:
\begin{itemize}
    \item For all $j \in J$, we have $Z_j^{(j)} = 0$,
    \item For any $j_1, j_2 \in J$, if $Z_s^{(j_1)} \ge Z_s^{(j_2)}$, then $Z_{s'}^{(j_1)} \ge Z_{s'}^{(j_2)}$ for all $s' \ge s$.
\end{itemize}
We will typically refer to each walk $Z_s^{(j)}$ as a \vocab{path} or \vocab{sample path} of $Z$.
\end{definition}

In words, a coalescent-walk process is a collection of random walks started at zero which do not cross. Furthermore, the definition implies that if $Z_s^{(j_1)} = Z_s^{(j_2)}$, then $Z_{s'}^{(j_1)} = Z_{s'}^{(j_2)}$ for all $s' \ge s$ -- that is, the two paths coalesce for all future time. This second condition allows us to establish a total order on the set of sample paths, or more precisely the starting points $J$, in one of the following two ways: 

\begin{definition}\label{totalorder}
Let $Z$ be a coalescent-walk process on $I$ with starting points in $J$. Define the order $\le^{\text{up}}$ on $J$ as follows:
\begin{itemize}
    \item For all $j \in J$, $j \le^{\text{up}} j$.
    \item For all $j_1 < j_2 \in J$, we have $j_1 \le^{\text{up}} j_2$ if $Z^{(j_1)}_{j_2} < 0$ and $j_2 \le^{\text{up}} j_1$ otherwise.
\end{itemize}
Also define the order $\le^{\text{down}}$ on $J$ similarly, except now with $j_1 \le^{\text{down}} j_2$ if $Z^{(j_1)}_{j_2} > 0$ and $j_2 \le^{\text{down}} j_1$ otherwise.
\end{definition}

\cref{coalescentpermutationexample} below provides an example. The same proof strategy as in \cite[Proposition 2.9]{borga2022scaling} shows that both $\le^{\text{up}}$ and $\le^{\text{down}}$ in fact yield total orders on $J$. This enables us to define two permutations associated to a coalescent-walk process: 

\begin{definition}\label{coalescentpermutation}
Let $Z$ be a coalescent-walk process on $I$ with starting points in $J$, and let $n = |J|$. Given a labeling $f: J \to [n]$ of the starting points, the permutation $\sigma = \sigma^{\text{up}}(Z)$ is the only map $[n] \to [n]$ such that for all $j_1, j_2 \in J$,
\[
    j_1 \le^{\text{up}} j_2 \iff \sigma(f(j_1)) \le \sigma(f(j_2)).
\]
Define $\sigma^{\text{down}}(Z)$ analogously but with the order $\le^{\text{down}}$ instead. 
\end{definition}

For any $j \in J$, we refer to both the sample path started at $j$ and the starting point itself as being labeled $f(j)$.

\begin{example}\label{coalescentpermutationexample}
Consider the coalescent-walk process $Z$ shown in \cref{fig:samplecoalescent}. The total orders on $J = \{2, 3, 5, 7\}$ induced by the process are $7 \le^{\text{up}} 5 \le^{\text{up}} 3 \le^{\text{up}} 2$ and $5 \le^{\text{down}} 2 \le^{\text{down}} 3 \le^{\text{down}} 7$, so the permutations given by the labeling $f(2) = 1, f(3) = 2, f(5) = 3, f(7) = 4$ are $\sigma^{\text{up}}(Z) = (4, 3, 2, 1)$ and $\sigma^{\text{down}}(Z) = (2, 3, 1, 4)$.
\end{example}

\begin{figure}[ht]
\begin{center}
\begin{tikzpicture}[scale=0.85]
\draw[thick, ->] (-1, 0) -- (8.5, 0);
\draw[thick, ->] (0, -2) -- (0, 3.5);
\draw[gray, thin] (1, -2) -- (1, 3);
\draw[gray, thin] (2, -2) -- (2, 3);
\draw[gray, thin] (3, -2) -- (3, 3);
\draw[gray, thin] (4, -2) -- (4, 3);
\draw[gray, thin] (5, -2) -- (5, 3);
\draw[gray, thin] (6, -2) -- (6, 3);
\draw[gray, thin] (7, -2) -- (7, 3);
\draw[gray, thin] (8, -2) -- (8, 3);
\draw[gray, thin] (-1, -2) -- (8, -2);
\draw[gray, thin] (-1, -1) -- (8, -1);
\draw[gray, thin] (-1, 1) -- (8, 1);
\draw[gray, thin] (-1, 2) -- (8, 2);
\draw[gray, thin] (-1, 3) -- (8, 3);
\draw[very thick, green!50!black] (2, 0) -- (3, 2) -- (4, 2) -- (5, 0);
\draw[very thick, green!50!black] (3, 0) -- (4, -1) -- (5, 0) -- (6, 2) -- (7, 1) -- (8, 0);
\draw[very thick, green!50!black] (7, 0) -- (8, -2);
\draw[green!50!black, fill=green!50!black] (2, 0) circle(0.1cm);
\draw[green!50!black, fill=green!50!black] (3, 0) circle(0.1cm);
\draw[green!50!black, fill=green!50!black] (5, 0) circle(0.1cm);
\draw[green!50!black, fill=green!50!black] (7, 0) circle(0.1cm);
\draw[very thick, dashed, yellow!50!black] (8, 0.25) -- (6, 2.25) -- (5, 0.4) -- (4.125, 2.2) -- (2.875, 2.2) -- (1.75, 0) -- (2, -0.4) -- (3.125, 1.8) -- (3.875, 1.8) -- (4.75, 0) -- (4, -0.75) -- (3, 0.25) -- (2.75, 0) -- (4, -1.25) -- (5.25, 0) -- (6.05, 1.7) -- (8, -0.25);
\draw[very thick, dashed, yellow!50!black] (8, -1.65) -- (7, 0.35) -- (6.75, 0) -- (7.8, -2);
\node[blue] at (1.6, 0.3) {\large{1}};
\node[blue] at (3.4, 0.3) {\large{2}};
\node[blue] at (5.4, 0.3) {\large{3}};
\node[blue] at (6.6, 0.3) {\large{4}};
\node[gray] at (1, -0.5) {\small{1}};
\node[gray] at (2, -0.5) {\small{2}};
\node[gray] at (3, -0.5) {\small{3}};
\node[gray] at (4, -0.5) {\small{4}};
\node[gray] at (5, -0.5) {\small{5}};
\node[gray] at (6, -0.5) {\small{6}};
\node[gray] at (7, -0.5) {\small{7}};
\node[gray] at (8, -0.5) {\small{8}};
\end{tikzpicture}
\end{center}
\caption{An example of a coalescent-walk process $Z$ on $I = [8]$ with $J = \{2, 3, 5, 7\}$. The gray labels indicate the points in $I$. The labeling used to recover the permutations $\sigma^{\text{up}}(Z)$ and $\sigma^{\text{down}}(Z)$, shown in blue, is given by $f(2) = 1, f(3) = 2, f(5) = 3, f(7) = 4$. The traversal of the forest is shown in gold (see the explanations below \cref{coalescentpermutationexample} for more details).} \label{fig:samplecoalescent}
\end{figure}
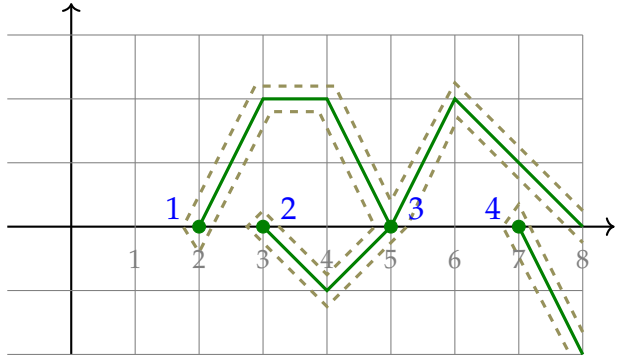

An alternate way to understand these permutations is by viewing the coalescent-walk process as a forest of trees with vertices given by the starting points of the sample paths. Here, each collection of paths that has coalesced together forms a tree, and edges of the tree correspond to pieces of paths between two starting points. The trees may then be ordered from bottom to top by the final height of the path at the end of the interval $I$. With this perspective, $\sigma^{\text{up}}(Z)$ traverses the trees from bottom to top, such that $i$ is the $(\sigma^{\text{up}}(Z)(i))$--th visited label by the \emph{clockwise} traversal. Similarly, $\sigma^{\text{down}}(Z)$ traverses the trees from top to bottom, such that $i$ is the $(\sigma^{\text{down}}(Z)(i))$--th visited label by the \emph{counterclockwise} traversal. This perspective aligns with how our Schnyder wood permutations are defined in terms of the traversals of the green and red trees (\cref{schnyderpermutationdefn}) and subsequently recovered by the traversals of the trees in the coalescent-walk processes (\cref{mainresult1}).

\begin{example}
Consider now the forest traversal marked in gold in \cref{fig:samplecoalescent}. The bottom-to-top (clockwise) traversal first traverses the tree containing the vertex labeled $4$, and then it traverses the tree with the vertices labeled $1$, $2$, and $3$, first visiting $3$ before visiting its children $2$ and finally $1$. Meanwhile, the top-to-bottom (counterclockwise) traversal first traverses the tree with vertices labeled $1, 2, 3$ (again visiting $3$ first, but now visiting label $1$ before label $2$) and then the tree with label $4$. These traversal orders agree with the permutations calculated in \cref{coalescentpermutationexample}.
\end{example}

In \cref{schnydersection}, two explicit coalescent-walk processes $Z^g_M$ and $Z^r_M$ will be associated to each Schnyder wood triangulation $M$, such that $\sigma_M^g = \sigma^{\text{up}}(Z^g_M)$ and $\sigma_M^r = \sigma^{\text{down}}(Z^r_M)$.  These coalescent-walk processes will be driven by a single 2-dimensional random walk $W_M = (X_M, Y_M)$ bijectively encoding $M$. The details are described in \cref{combinatorics}, but in short, the sample paths of $Z^r_M$ roughly follow the increments of $Y_M$ above the $x$-axis and of $-X_M$ below the $x$-axis, with some additional local rules near $0$; meanwhile, the sample paths of $Z^g_M$ do the same but with the time-reversal of $(X_M, Y_M)$ and with slightly different local rules. In particular, these (discrete) coalescent-walk processes will resemble the ``continuous coalescent-walk processes'' of the skew Brownian permuton defined in \cref{skewbrowniansdedefn}.

\section{The random 3-dimensional permuton limit of Schnyder wood permutations}\label{schnydersection}

The main goal of this section is to prove \cref{mainresult3}, that is, to show that uniform 3-dimensional Schnyder wood permutations, introduced in \cref{schnyderpermutationdefn}, converge in the permuton sense to the Schnyder wood permuton $\mu_S$ defined in the statement of \cref{mainresult3}.

Our main tool to show permuton convergence is to use coalescent-walk processes (\cref{coalescentdefn}). Hence, given a Schnyder wood triangulation $M$ with corresponding Schnyder wood permutation $\sigma_M = (\sigma_M^g, \sigma_M^r)$, we first encode the two 2-dimensional marginals $\sigma_M^g$ and $\sigma_M^r$, and in fact the tree structure of the green and red spanning trees of $M$, through two coalescent-walk processes (as shown in \cref{mainresult1}). We then use this encoding to prove \cref{mainresult3} in \cref{walklimits,permutonsubsection}.

The rest of this section is organized as follows:
\begin{itemize}
    \item The main goal of \cref{combinatorics} is to state the encoding of  Schnyder wood permutations in terms of coalescent-walk processes.
    We present here all the necessary combinatorial constructions and results needed in \cref{walklimits,permutonsubsection} and postpone their proofs to \cref{schnydercombinatoricssectionproofs}, since they are quite long but not relevant for our main goal.
    \item \cref{walklimits} develops the preliminary probabilistic results needed for the proof of \cref{mainresult3}. In \cref{walklimits-1}, we state the main result for convergence of discrete coalescent-walk processes to their continuous scaling limits (\cref{unconditionedconvergence}). The proof of this proposition, given in \cref{tech-proof}, is important but quite technically involved and should be skipped on a first read.

    \item Finally in \cref{permutonsubsection}, building on \cref{unconditionedconvergence}, we complete the proof of \cref{mainresult3}.
\end{itemize}

\subsection{Encoding Schnyder wood triangulations and permutations through pairs of coalescent-walk processes}\label{combinatorics}

Before we begin the constructions of the coalescent-walk processes, we state the following result, whose proof is postponed to the beginning of \cref{schnydercombinatoricssectionproofs}.

\begin{proposition}\label{schnyderwoodbijectperm}
The map $M \mapsto \sigma_M=(\sigma_M^g, \sigma_M^r)$ introduced in \cref{schnyderpermutationdefn} is injective, meaning that there is a bijection between Schnyder wood triangulations of size $n$ and Schnyder wood permutations of size $n$. 
\end{proposition}

Both coalescent-walk processes encoding $\sigma_M = (\sigma_M^g, \sigma_M^r)$ will now be defined in terms of a two-dimensional random walk which bijectively encodes the structure of the Schnyder wood triangulation $M$. We will do so with the help of the following result. 

\begin{proposition}[\cite{li2022schnyder}, Theorem 1.4]\label{schnyderwoodstring}
Each Schnyder wood triangulation $M$ of size $n$ can be associated with a \vocab{Schnyder wood string} $s_M$ with characters in $\{\texttt{g},\texttt{b},\texttt{r}\}$ -- containing $n$ copies of \texttt{g}, $n$ copies of \texttt{b} and $n$ copies of \texttt{r} -- in the following way (c.f.\ \cref{fig:sampleschnyderloop}). A loop is traced around the blue tree of $M$, starting on the outer side of the edge between the green and red root, circling around the green root in clockwise order, traveling along the boundary of the blue tree in clockwise order, circling around the red root in clockwise order, and finally returning to the starting point. A \texttt{g}, \texttt{b}, or \texttt{r} is marked down each time this loop crosses or visits a green, blue, or red edge (respectively) for the second time\footnote{For blue edges, the ``second time'' corresponds to the time when the loop visits the second side of each blue edge.}.

This Schnyder wood string has the property that for any prefix of the string, there are at least as many \texttt{g}s as \texttt{b}s and at least as many \texttt{b}s as \texttt{r}s, and \texttt{r} steps are never followed by a \texttt{b} step. Furthermore, the map $M \mapsto s_M$ is a bijection between Schnyder wood triangulations of size $n$ and strings with this property.
\end{proposition}

An example of this construction can be seen in \cref{fig:sampleschnyderloop}. Since both the loop (shown in gold) and the numbering of the blue tree are constructed via a clockwise traversal, the loop visits the vertices in ascending order with respect to the blue labeling.

\begin{figure}[ht]
\centerline{\includegraphics[width=14cm]{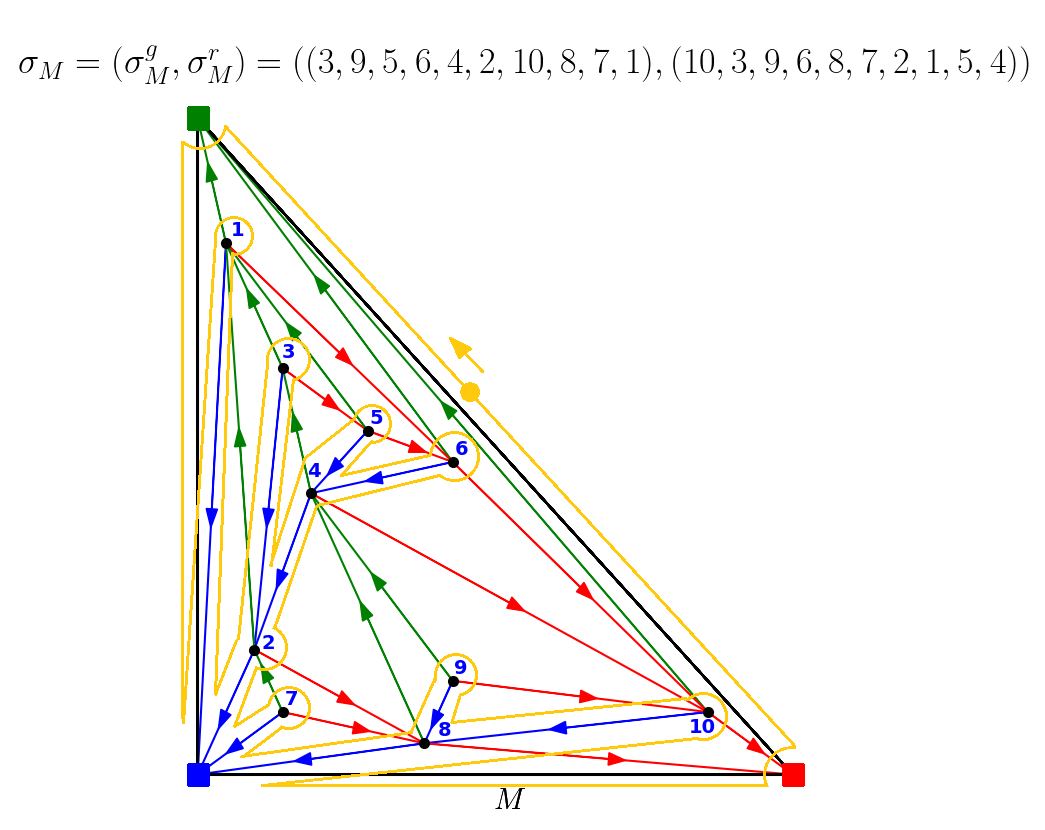}}
\caption{\label{fig:sampleschnyderloop}The Schnyder wood triangulation $M$ from \cref{fig:sampleschnyder} with traced loop described in \cref{schnyderwoodstring} shown in gold. The loop begins at the marked vertex and is traversed in the direction indicated by the gold arrow. Traveling around this loop, we can determine the Schnyder wood string $s_M = \texttt{gbggbgrgbrrgbbbgbrrggbrrrgbbrr}$, following the instructions given in \cref{schnyderwoodstring}.}
\end{figure}

We now describe a procedure for constructing a two-dimensional random walk $W_M = (X_M, Y_M)$ from the string $s_M$ obtained in \cref{schnyderwoodstring}. First, move the leading character of $s_M$ (which will always be a \texttt{g} by \cref{schnyderwoodstring}) to the end of $s_M$. The string now consists only of \texttt{b}s and (possibly empty) segments of \texttt{r}s followed by a single \texttt{g}. Replace each \texttt{b} with an increment of $(1, -1)$, and replace each \texttt{rr}$\cdots$\texttt{rrg} with an increment of $(-k, 1)$, where $k\geq 0$ is the number of \texttt{r}s in the segment. 

\begin{lemma}\label{randomwalkisuniformschnyder1}
The above process $s_M \mapsto W_M$ is well-defined and yields a bijection between Schnyder wood strings of length $3n$ and walks of length $2n$ with increments in the set $\{(-k, 1): k \ge 0\} \cup \{(1, -1)\}$, starting and ending at the origin, and conditioned to stay in $\{x \ge 0, y \ge -1\}$.
\end{lemma}
\begin{proof}
Schnyder's rule specifies that (possibly empty) blocks of contiguous \texttt{r} steps are always followed by \texttt{g} steps except at the end of the Schnyder wood string. Thus, by moving the first \texttt{g} to the end, all \texttt{r}s are part of an \texttt{rr}$\cdots$\texttt{rrg} segment. Since each step of $W_M$ corresponds to either a \texttt{b} or \texttt{g} in $s_M$, there will indeed be $2n$ steps, and they all have increments in the set listed. Furthermore, $W_M$ will indeed start and end at the origin because the $y$-coordinate is incremented by $1$ for each \texttt{g} and decremented by $1$ for each \texttt{b}, and the $x$-coordinate is incremented by $1$ for each \texttt{b} and decremented by $1$ for each \texttt{r}. The restriction to stay in $\{x \ge 0, y \ge -1\}$ is because (by \cref{schnyderwoodstring}) we must have at least as many \texttt{g}s as \texttt{b}s and at least as many \texttt{b}s as \texttt{r}s in any prefix before moving the first \texttt{g} to the end.

Finally, to check that this is a bijection, we can always take any walk satisfying these conditions, replace each increment of $(1, -1)$ with a \texttt{b} and each $(-k, 1)$ with $k$ \texttt{r}s followed by a \texttt{g}, and then finally move the final character (which must be a \texttt{g} because our walk must end with a negative $x$-increment) to the beginning. The reasoning in the previous paragraph shows this is a valid Schnyder wood string, so we have a valid inverse map.
\end{proof}

The right-hand side of \cref{fig:sampleschnyder2} displays the running values of the $y$-coordinate $Y_M$ (shifted up by $2$ units) and the running values of the negative $x$-coordinate $-X_M$ (shifted down by $2$ units) in gold for the walk $W_M = (X_M, Y_M)$ obtained from the string $s_M$ in \cref{fig:sampleschnyderloop}. We also define the \vocab{time-reversal} of $W_M$, denoted $W_M' = (X_M', Y_M')$, to be the two-dimensional walk which swaps the order and sign of all increments of $W_M$. The running values of $Y_M'$ and $-X_M'$ are shown in gold on the left-hand side of \cref{fig:sampleschnyder2}.

\begin{figure}[ht]
\centerline{\includegraphics[width=\textwidth]{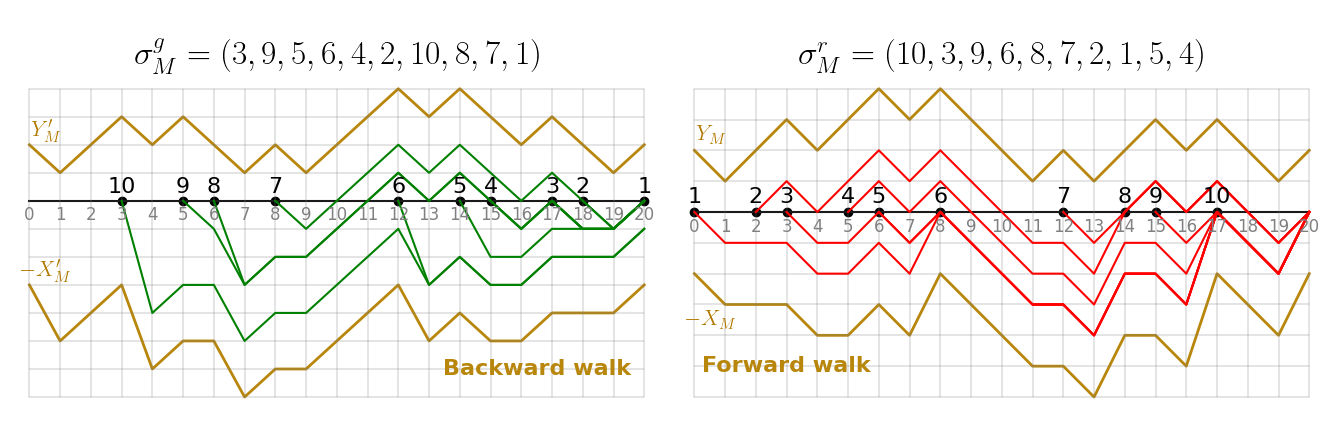}}
\caption{The two coalescent-walk processes introduced in \cref{cwrdefn2} and \cref{cwrdefn} associated to the Schnyder wood triangulation in \cref{fig:sampleschnyder}. Both are defined on the time interval $I=[0, 20]$; the black numbers (above the $x$-axis) are the labels of the walks started at those points. \textbf{Left:} This is the coalescent-walk process $Z^g_M = \WC^g(W_M')$. Note that here $J=\{3,5,6,8,12,14,15,17,18,20\}$. \textbf{Right:} This is the coalescent-walk process $Z^r_M = \WC^r(W_M)$. Note that here $J=\{0,2,3,5,6,8,12,14,15,17\}$.\label{fig:sampleschnyder2}}
\end{figure}

We are now ready to describe the construction of the two coalescent-walk processes (recall \cref{coalescentdefn}) which will encode the two 2-dimensional marginals $\sigma_M^g$ and $\sigma_M^r$. This will be done by defining two mappings $\WC^g$ and $\WC^r$, which each take as input a two-dimensional random walk and output a coalescent-walk process. The two coalescent-walk processes 
\begin{equation*}
    Z^g_M = \WC^g(W_M')\qquad \text{and} \qquad Z^r_M = \WC^r(W_M)
\end{equation*}
are then the central objects of study in later sections.

\begin{definition}\label{cwrdefn2}
Let $W = (X_t, Y_t)_{t\in I}$ be a random walk on a time interval $I = [a, b]$ whose increments all lie in the set $\{(k, -1): k \ge 0\} \cup \{(-1, 1)\}$. The \vocab{green coalescent-walk process associated to} $W$, denoted $\WC^g(W)$, is the coalescent-walk process $Z = \{(Z_s^{(j)})_{s \ge j, s \in I}\}_{j \in J}$ defined as follows. The starting points $J$ are the starting times of the $(k, -1)$ increments -- that is,  
\begin{equation*}
    J = \left\{j \in I \setminus \{b\}: Y_{j+1} - Y_j = -1\right\},
\end{equation*}
except that if there would be a starting point at the beginning of $I$, we move it to the end of $I$. Now for any $s \ge j$ such that $s+1 \in I$, the increment $Z_{s+1}^{(j)} - Z_s^{(j)}$ is determined as follows:
\begin{itemize}
    \item If $W_{s+1} - W_s = (k, -1)$, then $Z_{s+1}^{(j)} - Z_s^{(j)} = 
        \begin{cases} 
        -1 & \text{if }Z_s^{(j)} > 0, \\
        -k & \text{if }Z_s^{(j)} < 0, \\
        -k-1 & \text{otherwise.}
        \end{cases}$
    \item If $W_{s+1} - W_s = (-1, 1)$, then $Z_{s+1}^{(j)} - Z_s^{(j)} = 1$.
\end{itemize}
\end{definition}

The left-hand side of \cref{fig:sampleschnyder2} shows the green coalescent-walk process $\WC^g(W_M')$. In words, each sample path of the process typically copies the increments of $Y_M'$ while on or above the $x$-axis and the increments of $-X_M'$ otherwise. However, if the path is at height $0$ (on the $x$-axis) and there is a $(k, -1)$ increment, then it takes a step of $-k-1$ instead of just $-1$.

\begin{definition}\label{cwrdefn}
Let $W = (X_t, Y_t)_{t\in I}$ be a random walk on a time interval $I= [a, b]$ whose increments all lie in the set $\{(-k, 1): k \ge 0\} \cup \{(1, -1)\}$. The \vocab{red coalescent-walk process associated to} $W$, denoted $\WC^r(W)$, is the coalescent-walk process $Z = \{(Z_s^{(j)})_{s \ge j, s \in I}\}_{j \in J}$ defined as follows. The starting points $J$ are the ending times of the $(-k, 1)$ increments -- that is,  
\begin{equation*}
    J = \{j \in I\setminus\{a\}: Y_{j} - Y_{j-1} = +1\},
\end{equation*}
except that if there would be a starting point at the end of $I$, we move it to the beginning of $I$. Now for any $s \ge j$ such that $s+1 \in I$, the increment $Z_{s+1}^{(j)} - Z_s^{(j)}$ is determined as follows:
\begin{itemize}
    \item If $W_{s+1} - W_s = (-k, 1)$, then $Z_{s+1}^{(j)} - Z_s^{(j)} = 
        \begin{cases} 
        1 & \text{if }Z_s^{(j)} \ge 0, \\
        k & \text{if }Z_s^{(j)} < 0 \text{ and }Z_s^{(j)} + k \le 0, \\
        -Z_s^{(j)} & \text{otherwise.}
        \end{cases}$
    \item If $W_{s+1} - W_s = (1, -1)$, then $Z_{s+1}^{(j)} - Z_s^{(j)} = -1$.
\end{itemize}
\end{definition}

The right-hand side of \cref{fig:sampleschnyder2} shows the coalescent-walk process $\WC^r(W_M)$. Just like with the green coalescent-walk process, each sample path of this process typically copies the increments of $Y_M$ while on or above the $x$-axis and the increments of $-X_M$ otherwise. However, if the walk would have crossed over $0$, it instead only increments up to $0$. 

The reason for the special rules about starting points is related to the cycling of the Schnyder wood string in the construction above \cref{randomwalkisuniformschnyder1}, and it will become clearer in the proof of \cref{mainresult1} in \cref{schnydercombinatoricssectionproofs}.

\medskip

We are now ready to state the main result of this subsection, which states that the two $2$-dimensional marginals of a Schnyder wood permutation $\sigma_M=(\sigma_M^g,\sigma_M^r)$ -- and, in fact, even the shape of the corresponding green and red trees of $M$ -- can be read off from the two coalescent-walk processes $Z^g_M = \WC^g(W_M')$ and $Z^r_M = \WC^r(W_M)$. Recall from \cref{coalescentpermutation} that, given a coalescent-walk process $Z$ as in \cref{coalescentdefn}, one can canonically construct two corresponding permutations $\sigma^{\text{up}}(Z)$ and $\sigma^{\text{down}}(Z)$. 

\begin{proposition}\label{mainresult1}
For any Schnyder wood triangulation $M$ of size $n$, let $Z_M^g = \WC^g(W_M')$ and $Z^r_M = \WC^r(W_M)$ be its green and red coalescent-walk processes. Then $Z_M^g$ and $Z_M^r$ are indeed coalescent-walk processes in the sense of \cref{coalescentdefn}, so \cref{totalorder,coalescentpermutation} may be applied to them. Label the starting points of $Z_M^g$ in descending order $n, n-1, \cdots, 1$ from left to right, and label the starting points of $Z_M^r$ in ascending order $1, 2, \cdots, n$. Then the Schnyder wood permutation satisfies 
\begin{equation*}
    \sigma_M = \left(\sigma^{\text{up}}(Z^g_M), \sigma^{\text{down}}(Z^r_M)\right).
\end{equation*}
Additionally, for any $i, j \in [n]$ (we call a vertex of $M$ by its blue label), 
\begin{align*}
\text{there is a green}&\text{ (resp.\ red) directed edge from $i$ to $j$ in $M$} \\ 
&\iff \text{``the sample path labeled $i$ in $Z^g_M$ (resp.\ $Z^r_M$) first hits $j$,''}
\end{align*}
that is, if and only if the path in $Z_M^g$ (resp.\ $Z^r_M$) starting at the point with label $i$ passes through the starting point with label $j$, and it does so before passing through any others. 

Furthermore, there is a green (resp.\ red) directed edge from $i$ to the green (resp.\ red) root if and only if the path in $Z_M^g$ (resp.\ $Z^r_M$) labeled $i$ does not pass through any other starting points.
\end{proposition}

The proof of \cref{mainresult1} is given in \cref{schnydercombinatoricssectionproofs}.

For example, in the red coalescent-walk process of \cref{fig:sampleschnyder2}, the sample paths started at labels $9$, $4$, and $6$ all first hit the starting point labeled $10$, and correspondingly in \cref{fig:sampleschnyder}, there are red directed edges from vertices $9$, $4$ and $6$ to vertex $10$. ($9, 4, 6$ is also both the clockwise order of the children of vertex $10$ in $M$, as well as the order -- from top to bottom -- in which the paths appear in the red coalescent-walk process at the time labeled by $9$.) Similarly, in the green coalescent-walk process of \cref{fig:sampleschnyder2}, the sample paths started at $10$, $6$ and $1$ do not intersect any other starting points, and those are the vertices from which directed edges to the green root appear in the Schnyder wood triangulation in \cref{fig:sampleschnyder}. (And similarly, $10, 6, 1$ is the clockwise order of the children of the green root in $M$, as well as the order -- from bottom to top -- in which the paths appear in the green coalescent-walk process.)

\medskip

For conciseness, when referring to sample paths in coalescent-walk processes with labeled starting points, we will often simply write ``the sample path labeled $i$'' to indicate ``the sample path started at the vertex labeled by $i$''.

We conclude by noting that, as a consequence of the results discussed in this subsection, the diagram of functions in \cref{fig-comm-diagram} is a commutative diagram of bijections between objects of size $n$. 

\begin{figure}[h]
\centerline{\includegraphics[width=0.4\textwidth]{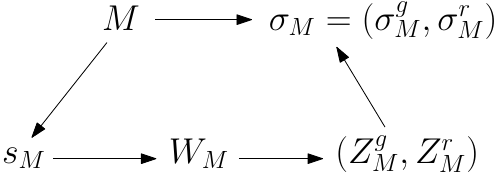}}
\caption{\label{fig-comm-diagram} The commutative diagram of bijections between objects of size $n$. The bijection $M \mapsto \sigma_M=(\sigma_M^g,\sigma_M^r)$ has been introduced in \cref{schnyderpermutationdefn} and proved to be a bijection in \cref{schnyderwoodbijectperm}, the bijection $M\mapsto s_M$ has been introduced in \cref{schnyderwoodstring}, the bijection $s_M\mapsto W_M$ has been introduced in \cref{randomwalkisuniformschnyder1}, the mapping $W_M \mapsto (Z_M^g,Z^r_M) = (\WC^g(W_M'),\WC^r(W_M))$ has been introduced in \cref{cwrdefn,cwrdefn2}, and finally, the mapping $(Z_M^g,Z^r_M) \mapsto \sigma_M=(\sigma_M^g,\sigma_M^r)$ has been introduced in \cref{mainresult1}. This proposition also shows that the diagram is a commutative diagram of bijections.}
\end{figure}

\subsection{Preliminary probabilistic results: scaling limits of the unconditioned coalescent-walk processes for Schnyder wood permutations}\label{walklimits}

The main goal of this section is to develop the preliminary probabilistic results needed for the proof of \cref{mainresult3} later in \cref{permutonsubsection}.

\subsubsection{Statements of the preliminary probabilistic results}\label{walklimits-1}

We begin by explaining how to sample a uniform Schnyder wood permutation of size $n$ via a conditioned 2-dimensional random walk.

\begin{proposition}\label{randomwalkisuniformschnyder2}
Let $W$ be a random walk in two dimensions started at the origin with iid steps drawn from the distribution 
\[
    \PP((1, -1)) = \frac{1}{2}, \quad \PP((-k, 1)) = \frac{1}{2^{k+2}} \quad\text{ for all }k \ge 0.
\]
Then the walk $W_n$, obtained from $W$ by conditioning the walk to end at time $2n$ at the origin and to stay in 
\[
    \left\{(x,y)\in\mathbb Z^2 : x \ge 0, y \ge -1\right\}
\]
from time $0$ to $2n$, is uniform among all possible walks of length $2n$ with these constraints. In particular, the corresponding random Schnyder wood permutation $\sigma_{n}=(\sigma_{n}^g,\sigma_{n}^r)$ of size $n$ obtained through the commutative diagram in \cref{fig-comm-diagram} is uniform.
\end{proposition}
\begin{proof}
The probability that $W_n$ is equal to any specific walk with the above constraints is exactly equal to  $\left(\frac{1}{2}\right)^n \cdot \left(\frac{1}{4}\right)^n \cdot \frac{1}{2^n}$, since we have $n$ steps of $(1, -1)$, $n$ steps of $(-k, 1)$ for some $k$, and the sum of the $-k$s must be $-n$ since the walk starts and ends at the origin. Since $\left(\frac{1}{2}\right)^n \cdot \left(\frac{1}{4}\right)^n \cdot \frac{1}{2^n}$ is independent of our choice of the specific walk, all walks do indeed occur with equal probability, as desired.

The latter claim follows immediately from the fact that the mappings in \cref{fig-comm-diagram} used to obtain $\sigma_{n}=(\sigma_{n}^g,\sigma_{n}^r)$ from $W_n$ are bijections.
\end{proof}

Let now $W_n$ be as in the statement of \cref{randomwalkisuniformschnyder2}. Our strategy for studying a uniform Schnyder wood permutation of size $n$ as $n \to \infty$ is thus to study the coalescent-walk processes $Z^g_n = \WC^g(W_n)$ and $Z^r_n = \WC^r(W_n)$, since they nicely encode the patterns of uniform Schnyder wood permutations (as shown in \cref{mainresult1}). Our first step is thus to study the limiting behavior of $Z^g_n$ and $Z^r_n$ by taking $n \to \infty$, scaling horizontally by $n$, and scaling vertically by $\sqrt{n}$. We will first study this limiting object when the driving random walk is unconditioned, as in \cref{unconditionedconvergence}, and then when it is conditioned, as in \cref{jointproductresult}. The reason is that the unconditioned case is simpler to analyze, and then information about the conditioned case can be extracted from the unconditioned one via absolute continuity arguments.

\begin{definition}\label{defnzunc}
Consider the (infinite) two-dimensional random walk $W^{*} = (X^{*}, Y^{*})$ started at the origin with iid steps drawn from the distribution 
\[
    \PP((1, -1)) = \frac{1}{2}, \quad \PP((-k, 1)) = \frac{1}{2^{k+2}} \quad\text{for all }k \ge 0.
\]
For any integer $n \ge 1$, we define the \vocab{unconditioned green coalescent-walk process} $Z^{*, g}_n$ and \vocab{unconditioned red coalescent-walk process} $Z^{*, r}_n$ as follows. Restrict $W^{*} = (X^{*}, Y^{*})$ to the interval $[0, 2n]$, and let $Z^{*, g}_n = \WC^g((W^{*})')$ and $Z^{*, r}_n = \WC^r(W^{*})$, where we recall that $(W^{*})'$ denotes the time reversal of $W^{*}$. 

Finally, extend all sample paths of $Z^{*, g}_n$ and $Z^{*, r}_n$ by setting them equal to zero to the left of their starting point. That is, for any starting point $j$, let $(Z^{*, g}_n)^{(j)}_t = 0$ for all integers $t < j$ and similar for $Z^{*, r}_n$.
\end{definition}

By a direct calculation, we see that the the increments $(X^{*}_i-X^{*}_{i-1}, Y^{*}_i-Y^{*}_{i-1})$ of the walk satisfy 
\begin{equation*}
    \EE[X^{*}_i-X^{*}_{i-1}] = 0,\quad \EE[(X^{*}_i-X^{*}_{i-1})^2] = 2, \quad \EE[Y^{*}_i-Y^{*}_{i-1}] = 0, \quad \EE[(Y^{*}_i-Y^{*}_{i-1})^2] = 1,
\end{equation*}
and 
\begin{equation*}
    \EE[(X^{*}_i-X^{*}_{i-1})(Y^{*}_i-Y^{*}_{i-1})] = -1,
\end{equation*}
so we may rescale in space and time and define the rescaled walks $\tilde{X}^{*}_n(t)$ and $\tilde{Y}^{*}_n(t)$, which are both functions from $[0, 1]$ to $\RR$, via 
\begin{equation}\label{xyunceqn}
    \tilde{X}^{*}_n(\cdot) = \text{LI}\left(\frac{1}{\sqrt{4n}} X^{*}(2nt)\right)_{t \in \{0, \frac{1}{2n}, \cdots, 1\}}, \quad \tilde{Y}^{*}_n (\cdot) = \text{LI}\left(\frac{1}{\sqrt{2n}} Y^{*}(2nt)\right)_{t \in \{0, \frac{1}{2n}, \cdots, 1\}},
\end{equation}
where $\text{LI}$ denotes the linear interpolation between adjacent multiples of $\frac{1}{2n}$. (In this paper, we will use tildes to represent rescaled and interpolated versions of our walks and processes.) Then, by Donsker’s theorem, $(\tilde{X}^{*}_n, \tilde{Y}^{*}_n)$ converges to a two-dimensional Brownian motion $(X^*_\rho, Y^*_\rho)$ on $[0, 1]$ with correlation $\rho = -\frac{\sqrt{2}}{2}$ and the same is true for the reversal $((\tilde{X}^{*}_n)', (\tilde{Y}^{*}_n)')$. 

We now wish to describe the law of the rescaled coalescent-walk processes driven by $(\tilde{X}^{*}_n, \tilde{Y}^{*}_n)$, defined as follows. For any starting point $j$ of the coalescent-walk process $Z^{*, \circ}$, and for $\circ\in \{g,r\}$, we define
\begin{equation}\label{zuncgeqn}
    \left(\tilde{Z}^{*, \circ}_n\right)^{(j/(2n))} (\cdot) = \text{LI}\left(\begin{cases} \frac{1}{\sqrt{2n}} (Z^{*, \circ}_n)^{(j)}_{2nt} & \text{if } (Z^{*, \circ}_n)^{(j)}_{2nt} \ge 0 \\ \frac{1}{\sqrt{4n}} (Z^{*, \circ}_n)^{(j)}_{2nt} & \text{if } (Z^{*, \circ}_n)^{(j)}_{2nt} < 0 \end{cases}\right)_{t \in \{0, \frac{1}{2n}, \cdots, 1\}}.
\end{equation}
The different rescaling in space above and below the $x$-axis is due to the difference in rescaling constants for $\tilde{X}$ and $\tilde{Y}$ above. $\tilde{Z}_n^{\ast, g}$ and $\tilde{Z}_n^{\ast, r}$ may be viewed as ``continuous-time coalescent-walk processes,'' and our main goal in this section is to show that they converge (along with their driving walks) to processes defined similarly to those in the skew Brownian permuton $\mu_{\rho, q}$ (recall \cref{skewbrowniansdedefn}), but driven by two-dimensional Brownian motions instead of two-dimensional Brownian excursions.

\begin{proposition}\label{unconditionedconvergence}
Let 
\begin{equation}
    \rho = -\frac{\sqrt{2}}{2}\qquad \text{and} \qquad q = \frac{1}{1 + \sqrt{2}}.
\end{equation}
Fix $u \in (0, 1)$, and let $j_n$ be the $x$-coordinate of the $\lceil nu \rceil$--th leftmost starting point of $\tilde{Z}_n^{*, r}$ for all $n$ (or $0$ if none exists). Then we have the following joint convergence\footnote{All of the spaces of continuous functions considered in this paper are implicitly endowed with the topology of uniform convergence on compact sets.} in $C([0, 1], \RR)^6$:
\begin{multline*}
\left(\tilde{X}_n^{*}\,,\, \tilde{Y}_n^{*}\,,\, (\tilde{X}_n^{*})'\,,\, (\tilde{Y}_n^{*})'\,,\, (\tilde{Z}_n^{*, g})^{(1 - j_n/(2n))}\,,\, (\tilde{Z}_n^{*\,,\, r})^{(j_n/(2n))}\right)\\
    \stackrel{d}{\longrightarrow}\quad \left(X^\ast_\rho\,,\, Y^\ast_\rho\,,\, (X^{\ast}_\rho)'\,,\, (Y^{\ast}_\rho)'\,,\, (Z^{\ast}_{\rho, q})'^{(1 - u)}\,,\, (Z^\ast_{\rho, q})^{(u)}\right),
\end{multline*}
where 
\begin{itemize}
    \item $(X^\ast_\rho, Y^\ast_\rho)$ is a Brownian motion of correlation $\rho$;  
    \item $((X^{\ast}_\rho)',(Y^{\ast}_\rho)')$ is its time-reversal; 
    \item $Z^\ast_{\rho, q}$ is the unique\footnote{\cite[Theorem 2.1]{borga2023skew} guarantees the uniqueness and existence of $Z_{\rho, q}^\ast$ and $(Z_{\rho, q})^\ast$ when the SDEs are driven by Brownian motions rather than Brownian excursions.} solution to the SDE in \cref{skewbrowniansdedefn} driven by $(X^\ast_\rho, Y^\ast_\rho)$; and 
    \item $(Z^{\ast}_{\rho, q})'$ is the unique solution to the SDE in \cref{skewbrowniansdedefn} driven by  $((X^{\ast}_\rho)', (Y^{\ast}_\rho)')$.
\end{itemize}
\end{proposition}

\medskip

The reason for considering $(\tilde{Z}_n^{*, g})^{(1 - j_n/(2n))}$ in the statement of \cref{unconditionedconvergence} is justified by the following simple fact:

\begin{lemma}\label{greenreddualstart}
Let $W$ be a random walk of length $2n$ with increments in the set $\{(-k, 1): k \ge 0\} \cup \{(1, -1)\}$. Then $\WC^r(W)$ has a starting point at time $j$ if and only if $\WC^g(W')$ has a starting point at time $2n - j$. 

Furthermore, those corresponding starting points have the same label if labeled (as in \cref{mainresult1}) in descending order from left to right for $Z^g_M$ and ascending order from left to right for $Z^r_M$.
\end{lemma}

The second part of this statement will be used later for the proof of \cref{jointproductpreresult}.

\begin{proof}
Recall the definitions of $\WC^g$ and $\WC^r$ from \cref{cwrdefn2} and \cref{cwrdefn}. The first part of the statement holds because the starting points of the red coalescent-walk processes are the ending times of the $(-k, 1)$ increments, while the starting points of the green coalescent-walk process are the starting times of the $(k, -1)$ increments. Thus for any $1 \le j \le 2n-1$, we have a starting point in $\WC^r(W)$ at $j$ if and only if $[j-1, j]$ is a $(-k, 1)$ increment in $W$, which occurs if and only if $[2n-j, 2n-j+1]$ is a $(k, -1)$ increment in the reversed walk $W'$, which occurs if and only if $n-j$ is a starting point in $\WC^g(W')$. Furthermore, we have a starting point in $\WC^r(W)$ at $0$ if and only if $[2n-1, 2n]$ is a $(-k, 1)$ increment in $W$, which occurs if and only if $[0, 1]$ is a $(k, -1)$ increment in $W'$, which occurs if and only if $2n$ is a starting point in $\WC^g(W')$. Since there is never a starting point in $\WC^r(W)$ (resp. $\WC^g(W')$) at $2n$ (resp. $0$), this shows the desired claim.

The second part then follows because $\WC^r(W)$ and $\WC^g(W')$ have the same number of starting points, and we label them in opposite order from left to right.
\end{proof}

The rest of this section is devoted to the proof of \cref{unconditionedconvergence}. We invite the reader to skip this proof at a first read (by skipping the next subsection) to first see how this result is used to complete the proof of \cref{mainresult3} in \cref{permutonsubsection}.

\subsubsection{Proofs of the preliminary probabilistic results}\label{tech-proof}

Our first goal is to show that each sample path $\left(\tilde{Z}^{*, \circ}_n\right)^{(\cdot/(2n))}$ converges to a skew Brownian motion of parameter $q$ (see \cref{foundq}). Recall that a skew Brownian motion is a Brownian motion for which the signs of the excursions away from $0$ are independently chosen to be positive with probability $q$ and negative with probability $1-q$; the reference \cite{skewbrownianmotion} provides explicit constructions and more background. We first determine the value of $q$, making use of the following fact about random walks. (Additional discussion and estimates may also be found in \cite[Section 2.1]{ngo2019limit}.)

\begin{proposition}[{\cite[Lemma D.1 and D.2]{borga2021permuton}, \cite[Proposition 11(i)]{doneyuniform}}]\label{nplimit}
Let $S_n$ be a one-dimensional random walk started at zero with iid centered steps and finite variance $\sigma^2$, and let $\tau = \inf\{n>0: S_n < 0\}$. Letting $S_\tau^i$ be iid copies of $S_\tau$, define
\[
    h(x) = \begin{cases} 1 + \sum_{j=1}^{\infty} \PP\left(S_{\tau}^1 + \cdots + S_{\tau}^j > -x\right) & \text{if }x > 0, \\ 0 & \text{otherwise}, \end{cases}
\]
and define $\tilde{h}(x)$ similarly for the walk $(-S_n)$. Then for any $x, y > 0$, we have
\[
    \PP((x + S_i)_{i \in [n]} > 0) \sim 2\frac{\EE[-S_{\tau}]}{\sigma\sqrt{2\pi}} \frac{h(x)}{\sqrt{n}}, 
\]
\[
    \PP\left(x + S_n = y, \quad (x + S_i)_{i \in [n]} > 0\right) \sim \frac{\EE[-S_\tau]}{\sigma\sqrt{2\pi} \sum_{z \in \NN} \tilde{h}(z) \PP(S_1 \le -z)} \frac{h(x)\tilde{h}(y)}{n^{3/2}},
\]
and these asymptotics hold uniformly over all $x, y$ with $\max(x, y) = o(\sqrt{n})$. 

Furthermore, suppose that for all negative $y$, we have $\PP(S_1 = y) = \alpha \gamma^{-y}$. Then $\EE[-S_\tau] = \frac{1}{1 - \gamma}$.
\end{proposition}

We will use the above result to study the return time of our Schnyder wood walks; specifically, we now obtain asymptotics on the length of the first excursion. To make this precise, we define an infinite version of \cref{defnzunc}.

\begin{definition}
Let $W^{*, r} = (X^{*, r}, Y^{*, r})$ be a random walk started at the origin with step distribution given as in \cref{defnzunc}, and let $W^{*, g} = (X^{*, g}, Y^{*, g})$ be a random walk started at the origin with the reversed step distribution (that is, $\PP((-1, 1)) = \frac{1}{2}$ and $\PP((k, -1)) = \frac{1}{2^{k+2}}$ for all $k \ge 0$).

Define the \vocab{infinite unconditioned green} (resp.\ \vocab{red}) \vocab{coalescent-walk process} $Z^{*, g}$ (resp.\ $Z^{*, r}$) to be the coalescent walk process on the interval $I = [0, \infty)$ given by $Z^{*, g} = \WC^g(W^{*, g})$ and $Z^{*, r}_n = \WC^r(W^{*, r})$, again extending all sample paths to the left by zero, except that we do not move starting points between the beginning and end of the interval as in \cref{cwrdefn2} and \cref{cwrdefn} because there is no right endpoint.
\end{definition}

We make the following remark for later use. 

\begin{remark}\label{infinitetofiniteremark}
Note that for any integer $n$, the restriction of $(W^{*, g}, Z^{*, g})$ (resp.\ $(W^{*, r}, Z^{*, r})$) to the interval $[0, 2n]$ has the same law as $((W^{*, g}_i)_{i \in [0, 2n]}, Z^{*, g}_n)$ (resp.\ $((W^{*, r}_i)_{i \in [0, 2n]}, Z^{*, r}_n)$) where $Z^{*, g}_n$ (resp.\ $Z^{*, r}_n$) is as in \cref{defnzunc}, except that the distributions of sample paths started at times $0$ and $2n$ may differ. Indeed, by the definitions of the maps $\WC^g$ and $\WC^r$ in \cref{cwrdefn2} and \cref{cwrdefn}, the value of the coalescent-walk process in any interval $[0, T]$ depends only on the value of the driving walk on $[0, T]$, except that the appearance of a starting point at $T$ depends on the increment of the walk in $[T, T+1]$ for the map $\WC^g$, and furthermore $Z_n^{*, g}$ and $Z_n^{*, r}$ may move starting points between $0$ and $T$ while $Z^{*, g}$ and $Z^{*, r}$ do not. These differences will not matter for our purposes, since we will always condition on sample paths passing through particular points in the subsequent asymptotics.
\end{remark}

We now establish a series of preliminary asymptotic estimates to later prove in \cref{foundq} that $\left(\tilde{Z}^{*, \circ}_n\right)^{(\cdot/(2n))}$ converges to a skew Brownian motion.

\begin{lemma}\label{asympreturntimes}
Fix an integer $u \ge 0$. Condition the driving walk $W^{*, g}$ on the interval $[0, u]$ so that some sample path of $Z^{\ast, g}$, say $(Z^{\ast, g})^{(j)}$ with $j \le u$, passes through the point $(u, 0)$, i.e.\ $(Z^{\ast, g})^{(j)}_u = 0$.. Define the random walk $G_i = (Z^{\ast, g})^{(j)}_{u + i}$ for all $i\geq 0$, and let $\tau_G$ be the first return time to $0$ for $G_i$. Define $R_i$ and $\tau_R$ analogously but using the coalescent-walk process $Z^{\ast, r}$ instead. Then, as $k \to \infty$, we have the asymptotics
\[
    \PP(\tau_G = 2k | G_1 > 0)\sim \frac{1}{\sqrt{4\pi}} k^{-3/2},\qquad \PP(\tau_G = k | G_1 < 0)\sim \frac{1}{\sqrt{\pi}} k^{-3/2},
\]
\[
    \PP(\tau_R = 2k | R_1 > 0) \sim \frac{1}{\sqrt{4\pi}} k^{-3/2},\qquad \PP(\tau_R = k | R_1 < 0) \sim \frac{1}{\sqrt{\pi}} k^{-3/2}.
\]
\end{lemma}

This lemma is phrased in terms of an arbitrary $u \ge 0$ rather than only considering starting points of the coalescent-walk processes, since we will want to apply the lemma to each excursion of the sample paths of $Z^{*, g}$ and $Z^{*, r}$ rather than just the initial one. Furthermore, by the discussion in \cref{infinitetofiniteremark}, these asymptotics also apply to the finite restrictions $Z^{*, g}_n$ and $Z^{*, r}_n$, though we must take $n \to \infty$ to get quantitative bounds as $k \to \infty$. This is indeed what we do in the proof of \cref{foundq}. 

\begin{proof}[Proof of \cref{asympreturntimes}]
For both the green and red processes, conditioning on starting positive, i.e.\ $G_1 > 0$ or $R_1 > 0$, means that the first increment is $+1$ and the subsequent excursion will evolve as a simple random walk. To avoid parity issues, only consider the walk at even times by combining groups of two adjacent steps into one (adding the increments together) and also halving the sizes of the increments (which does not change the distribution of the return time). So now our walk takes steps $+1, 0, -1$ with probability $1/4, 1/2, 1/4$ respectively; call this new walk $T_n$. 

Since we condition on the simple random walk starting positive and want asymptotics on the return time, we must have $T_1 = 1$ (since we can't have $T_1 = -1$, and if $T_1 = 0$ then our return time $\tau_G$ or $\tau_R$ is just $2$), which occurs with probability $\frac{1}{2}$. We then find for any $k > 2$ that 
\begin{align*}
    \PP(\tau_G = 2k | G_1 > 0) = \PP(\tau_R = 2k | R_1 > 0) &= \frac{1}{2} \cdot \PP(T_{k-1} = 1, T_k = 0, T_{[1, k-1]} > 0 | T_1 = 1) \\
    &= \frac{1}{8} \cdot \PP(1 + T_{k-2} = 1, (1 + T_i)_{i \in [k-2]} > 0)
\end{align*}
by translating time back by $1$ unit and because the final increment $T_k - T_{k-1}$ is $-1$ with probability $\frac{1}{4}$. Set $\tau = \inf\{n > 0: T_n < 0\}$. Applying \cref{nplimit}, we have $\EE[-T_\tau] = 1$, $\sigma = \frac{1}{\sqrt{2}}$, and $h(1) = \tilde{h}(1) = 1$ (since $T_\tau$ is always $-1$ and $-T$ and $T$ are identically distributed), so this simplifies to 
\begin{align*}
    \PP(\tau = 2k | G_1 > 0) = \PP(\tau = 2k | R_1 > 0) &\sim {\frac{1}{8}} \cdot \frac{1}{{\sqrt{\pi}} \sum_{z{\in \NN}} \tilde{h}(z) \PP(T_1 \le -z)} \frac{1}{(k-2)^{3/2}} \\
    &{\sim\frac{1}{8} \cdot \frac{1}{{\sqrt{\pi}} \cdot \tilde{h}(1) \PP(T_1 \le -1)} \frac{1}{(k-2)^{3/2}}} \\
    &\sim \frac{1}{\sqrt{4\pi}} k^{-3/2}. 
\end{align*}

For the other case, we condition $G$ and $R$ to lie below the $x$-axis during the excursion. We introduce some additional notation to keep equations more concise. Let $A$ (resp.\ $B$) be a one-dimensional random walk with iid increments $j$ (resp.\ $-j$) with probability $2^{-j-2}$ for all integer $j \ge -1$. Then {(because $R$ follows the increments of $-X^{\ast, r}$ below the $x$-axis)} $R$ and $A$ follow the same law except at $R$'s last step before hitting $0$, and {(because $G$ follows the increments of $-X^{\ast, g}$ below the $x$-axis)} $G$ and $B$ follow the same law, except that $G_1$ has the law of $B_1 - 1$ conditioned to be negative (due to the definition of the green coalescent-walk process on the $x$-axis).

We now begin our calculations. For the red walk, {the first increment is always $-1$ when conditioned to be negative, so} we may write
\[
    \PP(\tau_R = k | R_1 < 0) = \sum_{m = 1}^{\infty} 2^{-m-1} \PP\left((A_i)_{i\in[2, k-1]} < 0, A_{k-1} = -m \,\middle|\, A_1 = -1\right),
\]
where we've done casework on the last position of the walk before hitting the $x$-axis (the $2^{-m-1}$ term is the probability of making a sufficiently large upward jump from $-m$ so that the walk hits the $x$-axis at time $k$). Since the law of $-A_n$ is identical to the law of $B_n$, we may rewrite this {(first applying a time-translation by $1$ unit, then applying \cref{nplimit}) as} 
\begin{align*}
    \PP(\tau_R = k | R_1 < 0) &= \sum_{m = 1}^{\infty} 2^{-m-1} \PP\left((1 + B_i)_{i \in [1, k-2]} > 0, 1 + B_{k-2} = m\right) \\ 
    &\sim \sum_{m=1}^{\infty} 2^{-m-1} \frac{1/(1-1/2)}{\sqrt{2} \sqrt{2\pi} \sum_{z \in \NN} h_A(z) \PP(B_1 \le -z)} h_B(1) h_A(m) (k-2)^{-3/2},
\end{align*}
where this asymptotic holds term-by-term by \cref{nplimit} but also for the whole sum because of the $2^{-m-1}$ weighting factor. (More precisely, these asymptotics hold uniformly over all $m = o(\sqrt{k})$, so the $2^{-m-1}$ prefactor is enough for the remaining tail to decay faster than the main contribution.) Using that $h_A(z) = z$ because the $A$ walk is always at height $-1$ when it first becomes negative, we thus find that
\begin{align*}
    \PP(\tau_R = k | R_1 < 0) &\sim k^{-3/2} \sum_{m=1}^{\infty} \frac{2^{-m}}{\sqrt{4\pi} \sum_{z{ \in \NN}} z 2^{-z - 1}} \cdot 1 \cdot m \\
    &= k^{-3/2} \frac{1}{\sqrt{4\pi}} \sum_{m=1}^{\infty} m 2^{-m} \\
    &= \frac{1}{\sqrt{\pi}} k^{-3/2}. 
\end{align*}
Similarly, we do casework for the green walk based on the first step, {this time using that $h_B(1) = 1$:}
\begin{align*}
    \PP\left(\tau_G = k \,\middle|\, G_1 < 0\right) &= \sum_{m=1}^{\infty} 2^{-m} \,\PP\left((B_i)_{i \in [2, k-1]} < 0, B_{k-1} = -1, B_k = 0 \,\middle|\, B_1 = -m\right) \\
    &= \sum_{m=1}^{\infty} 2^{-m} \,\PP\left((A_i)_{i \in [2, k-1]} > 0, A_{k-1} = 1, A_k = 0 \,\middle|\, A_1 = m\right) \\
    &= \sum_{m=1}^{\infty} 2^{-m - 1} \,\PP\left((m + A_i)_{i \in [k-2]} > 0, m + A_{k-2} = 1\right) \\
    &\sim \sum_{m=1}^{\infty} 2^{-m - 1} \frac{1}{\sqrt{2} \sqrt{2\pi} \sum_{z{\in \NN}} h_B(z) \PP(A_1 \le -z)} \frac{h_A(m) h_B(1)}{(k-2)^{3/2}} \\
    &\sim k^{-3/2} \sum_{m=1}^{\infty} 2^{-m - 1} \frac{1}{\sqrt{2} \sqrt{2\pi} h_B({1}) \cdot \frac{1}{2}} m \cdot 1 \\
    &= k^{-3/2} \frac{1}{\sqrt{\pi}} \sum_{m=1}^{\infty} m 2^{-m-1} \\
    &= \frac{1}{\sqrt{\pi}} k^{-3/2},
\end{align*}
showing the desired asymptotics and completing the proof.
\end{proof}

\begin{corollary}\label{returnasymptotics}
Adopting the notation in \cref{asympreturntimes}, let $\tau$ be the first return time to $0$ for either $G_i$ or $R_i$. Then we have the asymptotics
\[
    \PP(\tau > n \,|\, \text{excursion is positive}) \sim \sqrt{\frac{2}{\pi}} n^{-1/2},
\]
\[
    \PP(\tau > n \,|\, \text{excursion is negative}) \sim \sqrt{\frac{4}{\pi}} n^{-1/2}.
\]
In particular, $\PP(\tau > n) \sim \frac{1}{\sqrt{2\pi}}(1 + \sqrt{2})n^{-1/2}$, {unless it is the first excursion of a sample path for $Z^{\ast, g}$, in which case $\PP(\tau > n) \sim \sqrt{\frac{4}{\pi}}n^{-1/2}$.}
\end{corollary}
\begin{proof}
This follows from summing up the asymptotic values from \cref{asympreturntimes} and comparing to an integral. For the negative excursions, we use that
\[
    \sum_{k > n} k^{-3/2} \sim \int_n^{\infty} x^{-3/2} \, dx = 2n^{-1/2},
\]
so that indeed $\PP(\tau > n \,|\, \text{excursion is negative}) \sim \frac{1}{\sqrt{\pi}} 2n^{-1/2}$, as desired. Similarly for the positive excursions, we find that $\PP(\tau > 2m \,|\, \text{excursion is positive}) \sim \frac{1}{\sqrt{4\pi}} 2m^{-1/2}$, and then changing variables $m = \frac{n}{2}$ yields the result. 

Finally, the last claim follows from the fact that each sample path always has probability $\frac{1}{2}$ of starting a positive or negative excursion from zero (because it follows the next increment of $Y^{\ast, g}$ or $Y^{\ast, r}$, which are simple random walks), except that sample paths of $Z^{\ast, g}$ always begin with a negative increment.
\end{proof}

We will use the above result to estimate the asymptotic probability of returning to the origin after some large number of steps by thinking of $\tau$ as the step distribution of a (positive-integer valued) random walk.

\begin{corollary}\label{asymptoticsforzeros}
Again use the notation of \cref{asympreturntimes}. We have the asymptotics
\[
    \PP(G_n = 0)\sim \frac{\sqrt{2}}{(1 + \sqrt{2})\sqrt{\pi}}n^{-1/2},\qquad \PP(R_n = 0) \sim \frac{\sqrt{2}}{(1 + \sqrt{2})\sqrt{\pi}}n^{-1/2}
\]
as $n \to \infty$. Therefore, for any sample path of $Z^{*, g}$ or $Z^{*, r}$ beginning at a starting point $j$, we have the asymptotics
\[
    \PP\left((Z^{*, g})^{(j)}_{j + n} = 0\right)\sim \frac{\sqrt{2}}{(1 + \sqrt{2})\sqrt{\pi}}n^{-1/2}, \qquad\PP\left((Z^{*, r})^{(j)}_{j + n} = 0\right) \sim \frac{\sqrt{2}}{(1 + \sqrt{2})\sqrt{\pi}}n^{-1/2}. 
\] 
\end{corollary}

\begin{proof}
We may apply\footnote{The right-hand side in \cite[Theorem B]{Doney_97} should actually read $\frac{1}{\Gamma(\alpha)\Gamma(1-\alpha)}$ instead of how it reads in the paper.} \cite[Theorem B]{Doney_97}. The necessary assumption is satisfied with $\alpha = 1/2$ thanks to \cref{returnasymptotics}. Note that even though the distribution of the first return time is different than subsequent ones for each path of the green process, this does not change the asymptotics because we've shown that the mass function for $\tau$ {still} decays faster than $\frac{1}{\sqrt{n}}$. 
\end{proof}

With these calculations verified, we will next prove the convergence of the rescaled sample paths $\left(\tilde{Z}^{*, \circ}_n\right)^{(\cdot/(2n))}$ introduced in \cref{zuncgeqn} to skew Brownian motions and then subsequently prove joint convergence of these sample paths along with their driving walks. 

\begin{proposition}\label{foundq}
Fix $u \in (0, 1)$. For all $n$, let $j_n$ be the $x$-coordinate of the $\lceil nu \rceil$--th leftmost starting point of $\tilde{Z}_n^{*, r}$ (or $0$ if none exists). Then we have the following convergence in distribution in the space $C([0, 1], \RR)$:
\[
    (\tilde{Z}_n^{*, r})^{(j_n/(2n))} \stackrel{d}{\longrightarrow} B_q^{(u)},
\]
where $B_q^{(u)}(t)$ is zero for $t \in [0, u]$, and for $t \in [u, 1]$ it is a skew Brownian motion of parameter $q = \frac{1}{1 + \sqrt{2}}$ started at zero at time $u$. The analogous statement holds for $(\tilde{Z}_n^{*, g})^{(1-j_n/(2n))}$ as well (with $1-u$ in place of $u$). 
\end{proposition}

\begin{proof}
We first prove the case when $j_n = 0$ for all $n$ (meaning that the starting points are at the very beginning of the interval for our coalescent-walk processes), and at the end we explain how to {adapt the proof to} the general case. Note that the event that $j_n = 0$ actually occurs with probability tending to zero, but we present the proof this way to simplify notation and highlight the important asymptotics.

First, we show convergence of one-dimensional marginal distributions -- we follow the strategy from \cite{ngo2019limit}, and this argument works identically for either the green or red process. 

For notational convenience, we will write $(\tilde{Z}_n^{*, g})^{(0)}(t)$ or $(\tilde{Z}_n^{*, r})^{(0)}(t)$ as $\tilde{Z}(t)$ for the remainder of this proof, and we will similarly abbreviate $(Z_n^{*, g})^{(0)}_k$ or $(Z_n^{*, r})^{(0)}_k$ as $Z_k$. For a function $f$ and event $A$, we also use the notation $\EE[f; A]$ as shorthand for $\EE[f \cdot 1_A]$.  

Let $\phi$ be a bounded Lipschitz function. We compute $\EE[\phi(\tilde{Z}(t)); \tilde{Z} > 0]$ and $\EE[\phi(\tilde{Z}(t)); \tilde{Z} < 0]$ separately, showing that both converge to the corresponding values for a skew Brownian motion of parameter $q=\frac{1}{1+\sqrt{2}}$. We'll do the first term first, breaking up the sum into cases based on the last time $k$ before time $\lfloor 2nt \rfloor$ that we hit zero. Recalling that while $Z$ is positive, its increments are independent and have law identical to $Y^{*}_1 - Y^{*}_0$, let $S$ denote a random walk started at zero at time zero with iid increments of that form. Then we have
\begin{align*}
    \EE[\phi(\tilde{Z}(t)); &\,\tilde{Z} > 0]\\ 
    &= \EE\left[\phi\left(\frac{Z_{\lfloor 2nt \rfloor}}{\sqrt{2n}}\right); Z_{\lfloor 2nt \rfloor} > 0\right] + O(n^{-1/2}) \\
    &= \sum_{k=0}^{\lfloor 2nt \rfloor} \PP(Z_k = 0) \cdot \PP\left((S_i)_{i \in [1, \lfloor 2nt \rfloor - k]} > 0\right)  \cdot\EE\left[\phi\left(\frac{S_{\lfloor 2nt \rfloor - k}}{\sqrt{2n}}\right) \middle| (S_i)_{i \in [1, \lfloor 2nt \rfloor - k]} > 0\right] + O(n^{-1/2}),
\end{align*}
where using $\lfloor 2nt \rfloor$ instead of $2nt$ only gives us the error term $O(n^{-1/2})$ because each step of the walk has bounded expectation and $\phi$ is Lipschitz. {By \cref{asymptoticsforzeros} and \cref{returnasymptotics}, respectively} (and using \cref{infinitetofiniteremark} to apply our asymptotics for $Z^{\ast, g}$ to $Z_n^{\ast, g}$), we know that $\PP(Z_k = 0) = O(k^{-1/2})$ and $\PP(S_{[1, \lfloor {2nt} \rfloor - k]} > 0) = O((\lfloor {2nt} - k \rfloor)^{-1/2})$ and both terms are bounded by $1$, and $\phi$ is bounded. Thus each term of the series is $O(n^{-1/2})$; therefore {we may neglect} the first and last $n^{1/4}$ terms {and find that} $\EE[\phi(\tilde{Z}(t)); \tilde{Z} > 0]$ is equal to
\begin{align}\label{eq:ewifvweivbf}
    \sum_{k=n^{1/4}}^{\lfloor 2nt \rfloor - n^{1/4}} \PP(Z_k = 0) \cdot \PP\left((S_i)_{i \in[1, \lfloor 2nt \rfloor - k]} > 0\right) \cdot \EE\left[\phi\left(\frac{S_{\lfloor 2nt \rfloor - k}}{\sqrt{2n}}\right) \middle| (S_i)_{i \in [1, \lfloor 2nt \rfloor - k]} > 0\right] + O(n^{-1/4}).
\end{align}
Now, by the main result of \cite{bolthausen}, the endpoint of a walk conditioned to stay positive converges to the endpoint of a Brownian meander -- that is, for any $m \le \lfloor 2nt \rfloor$ with $m \to \infty$, we have (because the increments of $S_m$ each have variance $1$)
\[
    \EE\left[\phi\left(\frac{S_m}{\sqrt{2n}}\right) \middle| (S_i)_{i \in [1, m]} > 0\right] \sim \int_0^{\infty} \phi\left(\sqrt{\frac{m}{2n}}u\right) u e^{-u^2/2} \, du = \frac{2n}{m} \int_0^{\infty} \phi(x) x e^{-nx^2/m} \, dx.
\]
Thus, thanks to the above equation and {again} \cref{returnasymptotics,asymptoticsforzeros}, for any $\eps > 0$, we can choose $n$ large enough so that all three of the terms in the sum are multiplicatively within $\eps$ of their asymptotic value for all $k$ in the summation; this means, as $n \to \infty$, the sum in \cref{eq:ewifvweivbf} will converge to (note the extra factor of $\frac{1}{2}$ compared to \cref{returnasymptotics}, since we want the probability of staying positive rather than the probability of avoiding zero given that we stay positive)
\begin{align*}
        \sum_{k=n^{1/4}}^{\lfloor 2nt \rfloor - n^{1/4}}&\frac{\sqrt{2}}{(1 + \sqrt{2})\sqrt{\pi}} k^{-1/2} \frac{1}{2}\sqrt{\frac{2}{\pi}} (\lfloor 2nt \rfloor - k)^{-1/2} \frac{2n}{\lfloor 2nt \rfloor - k} \int_0^{\infty} \phi(x) x e^{-nx^2/(\lfloor 2nt \rfloor - k)} \, dx\\
    &= \frac{2n}{\pi(1 + \sqrt{2})}\sum_{k=n^{1/4}}^{\lfloor 2nt \rfloor - n^{1/4}} k^{-1/2} (\lfloor 2nt \rfloor - k)^{-3/2} \int_0^{\infty} \phi(x) x e^{-nx^2/(\lfloor 2nt \rfloor - k)} \, dx\\
    &= \frac{1}{\pi(1 + \sqrt{2})} \cdot \frac{1}{2n}\sum_{k=n^{1/4}}^{\lfloor 2nt \rfloor - n^{1/4}} \left(\frac{k}{2n}\right)^{-1/2}\left (\frac{\lfloor 2nt \rfloor - k}{2n}\right)^{-3/2} \int_0^{\infty} \phi(x) x e^{-x^2/2((\lfloor 2nt \rfloor - k)/(2n))} \, dx.
\end{align*}
In particular, this sum is a Riemann sum approximation for the integral expression
\[
    \frac{1}{\pi(1 + \sqrt{2})} \int_0^{t} s^{-1/2} (t-s)^{-3/2} \int_0^{\infty} \phi(x) x e^{-x^2/(2(t-s))} \, dx \, ds.
\]
So, as $n \to \infty$, our expectation will approach this integral (since the integrand is continuous in $s \in (0, t)$). We can now evaluate this integral. First, substituting $y = s/t$, the integral is equal to 
\begin{multline*}
    \frac{1}{\pi(1 + \sqrt{2})} \cdot \frac{1}{t} \int_0^1 y^{-1/2} (1-y)^{-3/2}  \int_0^{\infty} \phi(x) x e^{-x^2/(2t(1-y))} \, dx \, dy\\
    = \frac{1}{\pi(1 + \sqrt{2})} \cdot  \frac{1}{t} \int_0^{\infty} x \phi(x) \int_0^1 y^{-1/2} (1-y)^{-3/2} e^{-x^2/(4t(1-y))} \, dy \, dx.
\end{multline*}
Now substituting $z = \frac{1}{1-y} - 1 = \frac{y}{1-y}$ yields
\begin{multline*}
    \frac{1}{\pi(1 + \sqrt{2})} \cdot \frac{1}{t} \int_0^{\infty} x \phi(x) \int_0^{\infty} \frac{1}{\sqrt{z}} e^{-x^2(z+1)/(2t)} \, dz \, dx\\
    = \frac{1}{\pi(1 + \sqrt{2})} \cdot \frac{1}{t} \int_0^{\infty} x \phi(x) e^{-x^2/2t} \int_0^{\infty} \frac{1}{\sqrt{z}} e^{-\frac{x^2}{2t}z} \, dz \, dx.
\end{multline*}
We can now use the classical fact $\int_0^{\infty} \frac{1}{\sqrt{t}} \exp\left(-\alpha t\right) dt = \sqrt{\frac{\pi}{\alpha}}$ to simplify this to
\[
    \frac{1}{\pi(1 + \sqrt{2})} \cdot \frac{1}{t} \int_0^{\infty} x \phi(x) e^{-x^2/2t} \sqrt{\frac{2\pi t}{x^2}} \, dx.
\]
Summarizing, we have thus computed the limiting expectation {(recalling that we suppressed the dependence on $n$ in our notation $\tilde{Z}$)}
\begin{equation}\label{brownianqpart1}
    \EE[\phi(\tilde{Z}(t)); \tilde{Z} > 0] \sim \frac{1}{1 + \sqrt{2}} \int_0^{\infty} \phi(x) \frac{2e^{-x^2/2t}}{\sqrt{2\pi t}} \, dx.
\end{equation}
We get a similar result for the expectation conditioned on $\tilde{Z} < 0$, with the main differences being that we integrate $x$ over $(-\infty, 0)$ and that the term $\PP((S_i)_{i \in [1, \lfloor 2nt \rfloor - k ]}) > 0$ in \eqref{eq:ewifvweivbf} becomes $\PP((S_i')_{i \in [1, \lfloor 2nt \rfloor - k ]}) < 0$, where $S'$ now has increments identically distributed to $-(X_1^{*} - X_0^{*})$. (We also rescale $Z$ by $\sqrt{4n}$ instead of $\sqrt{2n}$ so that we again get convergence to the endpoint of a Brownian meander, and this does not affect any of the other calculations.) Thus we gain a factor of $\sqrt{2}$ from the different constant in \cref{returnasymptotics}, yielding 
\begin{equation}\label{brownianqpart2}
    \EE[\phi(\tilde{Z}(t)); \tilde{Z} < 0] \sim \frac{\sqrt{2}}{1 + \sqrt{2}} \int_{-\infty}^0 \phi(x) \frac{2e^{-x^2/2t}}{\sqrt{2\pi t}} \, dx.
\end{equation}
Summing up \cref{brownianqpart1} and \cref{brownianqpart2}, we see the one-dimensional marginals of our sample path converge to the one-dimensional marginals of the skew Brownian motion of parameter $\frac{1}{1 + \sqrt{2}}$ {(see for instance \cite[p. 40]{walshskew})}, as desired.

\medskip

The next step is to show convergence of finite-dimensional distributions. But these calculations are completely analogous to \cite{ngo2019limit}'s proof (in which they work out the two-dimensional case), so we only describe the general strategy here. Suppose we have $m$ fixed times $0 \le s_1 < s_2 < \cdots < s_m \le 1$ and want the joint distribution of $\tilde{Z}$ at those times. To show convergence, we must compute the value of $\EE\left[\prod_{i=1}^m\phi_i(\tilde{Z}(s_i))\right]$ for arbitrary Lipschitz bounded functions $\phi_i$. 

We may do so by (joint) casework on whether the sample path hits zero between the times $\lfloor 2ns_i \rfloor$ and $\lfloor 2ns_{i+1} \rfloor$ for each $i$, since the increments of the path on disjoint time-intervals are independent. If the path does hit zero, then we sum over all first hitting times $\lfloor 2ns_i \rfloor \le k \le \lfloor 2ns_{i+1} \rfloor$ and use asymptotics to estimate the quantities
\[
    \PP\left(Z_k = 0 \,\,\middle|\,\,Z_{\lfloor 2ns_i \rfloor}\right) \quad \text{ and } \quad \PP\left(Z_{\lfloor 2ns_{i+1} \rfloor} = t\,\,\middle|\,\, Z_k = 0\right). 
\]    
Otherwise, we instead use asymptotics to estimate 
\[
    \PP\left(Z_{\lfloor 2ns_{i+1} \rfloor} = t \text{ and } Z \text{ does not hit zero between }\lfloor 2ns_i \rfloor \text{ and } \lfloor 2ns_{i+1} \rfloor\,\,\middle|\,\, Z_{\lfloor 2ns_i \rfloor}\right), \qquad\text{for all }t.
\]
In the same way as in the one-dimensional case, we cut off the terms in the resulting $m$-way summation where we hit zero within $n^{1/4}$ of one of the $s_i$s; this is a vanishingly small fraction of all terms, so this doesn't affect the integral Riemann sum. By doing so, we ensure that the asymptotic estimates converge, so as $n \to \infty$, all of the relevant probabilities will converge to those given by a skew Brownian motion. And since for any finite $m$ there are only finitely many cases to consider, we will indeed converge to the correct expectation given by a skew Brownian motion of parameter $\frac{1}{1+\sqrt{2}}$. 

\medskip

Finally, we prove tightness of the sample paths $\tilde{Z}(t)$ (recall that this denotes either $(\tilde{Z}_n^{*, g})^{(0)}(t)$ or $(\tilde{Z}_n^{*, r})^{(0)}(t)$). Thanks to \cite[Theorem 7.3]{billing}, since $\tilde{Z}({0})=0$ for all $n$, it suffices to check that for any $\eps>0$ and $\gamma>0$, there is some $\delta>0$ such that $\PP[\omega_n(\delta) \ge \eps] \le \gamma$ for all sufficiently large $n$, where $\omega_n$ denotes the modulus of continuity of $\tilde{Z}(t)$. But we can write
\[
    \omega_n(\delta) = \sup_{|s-t|\leq \delta}|\tilde{Z}(t)-\tilde{Z}(s)| \le \frac{1}{\sqrt{2n}} \sup_{\substack{i,j \in [n],\\ |i-j| \le 2n\delta}} |A(i) - A(j)| + \frac{1}{\sqrt{4n}} \sup_{\substack{i,j \in [n],\\ |i-j| \le 2n\delta}} |B(i) - B(j)| + \frac{1}{\sqrt{4n}},
\]
{where $A$ and $B$ are the one-dimensional random walks from the proof of \cref{asympreturntimes}} and the last term appears because the first step of each (unscaled) sample path of $Z_n^{{\ast}, g}$ has an extra increment of $-1$. Taking $n \to \infty$, then $\delta \to 0$, the probability each of these terms is larger than $\frac{\eps}{3}$ goes to zero. Indeed, this is clear for the last term, and for the first two terms we use that the increments of ${A}$ and ${B}$ are iid{, centered,} and of finite variance, so {because the modulus of continuity of a Brownian motion tends to zero as $\delta \to 0$ by \cite[Lemma 1]{bmmodulus}, for sufficiently large $n$ these two corresponding terms will do so as well.}. This completes the proof for the case when $j_n = 0$.

\medskip

Now, we turn to the case of general starting points, focusing on the sample path $(\tilde{Z}_n^{*, r})^{(j_n/(2n))}$ (the proof for $(\tilde{Z}_n^{*, g})^{(1-j_n/(2n))}$ is similar). We again just need to show convergence of finite-dimensional marginal distributions and tightness. Recall from \cref{cwrdefn} that the starting points of $Z_n^{\ast, r}$ are the ending times of the $+1$ increments of $Y^{\ast, r}$, which is an unconditioned simple random walk, possibly along with a starting point at time $0$. Thus the $\lceil nu\rceil$--th leftmost starting point $j_n$ of $Z_n^{\ast, r}$ is the sum of either $\lceil nu\rceil - 1$ or $\lceil nu\rceil$ independent Geom($1/2$) random variables, and by the law of large numbers this means $\frac{j_n}{2n}$ converges in probability to $u$ and furthermore that the $\lceil nu \rceil$--th  leftmost starting point of $Z_n^{\ast, r}$ does indeed exist with high probability.  Unpacking the definition, we have for any $\delta, \eps > 0$ that with probability at least $1 - \eps$, $j_n \in [2n(u - \delta), 2n(u + \delta)]$ for all sufficiently large $n$. 

For one-dimensional marginal distributions, the distribution of $\tilde{Z}(t)$ at any time $t < u$ will then converge to $0$ (because by picking $\delta$ small enough we have with high probability that the starting point satisfies $j_n > 2nt$ when $t < u$), and so will the distribution at $t = u$. On the other hand, writing $\tilde{Z}^{(j_n/(2n))}(t)$ for $(\tilde{Z}_n^{*, {r}})^{(j_n/(2n))}(t)$, we may write for any $t > u$
\[
    \EE[\phi(\tilde{Z}^{(j_n/(2n))}(t))] = \EE\left[\phi(\tilde{Z}^{(j_n/(2n))}(t)); \left|\frac{j_n}{2n} - u\right| \le \delta\right] + \EE\left[\phi(\tilde{Z}^{(j_n/(2n))}(t)); \left|\frac{j_n}{2n} - u\right| > \delta\right].
\]
Because $\phi$ is bounded, the latter term goes to zero as $\delta \to 0$. Moreover, taking $\delta$ going to zero fast enough when $n$ goes to infinity, one can show (mimicking the computations above) that the former term converges to $\EE[\phi(\tilde{B}_{t-u})]$ where $\tilde{B}$ is a skew Brownian motion. This is possible because\footnote{Here, we also use the fact that the location of the starting point $j_n$ depends only on the increments of the walk $W$ up to $j_n$ and the final increment of $W$, and thus even when conditioning on the value of $j_n$, no conditioning is needed on the future of the walk except at the last unit interval, which does not affect the scaling limit.} the distribution of a skew Brownian motion at time $t - u + x$ converges to that at time $t - u$ as $x \to 0$.  Other than accounting for this subtlety, the arguments for finite-dimensional marginals and tightness carry through in exactly the same way as the $j_n = 0$ case, completing the proof.
\end{proof}

Next, we show that the convergence in \cref{foundq} is joint between the green and red coalescent-walk processes and also joint with that of the driving random walks, proving \cref{unconditionedconvergence}.

\begin{proof}[Proof of \cref{unconditionedconvergence}]
The proof is similar to that of Theorem 4.5 in \cite{borga2021permuton}, and we note the key points here. We have $(\tilde{X}_n^{*}, \tilde{Y}_n^{*}) \stackrel{d}{\to} (X^\ast_\rho, Y^\ast_\rho)$ by Donsker's theorem, hence the same for the time-reversal. Additionally, we have $(\tilde{Z}_n^{*, r})^{(j_n/(2n))} \stackrel{d}{\to} (Z^\ast_{\rho, q})^{(u)}$ by \cref{foundq} above, since the sample paths of $Z^\ast_{\rho, q}$ are skew Brownian motions of parameter $q$; see \cite[Remark 2.6]{borga2023skew}. Similarly, we have the convergence  $(\tilde{Z}_n^{*, g})^{(1 - j_n/(2n))}\stackrel{d}{\to} (Z^\ast_{\rho, q})'^{(1-u)}$. Thus we just need to show that this convergence is joint in the components, which we can do by showing that all joint subsequential limits are equal (because by Prohorov's theorem we have a tight family of measures).

We know that any subsequential limit must be of the form $(\overline{X}_\rho, \overline{Y}_\rho, \overline{X}'_\rho, \overline{Y}'_\rho, \overline{Z}'_q, \overline{Z}_q)$, consisting of a two-dimensional Brownian motion of correlation $\rho$, its time reversal, and two $q$-skew Brownian motions started at $u$ and $1-u$, respectively. Furthermore, defining the sigma-algebra $(\mc{F}_t)_t$ to be the completion of $\sigma(\overline{X}_\rho, \overline{Y}_\rho, \overline{Z}_q: s \le t)$ by null sets, $(\overline{X}_\rho, \overline{Y}_\rho)$ is a $(\mc{F}_t)_t$-adapted Brownian motion. Now, on any open interval on which $\overline{Z}$ is bounded away from zero, we satisfy the equation \cref{skewbrowniansdedefn}, since for each finite $n$ the processes $\tilde{Z}_n$ and $\tilde{Y}_n$ differ by exactly a constant on that interval, meaning the stochastic integrals in the limit also behave as the SDE specifies, and the local time term in that equation comes from the local time description of a skew Brownian motion. The conclusion is that the joint limit $(\overline{X}_\rho, \overline{Y}_\rho, \overline{Z}_q)$ is unique.

This same logic but now considering $(\overline{X}'_\rho, \overline{Y}'_\rho, \overline{Z}'_q)$ shows that the backwards walk also satisfies the appropriate SDE to prove uniqueness. The result follows because $\overline{X}_\rho'$ is determined by $\overline{X}_\rho$ (by time-reversal), so the joint limit of all six functions is also unique. 
\end{proof}

\subsection{Permuton convergence for 3-dimensional Schnyder wood permutations}\label{permutonsubsection}

We can now bring everything together, showing that there exists a three-dimensional permuton limit for uniform Schnyder wood permutations, as stated in \cref{mainresult3}. The proof of this result will use the characterization of permuton convergence in terms of convergence of patterns, proved in \cref{tfaemain}. 

Recall the statement of \cref{randomwalkisuniformschnyder2}, which asserts that if $W_n$ is the conditioned random walk described in the proposition statement, then the corresponding random Schnyder wood permutation $\sigma_{n}=(\sigma_{n}^g,\sigma_{n}^r)$ of size $n$ obtained through the commutative diagram in  \cref{fig-comm-diagram} is uniform. In the previous subsection, we worked with the unconditioned version of this random walk. Now we consider the conditioned one. 

To do so, we first establish the following lemma, recalling from \cref{cwrdefn} that the starting time of the path labeled $i$ in the coalescent-walk process $Z^r_{n}$ corresponds to the ending time of the $i$--th $(-k, 1)$ increment of $W_n$.

\begin{lemma}\label{schnyderlocationtail}
For all $n$, let $W_{n} = (X_{n}, Y_{n})$ be as in \cref{randomwalkisuniformschnyder2} with associated coalescent-walk process $Z^r_{n}$.
For each $1 \le i \le n$, denote by $p_{n, i}$ the $x$-coordinate of the $i$--th leftmost starting point in $Z^r_n$. Then
\[  
    \PP\left(\sup_{i \in [n]} (2i - p_{n, i}) > n^{3/4}\right) = o(e^{-n}).
\]
\end{lemma}

Notice that we always have $i-1 \le p_{n, i} < 2i$ for all $i$. Indeed, by definition of $W_n$ in \cref{randomwalkisuniformschnyder2},
$Y_n$ has all increments $+1$ or $-1$, starts and ends at zero, and $Y_n(t)\geq -1$ for all $t$. Hence $Y_n$ may be interpreted as a Dyck path whose initial $+1$ increment has been moved to the end. Our claim then follows by recalling that the starting points in $Z^r_n$ occur at the right endpoints of the $+1$ increments of $Y_n$ (except that the starting point corresponding to that moved increment is still at the beginning of the interval).

\cref{schnyderlocationtail} essentially states that even though the starting points of the Schnyder wood coalescent-walk processes are not deterministically spaced out across the interval $[0, 2n]$, with high probability their positions will all be within $n^{3/4}$ of those deterministic locations.

\begin{proof}
We first prove an intermediate result, which is as follows: sample a uniform Dyck path of size $2n$, and let $T_i$ be the left endpoint of the $i$--th up step. Then we claim that
\begin{equation}\label{eq:partial-goal}
	\lim_{n \to \infty} \PP\left(\sup_{i \in [n]} (2i - T_i) > n^{3/4}\right) = 0.
\end{equation} 
Indeed, we can sample a uniform Dyck path by first letting $D_1, D_2, \cdots$ be iid $\text{Geom}(1/2) - 1$ random variables and creating a walk with an up step, then $D_1$ down steps, then an up step, then $D_2$ down steps, and so on, up to $D_n$ down steps, and then conditioning on the result being a Dyck path, meaning 
\begin{equation}\label{eq:widvfwb}
	D_1 + \cdots + D_n = n\quad\text{ and }\quad D_1 + \cdots + D_m \le m \quad \text{for all } m.
\end{equation} 
This conditioning results in a uniform Dyck path, because each configuration $(D_1=d_1,\dots,D_n=d_n)$ of down steps satisfying \cref{eq:widvfwb} is sampled with probability 
\begin{equation*}
	\prod_{j=1}^n \frac{1}{2^{d_j+ 1}} = \frac{1}{2^n \cdot 2^n}.
\end{equation*} 
With this notation, the $i$--th up step in the path occurs after $(D_1 + 1) + (D_2 + 1) + \cdots + (D_{i-1} + 1)$ steps. The random variables $D_i - 1$ are iid and centered, and \cref{eq:widvfwb} gives that their sum is $0$ and all partial sums are nonpositive. Thus, the process
\[
    \tilde{D}(x) = \text{LI}\left(\frac{(D_1 - 1) + (D_2 - 1) + \cdots + (D_{nx} - 1)}{\sqrt{n}}\right)_{x \in \{0, \frac{1}{n}, \cdots, 1\}}
\]
which linearly interpolates the partial sums between multiples of $\frac{1}{n}$ and then rescales by $n$ in time and $\sqrt{n}$ in space, converges uniformly to a negative Brownian excursion on $[0,1]$. By \cite[Theorem 1]{excursion}, the probability that the infimum of a negative Brownian excursion is less than $-0.1n^{1/4}$ is $o(e^{-n})$,  which means that with probability $1 - o(e^{-n})$, each rescaled partial sum $\frac{(D_1 - 1) + (D_2 - 1) + \cdots + (D_{i-1} - 1)}{\sqrt{n}}$ is in the interval $[-0.1n^{1/4}, 0]$, so $(D_1 + 1) + (D_2 + 1) + \cdots + (D_{i-1} + 1)$ is in the interval $[2(i-1)-0.1n^{3/4}, 2(i-1)]$. Since this quantity is the number of steps before the $i$--th up step, this implies \cref{eq:partial-goal}.

\medskip

To complete the proof of the lemma, notice that sampling a uniform random walk $W_{n} = (X_{n}, Y_{n})$ as in \cref{randomwalkisuniformschnyder2} can be done in two stages. Any such walk consists of $n$ ``down'' steps each of increment $(1, -1)$ and $n$ ``up'' steps of increments $(-k_1, 1), \cdots, (-k_n, 1)$. So our sampling process proceeds in the following way: 

\begin{itemize}
\item In stage $1$, we pick the values of the $k_i$s, weighted proportionally to the number of valid walks that may occur with those increments.
\item Then in stage $2$, we choose one of the walks with these increments staying in $\{(x, y) \in \mathbb{Z}^2: x \ge 0, y \ge -1\}$ uniformly at random. (Here, the only remaining randomness is the order in which up versus down steps are taken.)
\end{itemize}

The key observation is that for any fixed choice of $k_i$s, the sampling in stage $2$ can be described similarly to our intermediate result. Indeed, we are choosing a walk which takes $D_1'$ down steps, then an up step, then $D_2'$ down steps, then an up step, concluding with $D_n'$ down steps and finally an up step, and requiring that this walk stays within our desired region is equivalent to requiring
\[
    D_1' + \cdots + D_n' = n, \quad D_1' + \cdots + D_m' \le m \quad\text{and}\quad D_1' + \cdots + D_m' \ge k_1 + \cdots + k_m \quad \text{for all }m,
\]
where the conditions come from the $y$-coordinate being at least $-1$ and the $x$-coordinate being at least $0$, respectively. Just like in our intermediate result above, we may sample uniformly from all such possibilities by letting $D_i'$ be iid $\text{Geom}(1/2) - 1$ variables and conditioning on these inequalities. But now if we consider the rescaled process linearly interpolating between the partial sums of $(D_i' - 1)$, we may interpret the result as a negative Brownian excursion further conditioned to stay above the (negative) profile  (which is deterministic, since we condition on $k_1, \cdots, k_n$)
\[
    \tilde{K}(x) = \text{LI}\left(\frac{(k_1 - 1) + (k_2 - 1) + \cdots + (k_{nx} - 1)}{\sqrt{n}}\right)_{x \in \{0, \frac{1}{n}, \cdots, 1\}}.
\]
This conditioning can only bias the distribution of the (negative) infimum upward, meaning that it is still true that the partial sums $\frac{(D_1' - 1) + (D_2' - 1) + \cdots + (D_{i-1}' - 1)}{\sqrt{n}}$ are all in the interval $[-0.1n^{1/4}, 0]$ with high probability. Since the $i$--th leftmost starting point $Z^r_n$ is $(D_1' + 1) + \cdots + (D_{i-1}' + 1)$ (here noting that the first starting point is always $0$ because the final step in $W_n$ is always an  ``up step''), we find just like above that the starting time is in the interval $[2(i-1) - 0.1n^{3/4}, 2(i-1)]$ with high probability. Since this reasoning holds for any choice of $(k_1, \cdots, k_n)$, it also holds for $W_n = (X_n, Y_n)$ overall, completing the proof. 
\end{proof}

This estimate on the locations of starting points for the coalescent-walk processes now allows us to describe the sample paths started from those points. The next two results establish the conditioned version of the joint convergence proved in \cref{unconditionedconvergence} in the unconditioned case. We rescale our driving walk and coalescent-walk processes in the same way as in \cref{xyunceqn} and \cref{zuncgeqn}: letting $W_{n} = (X_{n}, Y_{n})$ as in \cref{randomwalkisuniformschnyder2} with associated coalescent-walk processes $Z^g_{n}$ and $Z^r_{n}$ as in \cref{cwrdefn2,cwrdefn}, we may define

\begin{equation}
    \tilde{X}_{n}(\cdot) = \text{LI}\left(\frac{1}{\sqrt{4n}} X_{n}(2nt)\right)_{t \in \{0, \frac{1}{2n}, \cdots, 1\}}, \quad \tilde{Y}_{n}(\cdot) = \text{LI}\left(\frac{1}{\sqrt{2n}} Y_{n}(2nt)\right)_{t \in \{0, \frac{1}{2n}, \cdots, 1\}},
\end{equation}
\begin{equation}\label{zgeqn}
    \left(\tilde{Z}^{\circ}_{n}\right)^{(j/(2n))} (\cdot) = \text{LI}\left(\begin{cases} \frac{1}{\sqrt{2n}} (Z^{\circ}_{n})^{(j)}_{2nt} & \text{if } (Z^{\circ}_{n})^{(j)}_{2nt} \ge 0 \\ \frac{1}{\sqrt{4n}} (Z^{\circ}_{n})^{(j)}_{2nt} & \text{if } (Z^{\circ}_{n})^{(j)}_{2nt} < 0 \end{cases}\right)_{t \in \{0, \frac{1}{2n}, \cdots, 1\}},
\end{equation}
where $\circ\in\{g,r\}$.

\begin{corollary}\label{jointproductpreresult}
Let $\rho = -\frac{\sqrt{2}}{2}$, $q = \frac{1}{1 + \sqrt{2}}$, and fix $u \in (0, 1)$. For all $n$, let $W_{n} = (X_{n}, Y_{n})$ be as in \cref{randomwalkisuniformschnyder2} with associated coalescent-walk processes $Z^g_{n}$ and $Z^r_{n}$. Let $j_n$ be the $x$-coordinate of the $\lceil nu \rceil$--th leftmost starting point of $Z^r_n$. Then we have the following joint convergence in $C([0, 1], \RR)^4 \times C((0, 1), \RR)^2$:
\begin{equation*}
	   \left(\tilde{X}_{n}, \tilde{Y}_{n}, \tilde{X}_{n}', \tilde{Y}_{n}', \left(\tilde{Z}^{g}_{n}\right)^{(1 - j_n/(2n))}, \left(\tilde{Z}^{r}_{n}\right)^{(j_n/(2n))}\right)
	\stackrel{d}{\longrightarrow} \left(X_\rho, Y_\rho, X_\rho', Y_\rho', (Z_{\rho, q}')^{(1 - u)}, (Z_{\rho, q})^{(u)}\right),
\end{equation*}
where $(X_\rho, Y_\rho)$ is a two-dimensional Brownian excursion of correlation $\rho$, $(X_\rho', Y_\rho')$ is its time-reversal, $Z_{\rho, q}$ is the unique solution to the SDE in \cref{skewbrowniansdedefn} driven by $(X_\rho, Y_\rho)$, and $Z_{\rho, q}'$ is the solution driven by $(X_\rho', Y_\rho')$.
\end{corollary}

\begin{proof}
We just sketch the proof, since the technical steps are similar to those given in \cite[Appendix E]{borga2021permuton}. By \cref{randomwalkisuniformschnyder2}, we know that $(X_{n}, Y_{n})$ must start and end at $(0, 0)$, so we may apply \cite[Theorem 6]{DURAJ20203920} translated down by $1$ unit (that is, using a cone centered at $(0, -1)$ rather than the origin). This shows that $(X_{n}, Y_{n})$ converges to $(X_\rho, Y_\rho)$ and similarly shows the analogous result for the time-reversal, so that we have joint convergence for those four walks.

Now we make use of absolute continuity arguments, detailed in \cite[Proposition A.1]{borga2023skew}. A Brownian motion and Brownian excursion of the same correlation $\rho$ are absolutely continuous with respect to each other on $[\eps, 1-\eps]$ for any positive $\eps$. Thus because we have already established the desired convergence in distribution in the unconditioned case, we also establish it in the conditioned case on $[\eps, 1-\eps]$ because restriction from $[0, 1]$ to the smaller interval is a continuous function, and on any such interval measurability of the coalescent-walk processes in terms of the driving walks carries over from the discrete case to the scaling limit. 

By \cref{schnyderlocationtail}, except on an event of asymptotically exponentially small probability, the location $j_n$ of the $\lceil nu \rceil$--th leftmost starting point in the red coalescent-walk process $Z^r_{n} = \WC^r(W_n)$ is within $n^{3/4}$ of $2nu$. Furthermore, \cref{greenreddualstart} tells us that the location $2n - j_n$ of this same label $\lceil nu \rceil$ in the green coalescent-walk process $Z^g_n = \WC^g(W_n')$ is within $n^{3/4}$ of $2n(1 - u)$. Thus in the scaling limit we have that restricted to $[\eps, 1 - \eps]$ for any $\eps < \min(u, 1-u)$, $\left(\tilde{Z}^{r}_{n}\right)^{(j_n/(2n))}$ converges to $(Z_{\rho, q})^{(u)}$ in distribution, and similarly $\left(\tilde{Z}^{g}_{n}\right)^{(1 - j_n/(2n))}$ converges to $(Z_{\rho, q}')^{(1-u)}$ in distribution.

From here, simultaneously applying the result for all sufficiently small rationals $\eps$ and using Skorokhod's theorem, this result can be upgraded to uniform convergence on intervals not containing $0$ or $1$, hence convergence on $(0, 1)$. Finally, we also obtain convergence for the four walks $(\tilde{X}_{n}, \tilde{Y}_{n}, (\tilde{X}_{n}'), (\tilde{Y}_{n}'))$ at the endpoints $0$ and $1$ because all of those walks start and end at zero.
\end{proof}

Finally, we upgrade this last result to random times $u$, which will be necessary for reading off random patterns from the coalescent-walk processes:

\begin{corollary}\label{jointproductresult}
For all $n$, let $W_{n} = (X_{n}, Y_{n})$ be as in \cref{randomwalkisuniformschnyder2} with associated coalescent-walk processes $Z^g_{n}$ and $Z^r_{n}$. Let $(u_i)_i$ be an infinite sequence of iid uniform $[0, 1]$ variables, and for each $i$, let $j_{n, i}$ be the $x$-coordinate of the $\lceil nu_i \rceil$--th leftmost starting point of $Z^r_n$. Then we have the following joint convergence in the product topology:
\begin{multline*}
	\left(\tilde{X}_{n}, \tilde{Y}_{n}, \tilde{X}_{n}', \tilde{Y}_{n}', \left(\left(\tilde{Z}^{g}_{n}\right)^{(1 - j_{n, i}/(2n))}\right)_{i \in \NN}, \left(\left(\tilde{Z}^{r}_{n}\right)^{(j_{n, i}/(2n))}\right)_{i \in \NN}\right)\\
	\stackrel{d}{\longrightarrow}\quad \left(X_\rho, Y_\rho, X_\rho', Y_\rho', \left((Z_{\rho, q}')^{(1 - u_i)}\right)_{i \in \NN}, \left((Z_{\rho, q})^{(u_i)}\right)_{i \in \NN}\right),
\end{multline*}
where $(X_\rho, Y_\rho)$ is a two-dimensional Brownian excursion of correlation $\rho$, $(X_\rho', Y_\rho')$ is its time-reversal, $Z_{\rho, q}$ is the unique solution to the SDE in \cref{skewbrowniansdedefn} driven by $(X_\rho, Y_\rho)$, and $Z_{\rho, q}'$ is the solution driven by $(X_\rho', Y_\rho')$.
\end{corollary}

\begin{proof}
\cref{jointproductpreresult} implies that this result holds for any finitely many deterministic $u_i$, because $Z_{\rho, q}'$ and $Z_{\rho, q}$ are measurable functions of their driving Brownian excursions (hence we may apply the corollary for the $u_i$ simultaneously). In particular, for any bounded continuous functional $\phi$, any deterministic $u_i$, and any finite $N$, we have 
\begin{multline*}
	    \EE\left[\phi\left(\tilde{X}_{n}, \tilde{Y}_{n}, \tilde{X}_{n}', \tilde{Y}_{n}', \left(\left(\tilde{Z}^{g}_{n}\right)^{(1 - j_{n, i}/(2n))}\right)_{i \in [N]}, \left(\left(\tilde{Z}^{r}_{n}\right)^{(j_{n, i}/(2n))}\right)_{i \in [N]}\right)\right] \\
	    \longrightarrow\quad
		 \EE\left[\phi\left(X_\rho, Y_\rho, X_\rho', Y_\rho', \left((Z_{\rho, q}')^{(1 - u_i)}\right)_{i \in [N]}, \left((Z_{\rho, q})^{(u_i)}\right)_{i \in [N]}\right)\right].
\end{multline*}
Now integrating each $u_i$ on $(0, 1)$ and swapping the order of integration by Fubini's theorem implies that this convergence of expectation holds if each $u_i$ is now uniform on $[0, 1]$, which is exactly the definition of convergence in the product topology.
\end{proof}

We are now ready to prove \cref{mainresult3}, showing permuton convergence via pattern convergence, thanks to \cref{tfaemain}.

\begin{proof}[Proof of \cref{mainresult3}]

For all $n$, let $W_{n} = (X_{n}, Y_{n})$ be the uniform conditioned walk as in \cref{randomwalkisuniformschnyder2} with associated uniform 3-dimensional Schnyder wood permutation $\sigma_{n}=(\sigma_{n}^g,\sigma_{n}^r)$, and let $Z^g_n$ and $Z^r_n$ be the corresponding random coalescent-walk processes obtained through the commutative diagram of bijections in \cref{fig-comm-diagram}.

By \cref{tfaemain}, to prove that $\mu_{\sigma_{n}}$ converges to the Schnyder wood permuton $\mu_S$, it suffices to prove that for any positive integer $k$, the random 3-dimensional pattern $\sigma_{n, k} = \pat_{I_{n, k}}(\sigma_n)$ converges in distribution to the random 3-permutation $P_{\mu_S}[k]$ of size $k$ (recall \cref{defnsamplefrompermuton}).

To do this, we first look at $\sigma_{n, k}$ and prove that it has a limiting distribution $\rho_k$ (Steps 1--4), and then we show that this limiting distribution is exactly $P_{\mu_S}[k]$ (Step 5). Finally, we prove the claim in \cref{eq:weifvwebfowep}, characterizing the marginals $\mu_S^{g}$ and $\mu_S^{r}$ of $\mu_S$ (Step 6).

\medskip

\noindent\emph{\underline{Step 1}: Sampling patterns from $\sigma_n$.}
Let $u_1,\cdots, u_k$ be iid uniform on $[0, 1]$, and let $\widetilde{\sigma}_{n,k}$ be the $3$-dimensional pattern of $\sigma_{n}$ of size $k$ on the indices $\widetilde{I}_{n,k} = \{\lceil nu_1 \rceil, \cdots, \lceil nu_k \rceil\}$ if they are all distinct, and the identity $3$-permutation otherwise. The latter case happens only with probability $O(\frac{1}{n})$ and is thus negligible as $n \to \infty$. In words, $\widetilde\sigma_{n, k}$ is, up to this $O(\frac{1}{n})$-probability event, the random permutation obtained when $k$ uniform points are selected from a uniform random Schnyder wood permutation of size $n$, and hence the total variation distance of $\widetilde\sigma_{n, k}$ and $\sigma_{n, k}$ tends to zero as $n \to \infty$. For this reason, from now on, we work conditioning on the event that the indices $\lceil nu_1 \rceil, \cdots, \lceil nu_k \rceil$ are all distinct, and we only look at $\widetilde\sigma_{n, k}$ (instead of $\sigma_{n, k}$). To simplify notation, we also drop the tilde from $\widetilde\sigma_{n, k}$, simply writing $\sigma_{n, k}$ (since we believe that this abuse of notation should not create any confusion).

\medskip

\noindent\emph{\underline{Step 2}: Reading patterns on $Z_n^g$ and $Z_n^r$.}
We now explain how to determine the pattern $\sigma_{n, k}=(\sigma_{n, k}^g,\sigma_{n, k}^r)$ from the coalescent walk processes $Z_n^g$ and $Z_n^r$. For all $1 \le i \le n$, let $g_{n, i}$ and $r_{n, i}$ be the $x$-coordinates of the starting points $\lceil nu_i \rceil$ in $Z^g_n$ and $Z^r_n$, respectively.
\begin{itemize}
\item Recall from \cref{mainresult1} that $\sigma_n^g = \sigma^{\text{up}}(Z^g_n)$, and recall also that the starting points of the paths in $Z^g_n$ are labeled in descending order from left to right. Therefore, recalling the definitions in \cref{totalorder} and \cref{coalescentpermutation}, for any $1 \le i, j \le k$ with $\lceil nu_i \rceil > \lceil nu_j \rceil$ (and hence $g_{n, i} < g_{n, j}$),
\begin{equation}\label{schnydergreenrelation}
\begin{split}
    \sigma_n^g(\lceil nu_i \rceil) < \sigma_n^g(\lceil nu_j \rceil) &\iff g_{n, i} \le^{\text{up}} g_{n, j} \\
    &\iff (Z^g_n)^{(g_{n, i})}_{g_{n, j}} < 0 \\
    &\iff \text{sgn}\left((\tilde{Z}^g_n)^{(g_{n, i}/(2n))}(g_{n, j}/(2n))\right) = -1.
\end{split}
\end{equation}
\item Similarly, $\sigma_n^r = \sigma^{\text{down}}(Z^r_n)$, and recall also that the starting points of the paths in $Z^r_n$ are labeled in ascending order from left to right. Therefore, for any $1 \le i, j \le k$ with $\lceil nu_i \rceil < \lceil nu_j \rceil$ (and hence $r_{n, i} < r_{n, j}$),
\begin{equation}\label{schnyderredrelation}
\begin{split}
    \sigma_n^r(\lceil nu_i \rceil) < \sigma_n^r(\lceil nu_j \rceil) &\iff r_{n, i} \le^{\text{down}} r_{n, j}\\
    &\iff (Z^r_n)_{r_{n, j}}^{(r_{n, i})} > 0 \\
    &\iff \text{sgn}\left((\tilde{Z}^r_n)^{(r_{n, i}/(2n))}(r_{n, j}/(2n)) \right) = +1.
\end{split}
\end{equation}
\end{itemize}

\noindent The patterns of $\sigma_n^g$ and $\sigma_n^r$ on the indices $\lceil nu_1 \rceil, \cdots, \lceil nu_k \rceil$ depend only on the above pairwise relations. Hence, in order to determine the distribution of the pattern $\sigma_{n, k}=(\sigma_{n, k}^g,\sigma_{n, k}^r)$, we need to study the signs of certain sample paths of $Z_n^g$ and $Z_n^r$ at certain times. This is the goal of the next step.

\medskip

\noindent\emph{\underline{Step 3:} The limiting behavior of the signs of paths in $Z_n^g$ and $Z_n^r$.} Since $u_1, \cdots, u_k$ are iid uniform, we may apply \cref{jointproductresult}. Additionally, by Skorokhod's theorem, we can assume that the convergence in \cref{jointproductresult} holds almost surely, and thus for each pair $i, j \in \binom{[k]}{2}$, setting
\begin{align*}
    &m^g_{n,i,j}:=\min(g_{n, i}, g_{n, j})/(2n),\qquad
    M^g_{n,i,j}:=\max(g_{n, i}, g_{n, j})/(2n),\\
    &m^r_{n,i,j}:=\min(r_{n, i}, r_{n, j})/(2n),\qquad
    M^r_{n,i,j}:=\max(r_{n, i}, r_{n, j})/(2n),
\end{align*}
we have that, almost surely,
\begin{equation}\label{eq:ebjwfwebfw1}
\text{sgn}\left((\tilde{Z}^g_{n})^{(m_{n,i,j}^{g})} (M_{n,i,j}^{g})\right)
    \longrightarrow \text{sgn}\left((Z'_{\rho, q})^{(\min(1 - u_i, 1 - u_j))} (\max(1 - u_i, 1 - u_j))\right),
\end{equation}
and also 
\begin{equation}\label{eq:ebjwfwebfw2}
    \text{sgn}\left((\tilde{Z}^r_{n})^{(m^r_{n,i,j})}(M^r_{n,i,j})\right)
    \longrightarrow \text{sgn}\left((Z_{\rho, q})^{(\min(u_i, u_j))}(\max(u_i, u_j))\right).
\end{equation}
Here we use that the quantities inside $\text{sgn}$ on the right-hand sides are nonzero almost surely (by \cite[Lemma 3.2]{borga2023skew}) and hence $\text{sgn}(\cdot)$ is a continuous function at this point almost surely. Note that the signs on the left-hand sides of these two statements are exactly the quantities needed to determine the pattern $\sigma_{n, k}=(\sigma_{n, k}^g, \sigma_{n, k}^r)$, as we explained in the previous step.

\medskip

\noindent\emph{\underline{Step 4:} Concluding that the 3-dimensional pattern $\sigma_{n,k}$ converges.} The convergence statements in \cref{eq:ebjwfwebfw1,eq:ebjwfwebfw2}, taken jointly for all $i, j$, imply that for any $k\geq 1$, the $3$-permutation pattern $\sigma_{n, k}= (\sigma_{n, k}^g,\sigma_{n, k}^r)$ indeed converges in distribution as $n \to \infty$ to the random limiting 3-permutation $\rho_k = (\rho_k^g,\rho_k^r)$, defined as follows:
\begin{align}
    &\text{for any } i > j, \qquad \rho_k^g(i) < \rho_k^g(j) \iff \text{sgn}\left((Z'_{\rho, q})^{(1 - v_i)}(1 - v_j)\right) = -1,\label{eq:rel1ord}\\
    &\text{for any }i < j, \qquad \rho_k^r(i) < \rho_k^r(j) \iff \text{sgn}\left((Z_{\rho, q})^{(v_i)}(v_j)\right) = +1,\label{eq:rel2ord}
\end{align}
where $v_1 < \cdots < v_k$ are uniform random variables on $[0,1]$, relabeled to be ordered. It remains to check that for all $k\geq 1$, $\rho_k=(\rho_k^g,\rho_k^r)$ is exactly distributed as $P_{\mu_S}[k]\in \mc{S}_{3, k}$, which is the goal of the next step.

\medskip 

\noindent\emph{\underline{Step 5}: Identifying the correct permuton limit.} Recall that $Z_{\rho, q}$ and $Z'_{\rho, q}$ induce
\begin{itemize}
    \item two total orders on a subset of $[0, 1]$ of Lebesgue measure one via \cref{skewbrownianordering}, which we call $<_{\rho, q}$ and $<'_{\rho, q}$, along with
    \item two random functions $\phi^r_{\rho, q}$ and $\phi^g_{\rho, q}$, defined as in \cref{eq:defn-ass-process}. 
\end{itemize}
We will thus also represent our permutation patterns in terms of these random functions. Note that, for any $i > j$ (we continue to have $i, j \in [k]$),
\begin{align*}
     \rho_k^g(i) < \rho_k^g(j) &\stackrel{\eqref{eq:rel1ord}}{\iff} \text{sgn}\left((Z'_{\rho, q})^{(1 - v_i)}(1 - v_j)\right) = -1 \\
    &\stackrel{\eqref{skewbrownianordering}}{\iff} 1 - v_i <'_{\rho, q} 1 - v_j \\
    &\stackrel{\eqref{rel-order}}{\iff} \phi^g_{\rho, q}(1 - v_i) < \phi^g_{\rho, q}(1 - v_j).
\end{align*}
Similarly, for any $i < j$,
\begin{align*}
    \rho_k^r(i) < \rho_k^r(j) &\stackrel{\eqref{eq:rel2ord}}{\iff}\text{sgn}\left((Z_{\rho, q})^{(v_i)}(v_j)\right) = +1\\
    &\stackrel{\eqref{skewbrownianordering}}{\iff} v_j <_{\rho, q} v_i \\
    &\stackrel{\eqref{rel-order}}{\iff} \phi_{\rho, q}^r(v_i) > \phi_{\rho, q}^r(v_j)\\
    &\iff 1-\phi_{\rho, q}^r(v_i)< 1- \phi_{\rho, q}^r(v_j).
\end{align*}
Summarizing, we get that for all $i$ and $j$ (by possibly swapping their roles in the implications above),
\begin{multline}\label{eq:condition}
      \left\{\rho_k^g(i) < \rho_k^g(j)\,\,,\,\,\rho_k^r(i) < \rho_k^r(j)\right\} 
      \iff\left\{\phi^g_{\rho, q}(1 - v_i) < \phi^g_{\rho, q}(1 - v_j)\,\,,\,\,1-\phi_{\rho, q}^r(v_i)< 1- \phi_{\rho, q}^r(v_j)\right\}.
\end{multline}
Recall now from \cref{schnyderpermutonexplicit} that $\mu_S$ is defined by
\begin{equation}\label{eq:dibcewib}
    \mu_S(A) = \text{Leb}\left(\left\{t \in [0, 1]: \left(t \,,\, \phi^g_{\rho,q}(1 - t) \,,\, 1 - \phi^r_{\rho,q}(t) \right) \in A\right\}\right),
\end{equation}
for all measurable subsets $A \subseteq [0, 1]^3$. Therefore, recalling \cref{eq:patt-from-perm}, we may obtain the permutation pattern
\[
    P_{\mu_S}[k]=\left(P^g_{\mu_S}[k],P^r_{\mu_S}[k]\right)
\]
by sampling the $k$ iid 3-dimensional points each with law $\mu_S$
\[
    \bigg(\left(v_1, \phi^g_{\rho,q}(1 - v_1), 1 - \phi^r_{\rho,q}(v_1)\right),\dots,\left(v_k, \phi^g_{\rho,q}(1 - v_k),1 - \phi^r_{\rho,q}(v_k)\right)\bigg),
\]
reordered so the $v_i$s are increasing. Then almost surely, for all $i$ and $j$ we have
\begin{equation*}
    (P^g_{\mu_S}[k])(i)<(P^g_{\mu_S}[k])(j)\iff \phi^g_{\rho,q}(1 - v_i)<\phi^g_{\rho,q}(1 - v_j)
\end{equation*}
and
\begin{equation*}
    (P^r_{\mu_S}[k])(i)<(P^r_{\mu_S}[k])(j)\iff 1 - \phi^r_{\rho,q}(v_i)<1 - \phi^r_{\rho,q}(v_j).
\end{equation*}
Comparing the latter two displayed equations with the conditions in \cref{eq:condition}, we immediately get that $\rho_k=P_{\mu_S}[k]$, as we wanted.

\medskip

\noindent\emph{\underline{Step 6:} Characterizing the marginals $\mu_S^{g}$ and $\mu_S^{r}$.} Finally, it remains to prove the claim in \cref{eq:weifvwebfowep}. Recalling that $f(x, y) = (1-x, y)$, we can rewrite the marginal $\mu_S^{g}$ as
\begin{align*}
    \mu_S^{g}(B\times C)= \mu_S(B\times C \times [0,1])
    &\stackrel{\eqref{eq:dibcewib}}{=}\text{Leb}\left(\left\{t \in [0, 1]: (t, \phi_{\rho, q}^g(1-t))\in B\times C\right\}\right)\\
    &\,\,=\,\,\text{Leb}\left(\left\{t \in [0, 1]: (1-t, \phi_{\rho, q}^g(t))\in B\times C\right\}\right)\\
    &\,\,=\,\,\text{Leb}\left( \left\{t \in [0, 1]: f(t, \phi_{\rho, q}^g(t))\in B\times C\right\}\right)\\
    &
    \stackrel{\eqref{eq:defn-perm}}{=}\mu_{\rho, q}^g(f^{-1}(B\times C)),
\end{align*}
where in the second equality we used that $t\mapsto1-t$ is a measure preserving map. Similarly, recalling that $h(x, y) = (x,1-y)$, we have that
\begin{align*}
    \mu_S^{g}(B\times C)\,\,=\,\,\mu_S(B \times [0,1] \times C )&\stackrel{
    \eqref{eq:dibcewib}
    }{=} \text{Leb}\left(\left\{t \in [0, 1]: (t, 1-\phi^r_{\rho, q}(t)) \in B \times C\right\}\right)\\
    &\,\,=\,\, \text{Leb}\left(\left\{t \in [0, 1]: h(t,\phi^r_{\rho, q}(t))\in B \times C\right\}\right)\\
    &\stackrel{\eqref{eq:defn-perm}}{=}  \mu^r_{\rho, q}(h^{-1}(B\times C)).
\end{align*}
This completes the proof.
\end{proof}

We conclude this section by proving \cref{marginaldetermines} and \cref{SLES-LQG}:

\begin{proof}[Proof of \cref{marginaldetermines}]
We show that the marginal $2$-permuton $\mu_S^{g}$ determines the Schnyder wood permuton
\begin{equation}\label{eq:edwviifvwe}
    \mu_S(A) = \text{Leb}\left(\left\{t \in [0, 1]: \left(t,  \phi^g_{\rho,q}(1 - t),
    1 - \phi^r_{\rho,q}(t)\right) \in A\right\}\right).
\end{equation}
The proof for $\mu_S^{r}$ is identical.

Recall from \cref{mainresult3} that $\mu_S^{g}=f_\ast(\mu_{\rho, q}^g)$ with $f(x, y) = (1-x, y)$, so equivalently $\mu_{\rho, q}^g = f^{-1}_\ast(\mu_S^g)$ is determined by $\mu_S^g$. Thus, it is enough to show that $\mu_{\rho, q}^g$ determines $\mu_S$.
Recall also from \cref{mainresult3} that $W_\rho$ is the two-dimensional Brownian excursion of correlation $\rho$ used in the construction of $\mu_{\rho, q}^r$ and $\phi_{\rho, q}^r$, and $W'_\rho$ is its time-reversal used in  the construction of $\mu_{\rho, q}^g$ and $\phi_{\rho, q}^g$.

By \cref{prop:det-perm},  $\mu_{\rho, q}^g$ determines $W'_\rho$, which in turn determines $W_\rho$. But $W'_\rho$ and $W_\rho$ clearly determine $\phi^g_{\rho,q}$ and $\phi^r_{\rho,q}$ (recall \cref{skewbrowniansdedefn2}). From this and \eqref{eq:edwviifvwe}, we can conclude that $\mu_S^{g}$ determines $\mu_S$.
\end{proof}

The reader unfamiliar with SLE and LQG can skip the next proof.

\begin{proof}[Proof of \cref{SLES-LQG}]
We will construct a coupled version of the objects involved in the proposition statement so that all of the equalities in distribution are actually almost sure equalities.

Recall that $(\eta^b,\eta^r,\eta^g)$ are three whole-plane space-filling SLE$_\kappa$ counter-flow lines of $\hat{h}$ of angle $\left(0,\frac{2\pi}{3},\frac{4\pi}{3}\right)$ respectively, and we are considering the curve-decorated quantum surface $((\mathbb{C} \cup \{\infty\}, h, \infty), \eta^b)$. Let $W_\rho=(X_\rho,Y_\rho)$ be the corresponding two-dimensional Brownian excursion of correlation $\rho$ given by the boundary LQG lengths as in \cref{lqgdescription2}.

Let $\chi^{b,g}: [0,1]\to[0,1]$ be a Lebesgue measurable function such that
\begin{equation}\label{eq:eivfouqwebfepw}
    \eta^b(1-t) = \eta^g\left(\chi^{b,g}(t)\right),\quad\text{ for all } t \in [0,1],
\end{equation}  
and let $\chi^{b,r}: [0,1]\to[0,1]$ be a Lebesgue measurable function such that 
\begin{equation}\label{eq:eivfouqwebfepw2}
    \eta^b(t) = \eta^r\left(\chi^{b,r}(t)\right),\quad\text{ for all } t \in [0,1].
\end{equation}
Also let $\mu_{\rho, q}^g$ and $\mu_{\rho, q}^r$ be (coupled) skew Brownian permutons of parameter $(\rho, q)$, driven by $W'_\rho=(X'_\rho,Y'_\rho)$ and $W_\rho=(X_\rho,Y_\rho)$ respectively, with associated random functions $\phi^g_{\rho,q}$ and $\phi^r_{\rho,q}$ as in \cref{skewbrowniansdedefn2}. By \cref{lqgdescription2} (used with $q=\frac{1}{1+\sqrt{2}}$ and so $\theta=\frac{2\pi}{3}$ thanks to \cref{rem:angle}), we have that almost surely
\begin{equation}\label{eq:eboufebbfephfe2}
    1-\chi^{b, r}(t)= \phi^r_{\rho,q}(t),\quad\text{ for almost every } t \in [0,1].
\end{equation}
Moreover, noting that the angle between the (time-reversed) SLE $\eta^b(1-t)$ and the (non-time-reversed) SLE $\eta^g(t)$ is $\frac{\pi}{3}$, by \cref{lqgdescription2} (used with $\widetilde{q}=1-\frac{1}{1+\sqrt{2}}$ and so $\widetilde{\theta}=\frac{\pi}{3}$ thanks to \cref{rem:angle}), we have that almost surely
\begin{equation*}
    1-\chi^{b, g}(t)= \widetilde{\phi}^g_{\rho,q}(t),\quad\text{ for almost every } t \in [0,1],
\end{equation*}
where $\widetilde{\phi}^g_{\rho,q}(t)$ is the random function in \cref{skewbrowniansdedefn2} when one considers\footnote{Note that here the coordinates of the reversed Brownian excursion are swapped; indeed, time-reversing an SLE also swaps the roles of left and right.} $(Y'_\rho,X'_\rho)$ as the driving Brownian excursion and $\widetilde{q}=1-\frac{1}{1+\sqrt{2}}$ as the skewness. Since the coalescent-walk process corresponding to $\widetilde{\phi}^g_{\rho,q}(t)$ can be obtained from the one corresponding to $\phi^g_{\rho,q}(t)$ by flipping vertically (i.e.\ along the $x$-axis) all of the sample paths (that is, exchanging the roles of $Y'_\rho$ and $X'_\rho$ and changing the parameter $\widetilde{q}=1-\frac{1}{1+\sqrt{2}}$ into $1-\widetilde{q}=q=q(\frac{2\pi}{3})=\frac{1}{1+\sqrt{2}}$), we conclude that $\widetilde{\phi}^g_{\rho,q}(t)=1-\phi^g_{\rho,q}(t)$, and so that \begin{equation}\label{eq:eboufebbfephfe3}
    \chi^{b, g}(t)= \phi^g_{\rho,q}(t),\quad\text{ for almost every } t \in [0,1].
\end{equation}

Recall now from \cref{eq:wefivfvweofu} that $\psi^{b,\circ}:[0,1]\to[0,1]$ for $\circ\in\{r,g\}$ are Lebesgue measurable functions such that (for later reasons, it is convenient to write the first relation in this alternative way)
\begin{equation}\label{eq:wejfboweubfwohewfiweoh}
    \eta^b(1-t)=\eta^{g}\left(\psi^{b,g}(1-t)\right),\qquad \text{for all $t\in[0,1]$,}
\end{equation}
and
\begin{equation}\label{eq:wefivhbweuofbweo2}
    \eta^b\left(t\right)=\eta^r\left(\psi^{b,r}(t)\right),\qquad \text{for all $t\in[0,1]$.}
\end{equation}
Comparing \cref{eq:wejfboweubfwohewfiweoh} with \cref{eq:eivfouqwebfepw}, and using that almost surely almost every point $x\in\mathbb{C}$ is a simple point of $\eta^g$, we get that, almost surely, 
\begin{eqnarray*}
    \phi^g_{\rho,q}(t)\stackrel{\eqref{eq:eboufebbfephfe3}}{=}\chi^{b,g}(t)=\psi^{b,g}(1-t),\quad\text{ for almost every } t \in [0,1].
\end{eqnarray*}
From this we can conclude that, almost surely,
\begin{equation}\label{eq:wefobuwebofcnwepifhwe}
    \phi^g_{\rho,q}(1-t)=\psi^{b,g}(t),\quad\text{ for almost every } t \in [0,1].
\end{equation}
Similarly, comparing \cref{eq:wefivhbweuofbweo2} and \cref{eq:eivfouqwebfepw2}, we get that,  almost surely,
\begin{eqnarray*}
    \phi^r_{\rho,q}(t)\stackrel{\eqref{eq:eboufebbfephfe2}}{=}1-\chi^{b,r}(t)=1-\psi^{b,r}(t),\quad\text{ for almost every } t \in [0,1].
\end{eqnarray*}
From this we can conclude that, almost surely, 
\begin{equation}\label{eq:wefobuwebofcnwepifhwe2}
    1 - \phi^r_{\rho,q}(t)=\psi^{b,r}(t),\quad\text{ for almost every } t \in [0,1].
\end{equation}
From \cref{eq:wefobuwebofcnwepifhwe,eq:wefobuwebofcnwepifhwe2}, we conclude that, almost surely, for all measurable sets $A\subset [0,1]^3$,
 \begin{align*}
     \mu_S(A) &\stackrel{\eqref{schnyderpermutonexplicit}}{=} \text{Leb}\left(\left\{t \in [0, 1]: \left(t,  \phi^g_{\rho,q}(1 - t),
     1 - \phi^r_{\rho,q}(t)\right) \in A\right\}\right)\\
     &\,=\,\text{Leb}\left(\left\{t \in [0, 1]: \left(t,\psi^{b,g}(t),
     \psi^{b,r}(t)\right) \in A\right\}\right).
 \end{align*}
This proves the claim in \cref{eq:wefivoweuiqbfqwop}. Moreover, it follows from the definition of $\mu^{b, \circ}$ in \cref{eq:permutonfromcoupledcurves} that almost surely $\mu_S^{g}=\mu^{b,g}$ and $\mu_S^{r}=\mu^{b,r}$. This ends the proof.
\end{proof}

\section{The random $d$-dimensional permuton limit of $d$-separable permutations}\label{dseparablesection}

We now turn to describing the $d$-dimensional permuton limit of $d$-separable permutations, proving \cref{mainresult4}.

\subsection{A bijection with labeled trees}

Recall the definition of $d$-separable permutation provided in \cref{dseparablepermutationdefn}. While a $d$-separable permutation may be decomposed in different block sums,\footnote{For instance, the identity $d$-permutation $\sigma$ of size $n$ can be decomposed as $\sigma=\sigma_1 \encircle{s} \sigma_2$, where $s=(+, \dots, +)$, $\sigma_1$ is an identity $d$-permutation of size $\ell \in [n-1]$, and $\sigma_2$ is an identity $d$-permutation of size $n - \ell$.} the sign sequence $s$ for the block sum is always unique. Indeed, if there were two ways of decomposing a $d$-separable permutation $\sigma$ with different signs in the $j$--th coordinate into pieces of size $a_1 + a_2$ and $b_1 + b_2$ (for some $a_1, a_2, b_1, b_2 \in \NN$ such that $a_1 + a_2 = b_1 + b_2 = n$), then the first $\min(a_1, b_1)$ values of $\sigma(i)^{(j)}$ need to be among the top $a_1$ and also the bottom $b_1$ positions (or vice versa), and the last $\min(a_2, b_2)$ values need to be among the top $a_2$ and also bottom $b_2$ positions (or vice versa). Since either $a_1 + b_1$ or $a_2 + b_2$ is at most $n$ (their sum is $2n$), this cannot occur. 

Thus, given any $d$-separable permutation $\sigma$ of size at least 2, there is a \emph{unique} sign sequence $s$, called a \vocab{primary block structure} as in \cite{asinowski2008separable}, such that $\sigma = \sigma_1 \encircle{s} \sigma_2$ for two $d$-separable permutations $\sigma_1, \sigma_2$. Furthermore, block summing with the \emph{same} sign sequence $s$ is associative, so we can uniquely decompose any $d$-separable permutation $\sigma$ of size at least 2 as 
\begin{equation}\label{eq:block-struct}
    \sigma=\sigma_1 \encircle{s} \sigma_2 \encircle{s} \cdots \encircle{s} \sigma_N  \quad\text{for some}\quad N \ge 2,
\end{equation}
where each $\sigma_i$ is a $d$-separable permutation which cannot be written further as $\sigma_{i,1} \encircle{s} \sigma_{i,2}$ for any $d$-separable permutations $\sigma_{i, 1}, \sigma_{i, 2}$. Using \cref{eq:block-struct}, we may thus recursively encode $\sigma$ in a signs-labeled rooted plane tree $S(\sigma)$ as follows: 
\begin{itemize}
\item If $\sigma$ is the trivial permutation of size $1$, then the tree $S(\sigma)$ is just a single (unlabeled) vertex. 
\item Otherwise, label the root of the tree $S(\sigma)$ with the sign sequence $s$, give the root $N$ children, and have those $N$ children be the signs-labeled plane trees for $\sigma_1, \cdots, \sigma_N$ in order from left to right. 
\end{itemize}
An example is shown in \cref{fig:separable} below, encoding the $3$-separable permutation
\[
    \sigma = ((1, 3, 2, 4), (4, 2, 3, 1))
    =
    \text{id} \encircleBig{{\footnotesize{(+, -)}}} ((2,1),(1,2))\normalsize \encircleBig{{\footnotesize{(+, -)}}} \text{id}
    =
    \text{id} \encircleBig{{\footnotesize{(+, -)}}} \Big(\normalsize\text{id} \encircleBig{{\footnotesize{(-, +)}}} \text{id}\Big)\normalsize \encircleBig{{\footnotesize{(+, -)}}} \text{id},
\]
where $\text{id}$ denotes the identity $3$-permutation.

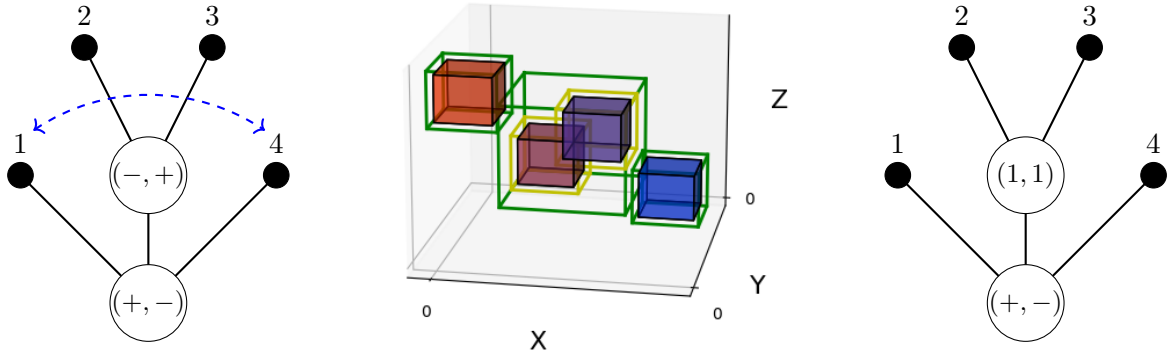
\begin{figure}[ht]
\begin{minipage}{.33\textwidth}
\begin{center}
\begin{tikzpicture}[scale=0.85]
\draw[thick] (0, 0) -- (-2, 2);
\draw[thick] (0, 0) -- (0, 2);
\draw[thick] (0, 0) -- (2, 2);
\draw[thick] (0, 2) -- (1, 4);
\draw[thick] (0, 2) -- (-1, 4);
\draw[fill=white] (0, 0) circle(0.6cm);
\node at (0, 0) {\small$(+, -)$};
\draw[fill=black] (-2, 2) circle(0.2cm);
\draw[fill=white] (0, 2) circle(0.6cm);
\draw[fill=black] (2, 2) circle(0.2cm);
\node at (0, 2) {\small$(-, +)$};
\draw[fill=black] (-1, 4) circle(0.2cm);
\draw[fill=black] (1, 4) circle(0.2cm);
\node[black] at (-2, 2.5) {$1$};
\node[black] at (-1, 4.5) {$2$};
\node[black] at (1, 4.5) {$3$};
\node[black] at (2, 2.5) {$4$};
\draw[blue, <->, thick, dashed] (-1.8, 2.7) arc (124:56:3.24);
\end{tikzpicture}
\end{center}
\end{minipage}
\begin{minipage}{.33\textwidth}
\centerline{\includegraphics[width=\textwidth]{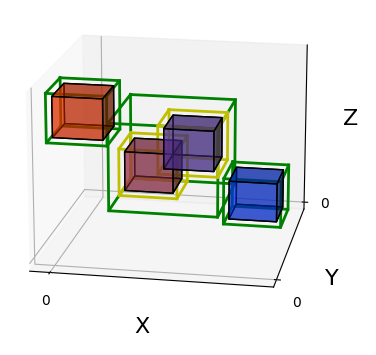}}
\end{minipage}%
\begin{minipage}{.33\textwidth}
\begin{center}
\begin{tikzpicture}[scale=0.85]
\draw[thick] (0, 0) -- (-2, 2);
\draw[thick] (0, 0) -- (0, 2);
\draw[thick] (0, 0) -- (2, 2);
\draw[thick] (0, 2) -- (1, 4);
\draw[thick] (0, 2) -- (-1, 4);
\draw[fill=white] (0, 0) circle(0.6cm);
\node at (0, 0) {\small$(+, -)$};
\draw[fill=black] (-2, 2) circle(0.2cm);
\draw[fill=white] (0, 2) circle(0.6cm);
\draw[fill=black] (2, 2) circle(0.2cm);
\node at (0, 2) {\small$(1, 1)$};
\draw[fill=black] (-1, 4) circle(0.2cm);
\draw[fill=black] (1, 4) circle(0.2cm);
\node[black] at (-2, 2.5) {$1$};
\node[black] at (-1, 4.5) {$2$};
\node[black] at (1, 4.5) {$3$};
\node[black] at (2, 2.5) {$4$};
\end{tikzpicture}
\end{center}
\end{minipage}
\caption{\textbf{Left}: The sign tree $S(\sigma)$ encoding the block structure of the separable $3$-permutation $\sigma = ((1,3,2,4), (4,2,3,1))$. The dashed blue arrow indicates that when flipping the order of the root's children to determine $\sigma^{(2)}$, vertex $1$ is now last and vertex $4$ is first. \textbf{Middle}: The $3$-dimensional diagram of the 3-permutation $\sigma$, with box decomposition indicating the block structure (dark green corresponding to $(+, -)$ and yellow corresponding to $(-, +)$). \textbf{Right}: The swap tree $\text{Sw}(\sigma)$ encoding the permutation $\sigma$.}
\label{fig:separable}
\end{figure}

We now describe the inverse process for recovering the original $d$-separable permutation from a tree of this form. Given a signs-labeled rooted plane tree $S(\sigma)$, we recover $\sigma^{(k)}$ for any $1 \le k \le d-1$ via the following process:
\begin{itemize}
\item Number the leaves in the $S(\sigma)$ from $1$ to $n$ in order from left to right.
\item For each internal vertex $v$ of the tree, if the $k$--th sign in the sign sequence of $v$ is $-1$, then flip the order of the children of $v$. That is, if $v$ previously had children $w_1, \cdots, w_N$ from left to right, then reorder them so that the children (together with their subtree) are $w_N, \cdots, w_1$ from left to right but keeping the subtrees rooted at $w_1, \cdots, w_N$ oriented in the same way. (Note that this procedure yields the same result regardless of the order that we perform the operations on internal vertices.)
\item Finally, $\sigma(1)^{(k)}, \cdots, \sigma(n)^{(k)}$ is the sequence of numbers for the leaves of the reordered tree from left to right.
\end{itemize}

For example in \cref{fig:separable}, $\sigma^{(2)} = (4, 2, 3, 1)$ because we reorder the children of the root (so that leaf $4$ is on the left and leaf $1$ is on the right) but not the children of the other internal vertex (so that leaf $2$ is still to the left of leaf $3$). 

\begin{proposition}\label{dsep-bijection}
The construction above yields a bijection $\sigma\mapsto S(\sigma)$ between $d$-separable permutations of size $n$ and rooted plane trees with $n$ leaves whose internal vertices all satisfy the following conditions:
\begin{itemize}
\item they are labeled by a sign sequence of length $(d-1)$,
\item they have at least two children, and
\item their sign sequence is different from that of their parent (if a parent exists).
\end{itemize}
Call such trees \vocab{sign trees} (of size $n$).
\end{proposition}
\begin{proof}
Since the primary block structure of any $d$-separable permutation is unique, the procedure for generating the tree is inductively well-defined. In any such rooted plane tree resulting from a $d$-separable permutation of size $n$, there are $n$ leaves (corresponding to the $n$ trivial permutations after the decomposition is finished), all internal vertices have degree at least $2$ (because the block sum $\sigma = \sigma_1 \encircle{s} \cdots \encircle{s} \sigma_N$ has at least two summands), and vertices have different labels from their parents because each $\sigma_i$ cannot be further written as a block sum using $\encircle{s}$, meaning that each vertex is either trivial (corresponding to a leaf) or of a different primary block structure  from its parent (hence a different label). Thus the constructions do map between the two sets of objects in the proposition statement. The fact that the map is a bijection easily follows from the above description of the inverse process.
\end{proof}

\begin{remark}\label{signtreeinverse}
Since we have a bijection, we may denote the inverse map $S(\sigma) \mapsto \sigma$ by $S^{-1}$. Note that if $T$ is a tree satisfying the first two conditions of \cref{dsep-bijection} but not necessarily the third, then $S^{-1}(T)$ is still well-defined and still yields a valid $d$-separable permutation, since this inverse construction is still taking a sequence of block sums of $d$-permutations of size $1$. This will be useful later when using these trees to determine patterns.
\end{remark}

Towards determining the limiting permuton of uniform $d$-separable permutations as $n \to \infty$, it is useful to obtain some independence by slightly transforming our encoding sign trees, removing the constraint that internal vertices have different labels from their parents. We do so with an alternate labeling:

\begin{definition}\label{defn:swap-seq}
A \vocab{swap sequence} of length $m$ is an element of the set $\{0, 1\}^{m} \setminus \{0^{m}\}$. 
\end{definition}

We now take the trees from \cref{dsep-bijection} and relabel them in the following alternate way. In a tree $S(\sigma)$, each \emph{non-root} internal vertex $v$ is labeled with a sign sequence $s_v$ of length $d-1$ which differs from its parents $p$'s sign sequence $s_p$. We will thus define a new tree which we call $\text{Sw}(\sigma)$ by replacing the label of each non-root internal vertex with a swap sequence. Specifically, we let the label at vertex $v$ be the binary string of length $d-1$ whose $j$--th coordinate is $0$ if $s_v^{(j)} = s_p^{(j)}$ and $1$ otherwise. We leave the label of the root unchanged, so that the root is still labeled with a sign sequence but all other vertices are labeled with a swap sequence. An example is shown in the right-hand side of \cref{fig:separable}.

Because the sign sequences $s_v$ and $s_p$ always differ in at least one coordinate, the label of $v$ in $\text{Sw}(\sigma)$ will be a swap sequence. And furthermore, given a tree labeled with swap sequences at all non-root internal vertices, we can uniquely determine the original sign sequences by exploring the tree from the root down to the leaves. Combined with \cref{dsep-bijection}, this thus yields the following result:

\begin{proposition}\label{dsep-bijection2}
The map $\sigma \mapsto \text{Sw}(\sigma)$ yields a bijection between $d$-separable permutations of size $n$ and rooted swap-labeled plane trees with $n$ leaves satisfying the following conditions:
\begin{itemize}
\item each internal vertex has at least two children,
\item the root is labeled by a sign sequence of length $(d-1)$, and
\item each other internal vertex is labeled with a swap sequence of length $(d-1)$.
\end{itemize}
Call such trees \vocab{swap trees} (of size $n$). 
\end{proposition}

Thus, to sample a uniform $d$-separable permutation of size $n$, we may sample such a swap tree with $n$ leaves uniformly at random and then apply the bijection in \cref{dsep-bijection2}. Since the swap tree and sign tree associated with a $d$-separable permutation have the same skeleton, we may refer to a vertex as being labeled with both a swap sequence and a sign sequence. 

We conclude this subsection by showing that the tree representations $S(\sigma)$ of a $d$-separable permutation $\sigma$ work well with reading off patterns:

\begin{lemma}\label{swaptreepattern}
Let $\sigma$ be a $d$-separable permutation of size $n$, and let $I = \{i_1 < i_2 < \cdots < i_k\} \subset \binom{[n]}{k}$. Then the pattern of $\sigma$ on $I$ can be obtained by constructing the subtree of $S(\sigma)$ induced by the leaves in $I$ as follows: construct the rooted plane subtree $T$ of $S(\sigma)$ whose vertices are the leaves in $S(\sigma)$ labeled with $I$ (in clockwise order) and the \vocab{closest common ancestor} of each pair of those leaves (that is, the first common {vertices} on the paths from those leaves to the root). Label the leaves of $T$ by $1, 2, \cdots, k$ in clockwise order instead of $i_1, \cdots, i_k$, and for each internal vertex $v$, label $v$ with its corresponding sign in $S(\sigma)$. The pattern on $I$ is then $S^{-1}(T)$ (here recall the discussion in \cref{signtreeinverse}).
\end{lemma}

For example, suppose we are given the swap tree in the right diagram of \cref{fig:separable} and wish to determine the pattern on the indices $\{2, 4\}$. We see that the closest common ancestor is the root, which is labeled $(+, -)$. Thus the pattern of $((1, 3, 2, 4), (4, 2, 3, 1))$ on those indices is $((1, 2), (2, 1))$. Similarly, if we want the pattern on indices $\{2, 3\}$, the closest common ancestor is the other internal vertex of the tree. After converting its label back to a sign $(-, +)$, we find that the pattern on those indices is $((2, 1), (1, 2))$. Importantly, while we only need to consider the labels of the \emph{sign} tree at all closest common ancestors, those labels do still depend on the labels of all \emph{swap} sequences between those closest common ancestors and the root of $\text{Sw}(\sigma)$.

\begin{proof}
To determine the pattern of $\sigma$ on $I$, it suffices to determine for each coordinate $1 \le j \le d-1$ and pair of indices $1 \le a < b \le k$ whether $\sigma_{i_a}^{(j)} < \sigma_{i_b}^{(j)}$. Let $p$ be the closest common ancestor of leaves $i_a$ and $i_b$ in the original sign tree $S(\sigma)$. Recall from above the definition of $S^{-1}$ and the subsequent proof of \cref{dsep-bijection}. While determining the $j$--th coordinate of $\sigma$, leaf $i_a$ stays to the left of leaf $i_b$ if and only if the $j$--th sign of $p$'s label is $+1$ (since the reorderings at all other internal vertices do not change the relative order of $i_a$ and $i_b$). When we construct our subtree $T$ induced by the leaves in $I$, the leaves marked $a$ and $b$ again have common parent $p$ labeled by the same sign as in $S(\sigma)$, so the $k$-permutation $\tau = S^{-1}(T)$ will indeed have $\tau_a^{(j)} < \tau_b^{(j)}$ if and only if the condition above holds. This is exactly the definition of the pattern on the indices $I$.
\end{proof}

\begin{remark}\label{topologicalordersampling}
For later convenience, we describe more explicitly the process of converting a swap sequence tree $\widetilde{T}$ into its corresponding sign sequence tree $T$. Let $(v_1, v_2, \cdots, v_N)$ be an ordering of the vertices in $\widetilde{T}$ such that all descendants of a vertex come after that vertex in the ordering (in particular, $v_1$ must be the root). The root $v_1$ is still labeled with a sign sequence, and for all $i \ge 2$, we label $v_i$ with a sign sequence $s$ as follows. The parent of $v_i$ has already been labeled with some sign sequence $s'$ because it comes earlier in the ordering; now for each coordinate $j$, we set $s^{(j)}$ to be the same as $(s')^{(j)}$ if $v_i$'s swap sequence has a $0$ in the $j$--th coordinate and $s^{(j)}$ to be different from $(s')^{(j)}$ otherwise.

More generally, this implies that if $v$ and $w$ are two vertices in $\widetilde{T}$ such that $w$ is a descendant of $v$ and they are separated by a path of length $\ell$, the sign sequences at $v$ and $w$ are related as follows: look at the swap sequences on the path from $v$ to $w$, including $w$ but not $v$. If there are an even number of $1$s in the $j$--th coordinates of those sequences, then the sign sequences at $v$ and $w$ agree. Otherwise, they disagree.
\end{remark}

\subsection{Permuton convergence for $d$-separable permutations}

As with Schnyder wood permutations, convergence to a $d$-dimensional permuton for $d$-separable permutations will be shown via convergence of patterns. The advantage of introducing the bijection in \cref{dsep-bijection2} is that sampling uniform swap trees with a fixed number of leaves can be described in terms of sampling conditioned Galton-Watson trees.

\begin{proposition}\label{separablegaltonwatson}
Let $a = 2^{d-1} - 1$ be the number of swap sequences of length $(d-1)$ (recall \cref{defn:swap-seq}), and let $b = 1 - \sqrt{\frac{a}{a+1}}$. Consider the random variable $\xi$ defined by the probability mass function
\[
    \PP(\xi = r) = \begin{cases} ab^{r-1} & r \ge 2, \\ 1 - \sum_{r=2}^{\infty} ab^{r-1} & r = 0, \\ 0 & r = 1. \end{cases}
\]
Then viewing $\xi$ as the offspring distribution for a Galton-Watson tree, $\xi$ is critical (mean $1$), aperiodic, and has finite variance. Furthermore, we may sample a uniform swap tree with $n$ leaves in the following manner:
\begin{enumerate}
    \item Sample a Galton-Watson tree with offspring distribution $\xi$, conditioned to have $n$ leaves.
    \item Label the root uniformly by one of the $2^d$ sign sequences uniformly at random.
    \item Label each internal vertex with one of the $a = 2^{d-1} - 1$ swap sequences uniformly at random.
\end{enumerate}
\end{proposition}
\begin{proof}
We first verify the properties of the offspring distribution $\xi$. Since $\PP(\xi = 2)$ and $\PP(\xi = 3)$ are nonzero, the distribution is aperiodic. Also, $\PP(\xi = 0)$ is indeed positive because 
\[
    \sum_{r=2}^{\infty} ab^{r-1} = \frac{ab}{1 - b} = \sqrt{a(a+1)} - a,
\]
so this is a valid probability distribution. We have
\[
    \EE[\xi] = \sum_{r=2}^{\infty} rab^{r-1} = a \frac{b(2-b)}{(1-b)^2},
\]
and by the definition of $b$ this simplifies to $a \cdot \frac{1}{a} = 1$, so the distribution is critical. Finally, we have
\[
    \EE[\xi^2] = \sum_{r=2}^{\infty} r^2ab^{r-1} = \frac{ab(b^2 - 3b + 4)}{(1-b)^3},
\]
which is again finite because $b < 1$ and thus $\text{Var}(\xi) = \EE[\xi^2] - \EE[\xi]^2$ is finite.

Next, note that the sampling process we describe always satisfies the three conditions of a swap tree from \cref{dsep-bijection}, and so it remains to show that for any swap tree $T$ of size $n$, the probability of sampling $T$ with this process is the same. Indeed, if $T$ has $m$ internal vertices (including the root) $v_1, \cdots, v_m$ with $d_1, \cdots, d_m$ descendants respectively (note that $d_i \ge 2$ for all $i$), then the probability that it is sampled is
\[
    \left(\PP(\xi = 0)^n \cdot \prod_{i=1}^m ab^{d_i - 1}\right) \cdot \frac{1}{2^d} \cdot \frac{1}{a^m} 
\]
where the term in parentheses comes from the product of the probabilities at each vertex in the tree to have the specified number of offspring, $\frac{1}{2^d}$ is the probability of any particular root label, and $\frac{1}{a}$ is the probability of any particular swap sequence at each of the $m$ internal vertices. But this expression simplifies to $\frac{1}{2^d} \PP(\xi = 0)^n b^{d_1 + \cdots + d_m - m}$, and $d_1 + \cdots + d_m$ counts each non-root vertex of the tree once (meaning it is equal to $m + n - 1$). Thus the probability of sampling $T$ is only a function of $d$ and $n$, meaning this process does indeed yield a uniform swap tree, as desired.
\end{proof}

This procedure allows us to make use of limit results for Galton-Watson trees, and we are now almost ready for the proof of our main result. First, we establish a concrete description of the limiting pattern frequencies for the Brownian separable $d$-permuton $\mu^B_{p_1, \cdots, p_{d-1}}$:

\begin{proposition}\label{prop:pattern-sep-perm}
Let $\mu^B_{p_1, \cdots, p_{d-1}}$ be the Brownian separable $d$-permuton  of parameters $(p_1, \cdots, p_{d-1}) \in [0,1]^{d-1}$ from \cref{brownian-separable-d-permuton}. Then, $P_{\mu^B_{p_1, \cdots, p_{d-1}}}[k]$ from \cref{defnsamplefrompermuton} has the same distribution as the $d$-permutation $\rho_k$ sampled in the following way:
\begin{itemize}
\item Sample a uniform rooted binary plane tree with $k$ leaves.
\item Label each internal vertex independently with a collection of signs $s=(s_1, \cdots, s_{d-1})$, where $s_i$ is $+1$ with probability $p_i$ and $-1$ with probability $1 - p_i$. This yields a tree $T$ satisfying the first two conditions of \cref{dsep-bijection}.
\item Set $\rho_k = S^{-1}(T)$ (see \cref{signtreeinverse} for why this is well-defined).
\end{itemize}
\end{proposition}

\begin{example}
\cref{prop:pattern-sep-perm} tells us that $P_{\mu^B_{p_1, p_2}}[2]$ is distributed as the $3$-permutation $\rho_2$ of size $2$ described in the proposition statement. Note that the only rooted binary plane tree with $2$ leaves is \begin{tikzpicture}[scale=0.3]
\draw[fill=black] (0, 0) circle(0.2cm);
\draw[fill=black] (-0.5, 1) circle(0.2cm);
\draw[fill=black] (0.5, 1) circle(0.2cm);
\draw (-0.5, 1) -- (0, 0) -- (0.5, 1);
\end{tikzpicture}, so the probabilities of $\rho_2$ being equal to each permutation of size $2$ depend only on the probabilities of different labels appearing at the root. Specifically, $\PP(\rho_2 = ((1, 2), (1, 2)))$, $\PP(\rho_2 = ((1, 2), (2, 1)))$, $\PP(\rho_2 = ((2, 1), (1, 2)))$, and $\PP(\rho_2 = ((2, 1), (2, 1)))$ are exactly the probabilities that the root is labeled $(+, +), (+, -), (-, +)$, and $(-, -)$, respectively, which are $p_1p_2, p_1(1-p_2), (1-p_1)p_2$, and $(1-p_1)(1-p_2)$.

For a more complicated example, take $k = 3$. The size-$3$ pattern $((1,2,3), (1,3,2))$ only occurs in the above procedure if the sample binary tree with $3$ leaves is of shape \begin{tikzpicture}[scale=0.3]
\draw[fill=black] (0, 0) circle(0.2cm);
\draw[fill=black] (-0.5, 1) circle(0.2cm);
\draw[fill=black] (0.5, 1) circle(0.2cm);
\draw[fill=black] (1, 2) circle(0.2cm);
\draw[fill=black] (0, 2) circle(0.2cm);
\draw (-0.5, 1) -- (0, 0) -- (0.5, 1);
\draw (1, 2) -- (0.5, 1) -- (0, 2);
\end{tikzpicture}, the root is labeled $(+, +)$, and the other internal vertex is labeled $(+, -)$. Thus the probability that $\rho_3 = ((1,2,3), (1,3,2))$ is $\frac{1}{2} p_1^2 p_2(1-p_2)$. On the other hand, the size-$3$ pattern $((1, 2, 3), (1, 2, 3))$ can occur for either the binary tree \begin{tikzpicture}[scale=0.3]
\draw[fill=black] (0, 0) circle(0.2cm);
\draw[fill=black] (-0.5, 1) circle(0.2cm);
\draw[fill=black] (0.5, 1) circle(0.2cm);
\draw[fill=black] (1, 2) circle(0.2cm);
\draw[fill=black] (0, 2) circle(0.2cm);
\draw (-0.5, 1) -- (0, 0) -- (0.5, 1);
\draw (1, 2) -- (0.5, 1) -- (0, 2);
\end{tikzpicture} or the binary tree \begin{tikzpicture}[scale=0.3]
\draw[fill=black] (0, 0) circle(0.2cm);
\draw[fill=black] (-0.5, 1) circle(0.2cm);
\draw[fill=black] (0.5, 1) circle(0.2cm);
\draw[fill=black] (-1, 2) circle(0.2cm);
\draw[fill=black] (0, 2) circle(0.2cm);
\draw (-0.5, 1) -- (0, 0) -- (0.5, 1);
\draw (0, 2) -- (-0.5, 1) -- (-1, 2);
\end{tikzpicture}
as long as both internal vertices are labeled $(+, +)$, so the probability that $\rho_3 = ((1, 2, 3), (1, 2, 3))$ is $\frac{1}{2}p_1^2p_2^2 + \frac{1}{2}p_1^2p_2^2 = p_1^2p_2^2$.
\end{example}

\begin{proof}[Proof of \cref{prop:pattern-sep-perm}]
Recall that the sources of randomness for the Brownian separable $d$-permuton are the one-dimensional Brownian excursion $e(t)$, as well as the variables $s(\ell) \in \{\pm 1\}^d$ at each local minimum $\ell$ of $e$.

Recall also from \cref{defnsamplefrompermuton} that given the Brownian separable $d$-permuton $\mu:=\mu^B_{p_1, \cdots, p_{d-1}}$, the permutation $P_{\mu}[k]$ is obtained as follows: conditioning on $\mu$,  sample $k$ independent points $\vec{x}_1, \cdots, \vec{x}_k$ in $[0, 1]^d$ with distribution $\mu$, and let $P_\mu[k]=\left(P_\mu[k]^{(1)},\dots,P_\mu[k]^{(d-1)}\right)$ be the unique $d$-dimensional permutation of size $k$ in the same relative order as the points $\vec{x}_i$.

By {\cref{brownian-separable-d-permuton-eqn} in} the definition of the Brownian separable $d$-permuton, the previous permutation can be equivalently obtained as follows: sample $k$ independent uniform points $(v_i)_{i\in [k]}$ on $[0, 1]$, ordered so that $v_1 < \cdots < v_k$. Then, for all $j \in [d-1]$, $P_\mu[k]^{(j)}$ is the permutation induced by the order of the points $(v_i)_{i\leq k}$ with respect to the order $<_e^{(j)}$. Equivalently, for $1 \le a<b\le k$,
\begin{eqnarray*}
	P_\mu[k]^{(j)}(a)<P_\mu[k]^{(j)}(b) \quad \text{if and only if} \quad s_{j}(v_a,v_b)=+1,
\end{eqnarray*}
where $\ell_{v_a,v_b}\in[v_a,v_b]$ is a.s.\ the unique location of the local minimum $\min_{[v_a,v_b]}e$.

We now define a random planar tree $\Tree\left(e, (v_1, \cdots, v_k)\right)$ with $k$ leaves corresponding to $v_1, \dots, v_k$ and internal vertices the local minima on intervals $[v_i, v_j]$ via the following recursive construction. If $k = 1$, then the tree is just a single leaf. Otherwise, $e$ almost surely achieves its unique minimum on $[v_1, v_k]$ at some point $m \in [v_p, v_{p+1}]$ for some $p$. We then define $\Tree\left(e, (v_1, \cdots, v_k)\right)$ to have a root with children $\Tree\left(e, (v_1, \cdots, v_p)\right)$ and $\Tree\left(e, (v_{p+1}, \cdots, v_k)\right)$ in that order.

As proved in \cite[Section 2.6]{legall} and rephrased in \cite[Lemma A.3]{separablelimit}, the distribution of $\widetilde{T}=\Tree\left(e, (v_1, \cdots, v_k)\right)$ is then uniform on the set of binary plane trees with $k$ leaves. Each internal vertex $v$ of $\widetilde{T}$ corresponds to a local minimum $m_{v}$ of $e$, and we then label each such $v$ with its sign sequence $s(m_{v})$. Then $S^{-1}(\widetilde{T})$ is clearly equal to $P_{\mu^B_{p_1, \cdots, p_{d-1}}}[k]$ thanks to the discussion in the previous paragraph and the description of $S^{-1}$ before \cref{dsep-bijection}. Since $\widetilde{T}$ has the same distribution of the tree $T$ in the proposition statement, we can conclude the proof.
\end{proof}

We may now prove the main result of this section.

\begin{proof}[Proof of \cref{mainresult4}]
Thanks to \cref{tfaemain}, it is enough to prove that, for each $k$, a uniform size-$k$ pattern of the uniform $d$-separable permutation $\sigma_n$ converges in distribution to $\rho_k$ as defined in \cref{prop:pattern-sep-perm} with all $p_i = \frac{1}{2}$.

Recall that by \cref{dsep-bijection2}, we may sample $\sigma_n$ by sampling a uniform swap tree of size $n$ and then applying the inverse map $\text{Sw}^{-1}$. Moreover, \cref{separablegaltonwatson} tells us that a uniform swap tree can be obtained by conditioning a Galton-Watson tree (with offspring distribution as in the proposition statement) to have $n$ leaves, then labeling the root with a uniform sign sequence and independently labeling all other internal vertices with a uniform swap sequence. Call this tree $T_n$.

It is well-known that if we choose a uniform subset $I$ of $k$ leaves of $T_n$ (as is needed to sample a uniform pattern of $\sigma_n$ as in \cref{swaptreepattern}) and consider the subtree $T_{n, I}$ induced by those leaves, i.e. the tree formed by the leaves in $I$ and their closest common ancestors, then with high probability (i.e.\ with probability tending to $1$ as $n \to \infty$) $T_{n, I}$ is a uniform binary plane tree with $k$ leaves and an extra vertex (corresponding to the root), and additionally all vertices in $T_n$ corresponding to the vertices in $T_{n,I}$ are at distance at least $n^{1/4}$. (See \cite[Section 3.3]{aldouscrt2} for a statement about heights of uniformly chosen vertices and \cite[Section 4]{Borga_2020} for the more refined result.) Thus, from here on, we may condition on this high probability event.

\medskip

The key point now is that in our sampling procedure for $T_n$, given the shape of the tree, we picked an independent uniform \emph{swap} sequence for each internal vertex. In contrast, in the (limiting) sampling procedure in \cref{prop:pattern-sep-perm} for sampling patterns, we instead pick an independent uniform \emph{sign} sequence. However, since we condition on all relevant distances being at least $n^{1/4}$, we claim the signs of our internal vertices of $T_{n, I}$ do become asymptotically uniform and independent. This claim is a simple consequence of the fact that in the procedure described in \cref{topologicalordersampling} to determine the signs of the internal vertices of $T_n$ corresponding to the internal vertices of $T_{n, I}$ from the swap sequences in $T_n$, the only relevant quantities are the parities of the number of $1$s in the swap sequences between such internal vertices of $T_n$. Such parities are asymptotically uniform and independent because the distances between those vertices are at least $n^{1/4}$.

We conclude that $T_{n, I}$ is with high probability a uniform binary plane tree with $k$ leaves with an extra vertex (corresponding to the root), and for each of those finitely many possible tree shapes, as $n \to \infty$ the internal sign sequences are asymptotically iid uniform sign sequences. Thus our random pattern is indeed equal in distribution to that of \cref{prop:pattern-sep-perm} with all $p_i = \frac{1}{2}$, concluding the proof.
\end{proof}

We conclude by showing that the Brownian separable $d$-permuton $\mu^B_{p_1, \cdots, p_{d-1}}$ is not a ``degenerate limit,'' in contrast to the Schnyder wood permuton (recall \cref{marginaldetermines}):

\begin{proposition}\label{prop:non-triviality}
Let $p_1, \cdots, p_{d-1} \in (0, 1)$, and let $S \subsetneq [d]$ be any proper subset of the coordinates which includes $1$. Then the marginal of $\mu^B_{p_1, \cdots, p_{d-1}}$ on the coordinate set $S$ does not determine the full law of the permuton.
\end{proposition}

This result can be thought of as the continuum analog of the fact that a random $d$-dimensional permutation $\sigma = (\sigma^{(1)}, \cdots, \sigma^{(d-1)})$ is  not specified by a strict subset of the $\sigma^{(i)}$s.

\begin{proof}
For each such subset $S$, we construct two different copies $\mu, \overline{\mu}$ of $\mu^B_{p_1, \cdots, p_{d-1}}$ whose marginals on the coordinate set $S$ are exactly identical. Let $\mu, \overline{\mu}$  be two Brownian separable $d$-permutons constructed as in \cref{brownian-separable-d-permuton} using the same Brownian excursion $e$ and with coupled sign sequences $s(\ell)$ and $\overline{s}(\ell)$ obtained as follows: let $s_{i-1}(\ell) = \overline{s}_{i-1}(\ell)$ if $i \in S \setminus \{1\}$, and otherwise sample $s_{i-1}(\ell)$ and $\overline{s}_{i-1}(\ell)$ independently from each other (each $+1$ with probability $p_{i-1}$ and $-1$ with probability $1 - p_{i-1}$).

Recalling from \cref{brownian-separable-d-permuton-eqn} that the marginal of the Brownian separable permuton on the coordinate set $S$ depends only on the random functions $\psi_e^{(i - 1)}(t)$ for $i \in S \setminus \{1\}$, which in turn only depend on $e$ and on the signs $(s_{i-1}(\ell))$, the marginals of $\mu$ and $\overline{\mu}$ on $S$ are identical. However, the full permutons $\mu$ and $\overline{\mu}$ are different with positive probability. Indeed, consider any local minimum $\ell$ whose $x$-coordinate is not in $\{0, \frac{1}{2}, 1\}$ and any $i \not\in S$. The probability that $s_{i-1}(\ell) = 1$ and $\overline{s}_{i-1}(\ell) = -1$ is $p_{i-1}(1 - p_{i-1}) > 0$. And on this event, the two permutons are different, because the $(1, i)$ marginal of $\mu$ is supported on $[0, \ell_x]^2 \cup [\ell_x, 1]^2$, while the $(1, i)$ marginal of $\overline{\mu}$ is supported on $[0, \ell_x] \times [1 - \ell_x, 1] \cup [\ell_x, 1] \times [0, 1 - \ell_x]$; these supports are disjoint on either $[0, \ell_x] \times [0, 1]$ or $[\ell_x, 1] \times [0, 1]$, but the $(1, i)$ marginal is a $2$-permuton and hence must have positive measure on both of those sets. Thus $\mu$ and $\overline{\mu}$ are indeed different probability measures on $[0, 1]^d$, as desired.
\end{proof}

\appendix

\section{Proofs of the combinatorial results for Schnyder woods}\label{schnydercombinatoricssectionproofs}

This appendix is dedicated to proving \cref{schnyderwoodbijectperm} and \cref{mainresult1}. The latter proof involves some additional combinatorial constructions, in which we define preliminary versions (denoted by $pZ^g_M$ and $pZ^r_M$) of the green and red coalescent-walk processes $Z^g_M$ and $Z^r_M$. 

Recall that each vertex of a Schnyder wood is assigned a blue, green, and red label between $1$ and $n$ inclusive. Unless otherwise specified, vertices referred to by number are always done so by their blue label.

\medskip

We first give the missing proof of \cref{schnyderwoodbijectperm}.

\begin{proof}[Proof of \cref{schnyderwoodbijectperm}] {It suffices to prove that if $\sigma = \sigma_M$ for some Schnyder wood $M$, then $\sigma$ determines the shapes of the green and blue trees of $M$.} Indeed, \cite[Section 3]{bonichonpaths}\footnote{To match notation with our paper, a Schnyder wood triangulation denoted $(T_0, T_1, T_2)$ in \cite{bonichonpaths} has green spanning tree $T_0$, blue spanning tree $T_1$, and red spanning tree $T_2$.} describes an algorithm that bijectively maps each Schnyder wood triangulation $M$ of size $n$ to a \vocab{star realizer} (that is, a Schnyder wood triangulation whose red spanning tree is a star centered at the red root) together with a sequence of \vocab{prefix flips} (which are local moves that turn one Schnyder wood triangulation into another one of the same size). In particular, these prefix flips do not change the shape of the green tree. Thus, once we determine the shape of the green tree, we determine uniquely the star realizer $M'$ of our Schnyder wood triangulation by \cite[Proposition 6]{bonichonpaths}. Furthermore, given the shape of the blue tree of $M$, we may determine the sequence of prefix flips via \cite[Proposition 10]{bonichonpaths} by computing the differences in children counts between the vertices in the green trees of $M$ and $M'$ of the same green label. Then we may apply the prefix flips to $M'$ in the order specified by the algorithm, uniquely determining $M$ in terms of $\sigma$. 

\medskip

{Proceeding with the proof, it suffices to show that $\sigma^g = \sigma^{(1)}$ determines the shape of the green spanning tree. Once we do this, because $\sigma(i)^{(1)}$ and $\sigma(i)^{(2)}$ are the green and red labels of the vertex in $M$ with blue label $i$ (by \cref{schnyderpermutationdefn}), we may cyclically permute the roles of the three colors by defining a $3$-permutation $\tau$ where $\tau(i)^{(1)}$ and $\tau(i)^{(2)}$ are the blue and green labels of the vertex in $M$ with red label $i$ (which can be obtained from $\sigma$). Repeating the argument with the cyclically permuted colors then shows that $\tau^{(1)}$ determines the shape of the blue spanning tree, which will complete the proof.} 

For this, we now state and prove two key facts. 

\medskip

\noindent\textbf{Fact 1}: If $v_1$ and $v_2$ are children of the same vertex in the green tree of a Schnyder wood triangulation with $v_1$ coming before $v_2$ in the green clockwise traversal, then the (blue) label of $v_1$ is always larger than the label of $v_2$. 

\medskip

(For example, the green children of vertex $1$ in \cref{fig:sampleschnyder} {have labels} $5, 3, 2$ in that clockwise order.) Indeed, consider the region $\mc{R}_{v_1}$ bounded by (1) the edge between the blue and green root, (2) the green path from $v_1$ to the green root, and (3) the blue path from $v_1$ to the blue root. Then $v_2$ must be either inside or on the boundary of $\mc{R}_{v_1}$, and all of these vertices are visited before $v_1$ in the blue tree's traversal, proving the claim. 

\medskip

\noindent\textbf{Fact 2}: If $w$ is a child of $v$ in the green tree, then the (blue) label of $w$ is larger than the label of $v$. 

\medskip

Indeed, this means $v$ is on the green path from $w$ to the green root, meaning in particular that it is on the boundary of $\mc{R}_w$, so $v$ will be visited before $w$ in the blue tree's traversal.

\medskip

Therefore {(recall now the interpretation of $(\sigma^g)^{-1}$ from \cref{inverseofschnyder})}, {since we know $\sigma$, we know the blue} labels {$(\sigma^g)^{-1}(1), \cdots, (\sigma^g)^{-1}(n)$} of the vertices {reached by} the green tree's clockwise traversal, in order. The shape of the {green} tree {can then be fully specified} by describing the {green} parents of these $n$ vertices{, and we can do so inductively as follows.}

The vertex with label ${(\sigma^g)^{-1}}(1)$ is visited first, {so its (green) parent must be the green root.} Now if we have already determined the parents of the first $i$ vertices (with labels ${(\sigma^g)^{-1}}(1), \cdots, {(\sigma^g)^{-1}}(i)$), then the parent of the vertex labeled ${(\sigma^g)^{-1}}(i+1)$ must be one of the vertices on the green path connecting the green root to the vertex labeled ${(\sigma^g)^{-1}}(i)$. But the labels along this path form an increasing sequence of integers, say $a_1 < \cdots < a_k = {(\sigma^g)^{-1}}(i)$, and now we consider different cases according to the value of ${(\sigma^g)^{-1}}(i+1)$ (which is some integer not equal to any of $a_1, \cdots, a_k$):
\begin{itemize}
    \item If ${(\sigma^g)^{-1}}(i+1) < a_1$, then by Fact $2$ the parent of our new vertex can only be the green root. 
    \item If ${(\sigma^g)^{-1}}(i+1) > a_k$, then by Fact $1$ the parent of our new vertex can only be the vertex labeled ${(\sigma^g)^{-1}}(i)$.
    \item Finally, if $a_j < {(\sigma^g)^{-1}}(i+1) < a_{j+1}$ for some $1 \le j \le k-1$, then by Facts $1$ and $2$ the parent can only be the vertex labeled $a_j$.
\end{itemize}
Thus the permutation $\sigma$ provides a systematic method for determining the shape of the green tree, {and we have already shown that applying this method twice is enough to determine the full Schnyder wood $M$, so the proof is complete.}
\end{proof}

We now turn to the proof of \cref{mainresult1}. We begin with the following definition.

\begin{definition}\label{prerandomwalk}
Let $M$ be a Schnyder wood triangulation of size $n$. The \vocab{pre-random walk} $pW_M = (pX_M, pY_M)$ associated to $M$ is the two-dimensional random walk which replaces \texttt{g}s, \texttt{b}s, and \texttt{r}s in the Schnyder wood string $s_M$ with increments of $(0, 1)$, $(1, -1)$, and $(-1, 0)$, respectively.
\end{definition}

Much like $Z^g_M$ and $Z^r_M$ are driven by $W_M'$ and $W_M$, our preliminary coalescent-walk processes will be driven by $pW_M'$ and $pW_M$, where $pW_M'$ is the time-reversal of $pW_M$. (Recall that $W_M$ takes each \texttt{rr}$\cdots$\texttt{rrg} segment and associates to it an increment of $(-k, 1)$, which is exactly the total increment taken by the pre-random walk over this segment.) We now describe the rules for these two new coalescent-walk processes:

\begin{figure}[ht]
\centerline{\includegraphics[width=\textwidth]{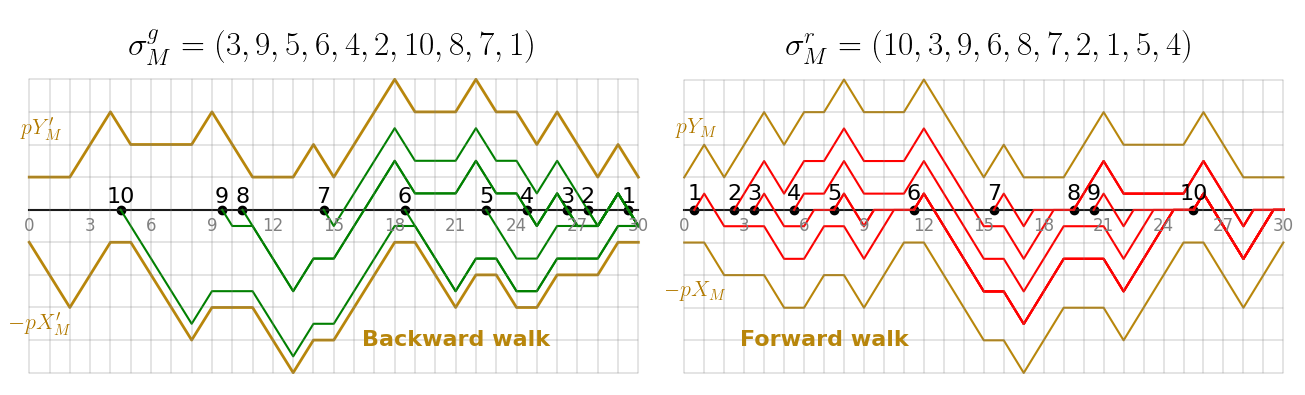}}
\caption{The pre-green and pre-red coalescent-walk processes associated to the Schnyder wood triangulation in \cref{fig:sampleschnyder}. Both are defined on the interval $[0, 30]$. Again, the black numbers (above the $x$-axis) are the labels of the paths started at those points, while the gray numbers are evenly spaced $x$-coordinates.}\label{fig:sampleschnyder3}
\end{figure}

\begin{definition}\label{pregreen}
Let $W = (X_n, Y_n)$ be a random walk on an interval $I = [a, b]$ whose increments all lie in the set $\{(0, -1), (-1, 1), (1, 0)\}$. The \vocab{pre-green coalescent-walk process associated to $W$}, denoted $\pWC^g(W)$, is the process $Z$ defined as follows. $Z$ is a collection of sample paths (like any other coalescent-walk process), but its starting points are the midpoints of the starting and ending times corresponding to the $(0, -1)$ increments -- that is, 
\[
    J = \left\{j+1/2 \in I \setminus \{b\}: Y_{j+1} - Y_j = -1\right\}.
\]
Furthermore, each sample path $Z^{(j)}$ is defined at all subsequent integer and half-integer points in $I$. We have $Z^{(j)}_j = 0$ for all $j \in J$, and for any $s \ge j$ such that $s \in I$ (resp.\ $s + \frac{1}{2} \in I$), the increment $Z_{s+\frac{1}{2}}^{(j)} - Z_s^{(j)}$ is determined as follows. Suppose the increment of $W$ on $[s, s+1]$ (resp.\ $[s-1/2, s+1/2]$) is $(x, y)$. If $Z_s^{(j)} \ge 0$, then $Z_{s+\frac{1}{2}}^{(j)} - Z_s^{(j)} = \frac{y}{2}$. Otherwise, $Z_{s+\frac{1}{2}}^{(j)} - Z_s^{(j)} = -\frac{x}{2}$.
\end{definition}

The left-hand side of \cref{fig:sampleschnyder3} shows the pre-green coalescent-walk process $\pWC^g(pW_M')$. In words, each sample path of the process copies the (half-)increments of $pY_M'$ while at or above the $x$-axis, and it follows the (half)-increments of $-pX_M'$ otherwise, with no additional rules.

\begin{definition}\label{prered}
Let $W = (X_n, Y_n)$ be a random walk on an interval $I = [a, b]$ whose increments all lie in the set $\{(0, 1), (1, -1), (-1, 0)\}$. The \vocab{pre-red walk process associated to $W$}, denoted $\pWC^r(W)$, is the process $Z$ defined identically to \cref{pregreen} but with two changes:
\begin{itemize}
\item The starting points are the midpoints of the $(0, 1)$ increments, rather than the $(0, -1)$ increments.
\item If a path $Z^{(j)}_s$ is at height $0$ on an interval $[c, d]$ on which the increments of $W$ are all $(-1, 0)$, then not only do we have $Z^{(j)}_t = 0$ for all $t \in [c, d]$, but we also have $Z^{(j)}_{d + \frac{1}{2}} = 0$ (if we are still in the interval -- that is, if $d + \frac{1}{2} \le b$).
\end{itemize}
\end{definition}

The right-hand side of \cref{fig:sampleschnyder3} shows the pre-red coalescent-walk process $\pWC^r(pW_M)$. In words, this means that a path on the $x$-axis continues horizontally for an extra half-unit beyond the region of $(-1, 0)$ steps of $W$. For example, the walk started at the point labeled $5$ intersects the starting point labeled $6$, despite the increment of $W$ indicating that it would otherwise have an increment of $+\frac{1}{2}$.

Our next result is the main ingredient for the proof of \cref{mainresult1}, stating that Schnyder wood permutations also agree with the permutations from the pre-green and pre-red processes:

\begin{proposition}\label{mainresult2}
For any Schnyder wood triangulation $M$ of size $n$, let $pZ_M^g = \pWC^g(pW_M')$ and $pZ^r_M = \pWC^r(pW_M)$ be its pre-green and pre-red coalescent-walk processes. Then $pZ_M^g$ and $pZ_M^r$ are indeed coalescent-walk processes in the sense of \cref{coalescentdefn}, so \cref{totalorder,coalescentpermutation} may be applied to them. Label the starting points of $pZ_M^g$ in descending order $n, n-1, \cdots, 1$ from left to right, and label the starting points of $pZ_M^r$ in ascending order $1, 2, \cdots, n$. Then the Schnyder wood permutation $\sigma_M$ associated with $M$ satisfies 
\[
    \sigma_M = (\sigma^{\text{up}}(pZ^g_M), \sigma^{\text{down}}(pZ^r_M)).
\]
Additionally, the shapes of the red and green trees of $M$ may be read off from the processes just like in \cref{mainresult1}, but with $Z^g_M$ and $Z^r_M$ replaced with $pZ^g_M$ and $pZ^r_M$, respectively.
\end{proposition}

We will prove \cref{mainresult2} first -- working with the pre-green and pre-red processes is more convenient for the combinatorics than the green and red processes, because each increment of the random walk now corresponds to a specific edge traversal of the Schnyder wood triangulation. Our goal will be reached through a sequence of lemmas.

Recall how the edges of a Schnyder wood triangulation are oriented and how its vertices are enumerated from \cref{fig:sampleschnyder}. Also recall that we continue to refer to vertices of a Schnyder wood by their blue label.

\begin{lemma}\label{ordering}
Let $M$ be an arbitrary Schnyder wood triangulation of size $n$, and label the green and red roots $0$ and $n+1$ respectively. Then green edges in $M$ are always directed from larger to smaller numbered vertices, and red edges are always directed from smaller to larger numbered vertices.

Additionally, if two vertices $v_1, v_2$ in a Schnyder wood triangulation are such that one is a descendant of the other in the blue tree, then there are no green or red edges between $v_1$ and $v_2$.
\end{lemma}

\begin{proof}
There are no outgoing edges from the green or red root, and any incoming edges to those roots automatically satisfy the statement. Thus we only need to consider internal vertices for this lemma.

First, we do the green case. Suppose for the sake of contradiction that a green edge points from $s$ to $t$, where $1 \le s < t \le n$. Let $r$ be the common blue tree ancestor of $s$ and $t$ (that is, the first point of intersection on the blue paths from $s$ and $t$ to the blue root), and draw the region $\mc{R}$ of the Schnyder wood triangulation encircled by the green edge $s \to t$ and the blue paths from $s$ and $t$ to $r$. As shown in \cref{fig:schnyderlemma1} and \cref{fig:schnyderlemma2}, we must consider two possibilities.
\begin{itemize}
\item In case $1$, $r$ is neither $s$ nor $t$. This means that $t$ is on a later branch than $s$ in the blue-tree traversal (that is, $t$ shows up on a branch that is further clockwise). Edges may not cross each other in a Schnyder wood triangulation, and vertex $r$ is connected to the blue root by a sequence of further blue edges. Thus, the green edge from $s$ to $t$ must go clockwise around $r$ rather than counterclockwise, as shown in \cref{fig:schnyderlemma1}.

\begin{figure}[ht]
\begin{center}
\begin{tikzpicture}
\draw[blue, thick, ->] (2, 4) -- (1.5, 3);
\draw[blue, thick] (1.5, 3) -- (1, 2);
\draw[blue, thick, ->] (1, 2) -- (0.5, 1);
\draw[blue, thick] (0.5, 1) -- (0, 0);
\draw[blue, thick, ->] (-2, 4) -- (-1.75, 3.5);
\draw[blue, thick] (-1.75, 3.5) -- (-1.5, 3);
\draw[blue, thick, ->] (-1.5, 3) -- (-1, 2);
\draw[blue, thick] (-1, 2) -- (-0.5, 1);
\draw[blue, thick, ->] (-0.5, 1) -- (-0.25, 0.5);
\draw[blue, thick] (-0.25, 0.5) -- (0, 0);
\draw[blue, thick, ->] (0, 3) -- (0, 2.5);
\draw[blue, thick] (0, 2.5) -- (0, 2);
\draw[blue, thick, ->] (1, 3) -- (1, 2.5);
\draw[blue, thick] (1, 2.5) -- (1, 2);
\draw[blue, thick, ->] (0, 2) -- (0, 1);
\draw[blue, thick] (0, 1) -- (0, 0);
\draw[blue, thick, ->] (-1, 3) -- (-0.5, 2.5);
\draw[blue, thick] (-0.5, 2.5) -- (0, 2);
\draw[blue, fill=blue] (0, 0) circle(0.1cm);
\draw[blue, fill=blue] (1, 2) circle(0.1cm);
\draw[blue, fill=blue] (2, 4) circle(0.1cm);
\draw[blue, fill=blue] (-2, 4) circle(0.1cm);
\draw[blue, fill=blue] (-0.5, 1) circle(0.1cm);
\draw[blue, fill=blue] (-1.5, 3) circle(0.1cm);
\draw[blue, fill=blue] (1, 3) circle(0.1cm);
\draw[green!50!black, thick, ->] (-2, 4) -- (0, 4);
\draw[green!50!black, thick] (0, 4) -- (2, 4);
\draw[blue, fill=blue] (0, 2) circle(0.1cm);
\draw[blue, fill=blue] (0, 3) circle(0.1cm);
\draw[blue, fill=blue] (-1, 3) circle(0.1cm);
\node at (2.2, 4.2) {$t$};
\node at (-2.2, 4.2) {$s$};
\node at (0, -0.3) {$r$};

\begin{scope}[shift={(7,0)}]
\draw[fill=gray!40!white] (-0.5, 1) -- (-2, 4) -- (2, 4) -- (0, 3) -- (-1, 3) -- (-0.5, 1);
\draw[blue, thick, ->] (2, 4) -- (1.5, 3);
\draw[blue, thick] (1.5, 3) -- (1, 2);
\draw[blue, thick, ->] (1, 2) -- (0.5, 1);
\draw[blue, thick] (0.5, 1) -- (0, 0);
\draw[blue, thick, ->] (-2, 4) -- (-1.75, 3.5);
\draw[blue, thick] (-1.75, 3.5) -- (-1.5, 3);
\draw[blue, thick, ->] (-1.5, 3) -- (-1, 2);
\draw[blue, thick] (-1, 2) -- (-0.5, 1);
\draw[blue, thick, ->] (-0.5, 1) -- (-0.25, 0.5);
\draw[blue, thick] (-0.25, 0.5) -- (0, 0);
\draw[blue, thick, ->] (0, 3) -- (0, 2.5);
\draw[blue, thick] (0, 2.5) -- (0, 2);
\draw[blue, thick, ->] (1, 3) -- (1, 2.5);
\draw[blue, thick] (1, 2.5) -- (1, 2);
\draw[blue, thick, ->] (0, 2) -- (0, 1);
\draw[blue, thick] (0, 1) -- (0, 0);
\draw[blue, thick, ->] (-1, 3) -- (-0.5, 2.5);
\draw[blue, thick] (-0.5, 2.5) -- (0, 2);
\draw[blue, fill=blue] (0, 0) circle(0.1cm);
\draw[blue, fill=blue] (1, 2) circle(0.1cm);
\draw[blue, fill=blue] (2, 4) circle(0.1cm);
\draw[blue, fill=blue] (-2, 4) circle(0.1cm);
\draw[blue, fill=blue] (-0.5, 1) circle(0.1cm);
\draw[blue, fill=blue] (-1.5, 3) circle(0.1cm);
\draw[blue, fill=blue] (1, 3) circle(0.1cm);
\draw[green!50!black, thick, ->] (-2, 4) -- (0, 4);
\draw[green!50!black, thick, ->] (2, 4) -- (1, 3.5);
\draw[green!50!black, thick] (1, 3.5) -- (0, 3);
\draw[green!50!black, thick, ->] (0, 3) -- (-0.5, 3);
\draw[green!50!black, thick] (-0.5, 3) -- (-1, 3);
\draw[green!50!black, thick, ->] (-1, 3) -- (-0.75, 2);
\draw[green!50!black, thick] (-0.75, 2) -- (-0.5, 1);
\draw[green!50!black, thick] (0, 4) -- (2, 4);
\draw[blue, fill=blue] (0, 2) circle(0.1cm);
\draw[blue, fill=blue] (0, 3) circle(0.1cm);
\draw[blue, fill=blue] (-1, 3) circle(0.1cm);
\node at (2.2, 4.2) {$t$};
\node at (-2.2, 4.2) {$s$};
\node at (0, -0.3) {$r$};
\node at (-0.7, 0.8) {$t_1$};
\node at (-0.5, 3.5) {\Large$\mc{R}$};
\end{scope}

\end{tikzpicture}
\end{center}
\caption{A sample illustration of case 1 for the green edges. The region $\mc{R}$ is shaded in gray on the right.}\label{fig:schnyderlemma1}
\end{figure}
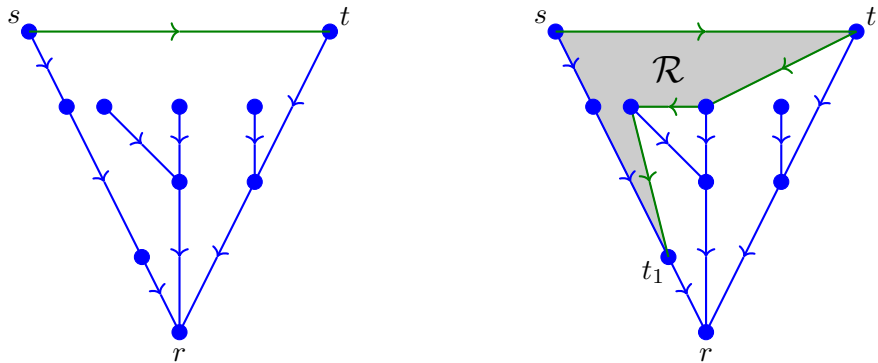

By Schnyder's rule at vertex $t$, the outgoing green edge from $t$ must point inside the enclosed region $\mc{R}$. If we continue following the outgoing green edge path starting from $t$, we must eventually terminate at the green root (because all green paths do), so this path must intersect some vertex on the boundary of the region. (One possibility is shown in the right image of \cref{fig:schnyderlemma1}.) Vertices $s$ and $t$ have already been visited (so the green path cannot hit them again because that would create a cycle), and exiting on the right boundary violates Schnyder's rule, so the green path must exit at some vertex $t_1$ strictly between $s$ and $r$.

Now by Schnyder's rule at vertex $t_1$, the red outgoing edge from $t_1$ must point inside the smaller shaded region (labeled as $\mc{R}$ in our figure) formed by the green and blue paths from $s$ to $t_1$. But then the red path originating from $t_1$ must exit this shaded region at some other vertex, and doing so will always violate Schnyder's rule at the point of exit. Thus we get a contradiction, and the green edge cannot point from $s$ to $t$ in this configuration. 

\item In case $2$, one of $s$ and $t$ is their common ancestor -- since $s < t$, it must be that $s$ is the ancestor of $t$. The green edge from $s$ to $t$ is then either to the right or to the left of the blue path between the two vertices:

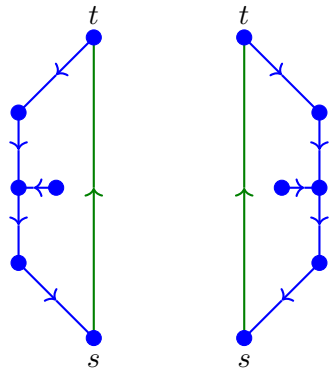
\begin{figure}[ht]
\begin{center}
\begin{tikzpicture}
\draw[blue, thick, ->] (0, 4) -- (-0.5, 3.5);
\draw[blue, thick] (-0.5, 3.5) -- (-1, 3);
\draw[blue, thick, ->] (-1, 3) -- (-1, 2.5);
\draw[blue, thick] (-1, 2.5) -- (-1, 2);
\draw[blue, thick, ->] (-1, 2) -- (-1, 1.5);
\draw[blue, thick] (-1, 1.5) -- (-1, 1);
\draw[blue, thick, ->] (-1, 1) -- (-0.5, 0.5);
\draw[blue, thick] (-0.5, 0.5) -- (0, 0);
\draw[blue, thick, ->] (-0.5, 2) -- (-0.8, 2);
\draw[blue, thick] (-0.8, 2) -- (-1, 2);
\draw[green!50!black, thick, ->] (0, 0) -- (0, 2);
\draw[green!50!black, thick] (0, 2) -- (0, 4);
\draw[blue, fill=blue] (0, 0) circle(0.1cm);
\draw[blue, fill=blue] (0, 4) circle(0.1cm);
\draw[blue, fill=blue] (-1, 3) circle(0.1cm);
\draw[blue, fill=blue] (-1, 2) circle(0.1cm);
\draw[blue, fill=blue] (-1, 1) circle(0.1cm);
\draw[blue, fill=blue] (-0.5, 2) circle(0.1cm);
\node at (0, 4.3) {$t$};
\node at (0, -0.3) {$s$};

\draw[blue, thick, ->] (2, 4) -- (2.5, 3.5);
\draw[blue, thick] (2.5, 3.5) -- (3, 3);
\draw[blue, thick, ->] (3, 3) -- (3, 2.5);
\draw[blue, thick] (3, 2.5) -- (3, 2);
\draw[blue, thick, ->] (3, 2) -- (3, 1.5);
\draw[blue, thick] (3, 1.5) -- (3, 1);
\draw[blue, thick, ->] (3, 1) -- (2.5, 0.5);
\draw[blue, thick] (2.5, 0.5) -- (2, 0);
\draw[blue, thick, ->] (2.5, 2) -- (2.8, 2);
\draw[blue, thick] (2.8, 2) -- (3, 2);
\draw[green!50!black, thick, ->] (2, 0) -- (2, 2);
\draw[green!50!black, thick] (2, 2) -- (2, 4);
\draw[blue, fill=blue] (2, 0) circle(0.1cm);
\draw[blue, fill=blue] (2, 4) circle(0.1cm);
\draw[blue, fill=blue] (3, 3) circle(0.1cm);
\draw[blue, fill=blue] (3, 2) circle(0.1cm);
\draw[blue, fill=blue] (3, 1) circle(0.1cm);
\draw[blue, fill=blue] (2.5, 2) circle(0.1cm);
\node at (2, 4.3) {$t$};
\node at (2, -0.3) {$s$};
\end{tikzpicture}
\end{center}
\caption{A sample illustration of case 2 for the green edges.}\label{fig:schnyderlemma2}
\end{figure}

However, the left diagram cannot occur -- applying Schnyder's rule at vertex $s$, the blue outgoing edge from $s$ must point inside the region (bounded by the blue and green paths between $s$ and $t$), and the continuation of that blue path cannot exit at any vertex on the boundary because that would create a blue cycle. And the right diagram also fails because applying Schnyder's rule at vertex $t$, the green outgoing edge from $t$ must point into the region, and anywhere it intersects the boundary will again break Schnyder's rule. Thus again we have a contradiction, meaning we cannot have an edge from $s$ to $t$ of this form either. 

Furthermore, it is not possible to have the green edge in \cref{fig:schnyderlemma2} point from $t$ to $s$ either. Again, this is because we would need to break Schnyder's rule for any configuration. In the left case, the red outgoing edge from $s$ would need to point into the region and cannot exit in any valid way, and in the right case, the blue outgoing edge from $s$ would need to point into the region and thus form a loop when it exits. Therefore we cannot have a green edge between any vertex and any of its blue ancestors.
\end{itemize}

Next, we do the red case. Suppose for the sake of contradiction that there is a red edge from $s$ to $t$ with $n \ge s > t \ge 1$ -- again we consider two possibilities. 

\begin{itemize}
\item In case $1$, again the common ancestor of $s$ and $t$ is some other vertex $r$. Then Schnyder's rule says that the outgoing red edge from $t$ must point inside this region, and then the only place where the subsequent red path may exit the region is some vertex $v_1$ on the right boundary strictly between $r$ and $s$:

\begin{figure}[ht]
\begin{center}
\begin{tikzpicture}
\draw[fill=gray!40!white] (1, 2) -- (2, 4) -- (-2, 4) -- (0, 3) -- (1, 2);
\draw[blue, thick, ->] (2, 4) -- (1.5, 3);
\draw[blue, thick] (1.5, 3) -- (1, 2);
\draw[blue, thick, ->] (1, 2) -- (0.5, 1);
\draw[blue, thick] (0.5, 1) -- (0, 0);
\draw[blue, thick, ->] (-2, 4) -- (-1.75, 3.5);
\draw[blue, thick] (-1.75, 3.5) -- (-1.5, 3);
\draw[blue, thick, ->] (-1.5, 3) -- (-1, 2);
\draw[blue, thick] (-1, 2) -- (-0.5, 1);
\draw[blue, thick, ->] (-0.5, 1) -- (-0.25, 0.5);
\draw[blue, thick] (-0.25, 0.5) -- (0, 0);
\draw[blue, thick, ->] (0, 3) -- (0, 2.5);
\draw[blue, thick] (0, 2.5) -- (0, 2);
\draw[blue, thick, ->] (1, 3) -- (1, 2.5);
\draw[blue, thick] (1, 2.5) -- (1, 2);
\draw[blue, thick, ->] (0, 2) -- (0, 1);
\draw[blue, thick] (0, 1) -- (0, 0);
\draw[blue, thick, ->] (-1, 3) -- (-0.5, 2.5);
\draw[blue, thick] (-0.5, 2.5) -- (0, 2);
\draw[blue, fill=blue] (0, 0) circle(0.1cm);
\draw[blue, fill=blue] (1, 2) circle(0.1cm);
\draw[blue, fill=blue] (2, 4) circle(0.1cm);
\draw[blue, fill=blue] (-2, 4) circle(0.1cm);
\draw[blue, fill=blue] (-0.5, 1) circle(0.1cm);
\draw[blue, fill=blue] (-1.5, 3) circle(0.1cm);
\draw[blue, fill=blue] (1, 3) circle(0.1cm);
\draw[red, thick, ->] (2, 4) -- (0, 4);
\draw[red, thick] (0, 4) -- (-2, 4);
\draw[red, thick, ->] (-2, 4) -- (-1, 3.5);
\draw[red, thick] (-1, 3.5) -- (0, 3);
\draw[red, thick, ->] (0, 3) -- (0.5, 2.5);
\draw[red, thick] (0.5, 2.5) -- (1, 2);
\draw[blue, fill=blue] (0, 2) circle(0.1cm);
\draw[blue, fill=blue] (0, 3) circle(0.1cm);
\draw[blue, fill=blue] (-1, 3) circle(0.1cm);
\node at (2.2, 4.2) {$s$};
\node at (-2.2, 4.2) {$t$};
\node at (0, -0.3) {$r$};
\node at (1.4, 1.8) {$v_1$};
\end{tikzpicture}
\end{center}
\caption{A sample illustration of case 1 for the red edges.}\label{fig:schnyderlemma3}
\end{figure}
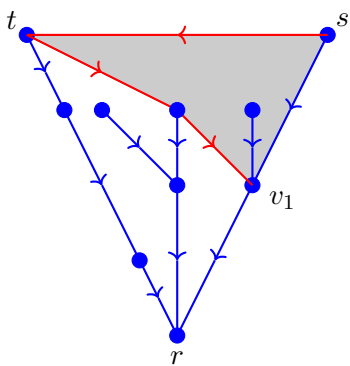

Now applying Schnyder's rule at vertex $v_1$, the green outgoing edge from $v_1$ must point within the gray shaded region, but there is no legal place for it to intersect the boundary, so this configuration is impossible. 

\item Similarly, in case $2$ we must have $t$ the ancestor of $s$ in the blue tree, meaning that we have one of the two diagrams in \cref{fig:schnyderlemma4}.

\begin{figure}[ht]
\begin{center}
\begin{tikzpicture}
\draw[blue, thick, ->] (0, 4) -- (-0.5, 3.5);
\draw[blue, thick] (-0.5, 3.5) -- (-1, 3);
\draw[blue, thick, ->] (-1, 3) -- (-1, 2.5);
\draw[blue, thick] (-1, 2.5) -- (-1, 2);
\draw[blue, thick, ->] (-1, 2) -- (-1, 1.5);
\draw[blue, thick] (-1, 1.5) -- (-1, 1);
\draw[blue, thick, ->] (-1, 1) -- (-0.5, 0.5);
\draw[blue, thick] (-0.5, 0.5) -- (0, 0);
\draw[blue, thick, ->] (-0.5, 2) -- (-0.8, 2);
\draw[blue, thick] (-0.8, 2) -- (-1, 2);
\draw[red, thick, ->] (0, 4) -- (0, 2);
\draw[red, thick] (0, 2) -- (0, 0);
\draw[blue, fill=blue] (0, 0) circle(0.1cm);
\draw[blue, fill=blue] (0, 4) circle(0.1cm);
\draw[blue, fill=blue] (-1, 3) circle(0.1cm);
\draw[blue, fill=blue] (-1, 2) circle(0.1cm);
\draw[blue, fill=blue] (-1, 1) circle(0.1cm);
\draw[blue, fill=blue] (-0.5, 2) circle(0.1cm);
\node at (0, 4.3) {$s$};
\node at (0, -0.3) {$t$};

\draw[blue, thick, ->] (2, 4) -- (2.5, 3.5);
\draw[blue, thick] (2.5, 3.5) -- (3, 3);
\draw[blue, thick, ->] (3, 3) -- (3, 2.5);
\draw[blue, thick] (3, 2.5) -- (3, 2);
\draw[blue, thick, ->] (3, 2) -- (3, 1.5);
\draw[blue, thick] (3, 1.5) -- (3, 1);
\draw[blue, thick, ->] (3, 1) -- (2.5, 0.5);
\draw[blue, thick] (2.5, 0.5) -- (2, 0);
\draw[blue, thick, ->] (2.5, 2) -- (2.8, 2);
\draw[blue, thick] (2.8, 2) -- (3, 2);
\draw[red, thick, ->] (2, 4) -- (2, 2);
\draw[red, thick] (2, 2) -- (2, 0);
\draw[blue, fill=blue] (2, 0) circle(0.1cm);
\draw[blue, fill=blue] (2, 4) circle(0.1cm);
\draw[blue, fill=blue] (3, 3) circle(0.1cm);
\draw[blue, fill=blue] (3, 2) circle(0.1cm);
\draw[blue, fill=blue] (3, 1) circle(0.1cm);
\draw[blue, fill=blue] (2.5, 2) circle(0.1cm);
\node at (2, 4.3) {$s$};
\node at (2, -0.3) {$t$};
\end{tikzpicture}
\end{center}
\caption{A sample illustration of case 2 for the red edges.}\label{fig:schnyderlemma4}
\end{figure}
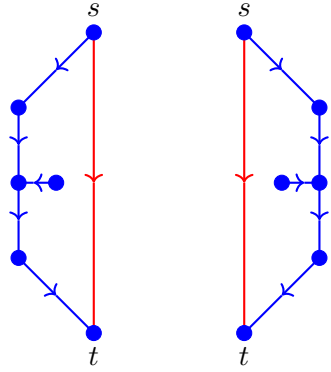

In the left diagram, Schnyder's rule at $t$ forces the blue outgoing edge to point inside the bounded region, and in the right diagram, Schnyder's rule at $s$ forces the green outgoing edge to point inside. Both of these cannot be valid (the former because we form a cycle and the latter because we violate Schnyder's rule), so the edge from $s$ to $t$ again cannot exist.

And even if the red edge were directed from $t$ to $s$ instead, the left case would force the red outgoing edge at $s$ to point inside the region, while the right case would force the blue outgoing edge at $t$ to point inside, which would violate Schnyder's rule and form a cycle in the blue tree, respectively.
\end{itemize}

Putting everything together, we indeed see that vertex numbers decrease as we follow green edges and increase as we follow red edges, and that no red or green edges exist between blue descendants in any cases. This completes the proof. 
\end{proof}

We may now begin proving that edges in a Schnyder wood triangulation correspond to intersections of walks in the two pre-coalescent-walk processes, and we begin by considering edges between internal vertices.

\begin{remark}
In the next three lemmas, we will correspond steps of the driving walk $pW_M$ and its reversal $pW_M'$ with crossings of colored edges by the Schnyder wood loop (recall the gold loop from \cref{schnyderwoodstring} and \cref{fig:sampleschnyderloop}). Thus, on unit integer intervals in which $pW_M$ increments by $(0, 1)$, $(1, -1)$, or $(-1, 0)$, we say that the walk takes a \texttt{g}, \texttt{b}, or \texttt{r} step respectively. Similarly, we correspondingly also use that notation when $pW_M'$ increments by $(0, -1)$, $(-1, 1)$, or $(1, 0)$, respectively.

Additionally, we assume from here on that we label the starting points as in \cref{mainresult2} (so in descending order for $pZ^g_M$ and ascending order for $pZ^r_M$). As suggested by the proposition statement, these labels should be thought of as corresponding to the matching blue labels in the Schnyder wood $M$.
\end{remark}

\begin{lemma}\label{lemmareg}
Let $M$ be a Schnyder wood triangulation of size $n$, and let $1 \le s, t \le n$. If there is a green (resp.\ red) edge from $s$ to $t$ in $M$, then in $pZ^g_M$ (resp.\ $pZ^r_M$), the sample path started at the point labeled $s$ will intersect the starting point labeled $t$, and it will do so before intersecting any other starting point.
\end{lemma}
\begin{proof}
We will write out the proofs for the (pre-)green and (pre-)red processes separately. Because the definitions of the processes are similar, the proofs will also look similar but differ in some important details.

\medskip

By the rules governing the pre-green process (\cref{pregreen}), a sample path may only intersect another starting point from above (that is, by taking a \texttt{g} step when the walk is at height $0.5$). Since a new path is always started in the middle of a \texttt{g} step, such a step will always result in an intersection. Thus, the sample path started at label $s$ intersects label $t$ if and only if it takes the following sequence of steps:
\begin{itemize}
    \item[(a)] the initial \texttt{g} step, during which the sample path labeled $s$ is created and moves to height $-0.5$,
    \item[(b)] a possibly empty string of steps with an equal number of \texttt{r}s and \texttt{b}s, during which the path stays negative and ends up at $-0.5$ (meaning the \texttt{r}s and \texttt{b}s form a Dyck path; the \texttt{g}s don't affect the height of the path),
    \item[(c)] an intermediate \texttt{b} step, in which the path moves to $+0.5$,
    \item[(d)] a possibly empty string of steps with an equal number of \texttt{b}s and \texttt{g}s, during which the path stays positive and ends up at $+0.5$,
    \item[(e)] a final \texttt{g} step, in which the path intersects the sample path labeled $t$.
\end{itemize}

Recall that the pre-green process $pZ^g_M$ is driven by the reversed pre-random walk $pW_M'$, which can be interpreted combinatorially in the following way. Since the increments of $pW_M$ come from traversing a loop around the blue tree and marking down the second visits to green, blue, and red edges (recall \cref{schnyderwoodstring}), the increments of $pW_M'$ can be interpreted as traversing this loop in reverse and marking down the first visits to green, blue, and red edges instead.

This process will visit the vertices in descending numerical order (because the blue vertex labels come from the order in which the forward loop visits them), and \cref{ordering} shows that each green edge points from a larger to a smaller numbered vertex (in particular, $s > t$). Thus, this already allows us to correspond certain steps in the sequence above to particular parts of the loop. Specifically, step (a) above corresponds to the part of the loop crossing the outgoing green edge from vertex $s$, and step (e) corresponds to crossing the outgoing green edge from vertex $t$ (in words, the reversed loop visits each green edge for the first time at the source, because its blue label is larger than the target). Our goal is to show that the remaining sequence of steps ((b), (c), and (d)) does in fact occur if there is a green edge from $s$ to $t$.

First, we claim that step (c) corresponds to the first visit of the incoming blue edge from vertex $t$ (that is, having the loop run parallel to it on one side). Indeed, let $r$ be the closest common parent of $s$ and $t$ in the blue tree (recall this is the first common vertex on the paths from $s$ and from $t$ to the blue root). By the last claim of \cref{ordering}, $r$ will not be either $s$ or $t$, and thus these vertices are organized as in the left diagram of \cref{fig:schnyderlemma5}. The sample path of $pZ^g_M$ labeled $s$ begins when the loop makes the crossing marked ``step (a),'' and we wish to show that it does not cross the $x$-axis until the traversal marked ``step (c).''

\begin{figure}[ht]
\begin{center}
\begin{tikzpicture}
\draw[yellow!50!black, thick, dashed, ->] (-1.35, 3.23) -- (-1.65, 3.83);
\node[yellow!50!black] at (-0.7, 3.5) {step (c)};
\draw[yellow!50!black, thick, dashed, ->] (2, 4)+(60:0.4cm) arc[start angle=60, end angle=215, radius=0.4cm];
\node[yellow!50!black] at (0.9, 4.3) {step (a)};
\draw[blue, thick, ->] (2, 4) -- (1.5, 3);
\draw[blue, thick] (1.5, 3) -- (1, 2);
\draw[blue, thick, ->] (1, 2) -- (0.5, 1);
\draw[blue, thick] (0.5, 1) -- (0, 0);
\draw[blue, thick, ->] (-2, 4) -- (-1.75, 3.5);
\draw[blue, thick] (-1.75, 3.5) -- (-1.5, 3);
\draw[blue, thick, ->] (-1.5, 3) -- (-1, 2);
\draw[blue, thick] (-1, 2) -- (-0.5, 1);
\draw[blue, thick, ->] (-0.5, 1) -- (-0.25, 0.5);
\draw[blue, thick] (-0.25, 0.5) -- (0, 0);
\draw[blue, thick, ->] (0, 3) -- (0, 2.5);
\draw[blue, thick] (0, 2.5) -- (0, 2);
\draw[blue, thick, ->] (1, 3) -- (1, 2.5);
\draw[blue, thick] (1, 2.5) -- (1, 2);
\draw[blue, thick, ->] (0, 2) -- (0, 1);
\draw[blue, thick] (0, 1) -- (0, 0);
\draw[blue, thick, ->] (-1, 3) -- (-0.5, 2.5);
\draw[blue, thick] (-0.5, 2.5) -- (0, 2);
\draw[blue, fill=blue] (0, 0) circle(0.1cm);
\draw[blue, fill=blue] (1, 2) circle(0.1cm);
\draw[blue, fill=blue] (2, 4) circle(0.1cm);
\draw[blue, fill=blue] (-2, 4) circle(0.1cm);
\draw[purple!70!blue, fill=purple!70!blue] (-0.5, 1) circle(0.1cm);
\draw[purple!70!blue, fill=purple!70!blue] (-1.5, 3) circle(0.1cm);
\draw[purple!70!blue, fill=purple!70!blue] (1, 3) circle(0.1cm);
\draw[green!50!black, thick, ->] (2, 4) -- (0, 4);
\draw[green!50!black, thick] (0, 4) -- (-2, 4);
\draw[purple!70!blue, fill=purple!70!blue] (0, 2) circle(0.1cm);
\draw[purple!70!blue, fill=purple!70!blue] (0, 3) circle(0.1cm);
\draw[purple!70!blue, fill=purple!70!blue] (-1, 3) circle(0.1cm);
\node at (2.2, 4.2) {$s$};
\node at (-2.2, 4.2) {$t$};
\node at (0, -0.4) {$r$};
\begin{scope}[shift={(7, 1)}]
\draw[yellow!50!black, thick, dashed, ->] (0.65, -0.77) -- (0.35, -0.17);
\node[yellow!50!black] at (1.3, -0.5) {step (c)};
\draw[yellow!50!black, thick, dashed, ->] (0, 0)+(135:0.4cm) arc[start angle=135, end angle=210, radius=0.4cm];
\node[yellow!50!black] at (-1.2, -0.2) {step (e)};
\draw[green!50!black, thick, ->] (0, 0) -- (-1.5, 0.75);
\draw[green!50!black, thick, ->] (3, 0) -- (1.5, 0);
\draw[green!50!black, thick] (1.5, 0) -- (0, 0);
\draw[red, thick, ->] (0, 0) -- (1.5, 0.75);
\draw[blue, thick, ->] (0.5, 1) -- (0.25, 0.5);
\draw[blue, thick] (0.25, 0.5) -- (0, 0);
\draw[blue, thick, ->] (0, 2) -- (0.25, 1.5);
\draw[blue, thick] (0.25, 1.5) -- (0.5, 1);
\draw[blue, thick, ->] (1, 2) -- (0.75, 1.5);
\draw[blue, thick] (0.75, 1.5) -- (0.5, 1);
\draw[blue, thick, ->] (-0.5, 1) -- (-0.25, 0.5);
\draw[blue, thick] (-0.25, 0.5) -- (0, 0);
\draw[blue, fill=blue] (0, 0) circle(0.1cm);
\draw[blue, thick, ->] (0, 0) -- (0.5, -1);
\node at (-0.2, -0.2) {$t$};
\draw[blue, fill=blue] (3, 0) circle(0.1cm);
\draw[purple!70!blue, fill=purple!70!blue] (0.5, 1) circle(0.1cm);
\draw[purple!70!blue, fill=purple!70!blue] (-0.5, 1) circle(0.1cm);
\draw[purple!70!blue, fill=purple!70!blue] (0, 2) circle(0.1cm);
\draw[purple!70!blue, fill=purple!70!blue] (1, 2) circle(0.1cm);
\node at (3.2, -0.2) {$s$};
\end{scope}
\end{tikzpicture}
\end{center}
\caption{Three key edge traversals of the Schnyder wood loop, along with the steps they correspond to for a particular sample path of $pZ^g_M$ (the one labeled $s$).}\label{fig:schnyderlemma5}
\end{figure}
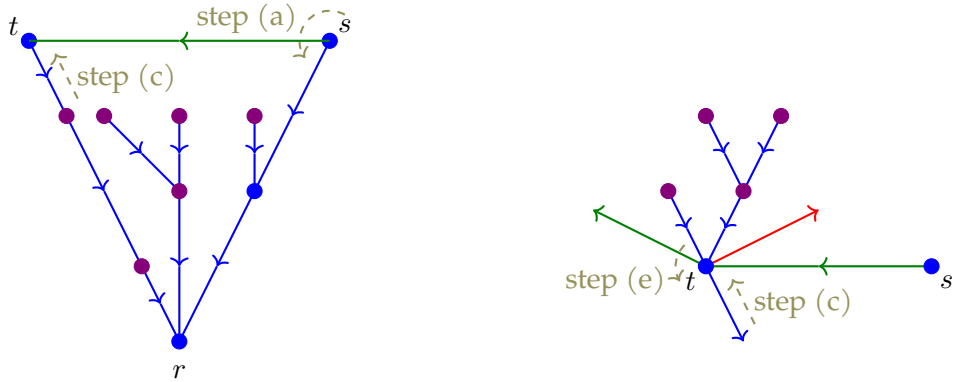

Each vertex in the left-hand side of \cref{fig:schnyderlemma5} (except $r$ if it is the blue root) has exactly one outgoing red edge. Consider the region bounded by the green edge from $s$ to $t$ and the blue paths from $s$ and $t$ to $r$. Schnyder's rule now specifies that the red edges from $t$ and from the vertices on the right boundary (between $s$ and $r$ inclusive) will not be contained within this region, while the red edges from all other vertices (marked in purple) will be. Therefore, the only sources of red edges that will be crossed between steps (a) and (c) are those originating from purple vertices. Furthermore, by \cref{ordering}, each outgoing red edge from a vertex labeled $x$ is connected to some vertex labeled $y > x$. Since the traversal (reversed loop) makes its way around all edges surrounding vertex $y$ before all edges surrounding vertex $x$, this means (for each purple vertex $x$) the outgoing red edge at $x$ is hit before the outgoing blue edge at $x$. Thus while we are within this enclosed region, the number of traversed red edges after step (a) is always at least the number of traversed blue edges. (By Schnyder's rule, any red edges pointing to $s$ lie between the green outgoing and blue outgoing edges, so they are indeed traveled after step (a).)

The only blue edges traversed before step (c) are the blue edges originating from purple vertices. Thus, when the final purple vertex in this region has been traversed, the walk $pW_M$ has taken an equal number of \texttt{r} and \texttt{b} steps, during which our sample path has an equal number of $-1$ and $+1$ increments but always stays below height $-0.5$ (along with some green steps which do not affect the height). After that, the next non-green edge crossed is the one outgoing from $t$, labeled ``step (c)'' (if there are no purple vertices on the left boundary, then this is just the edge from $t$ to $r$), and this will increment the sample path's height from $-0.5$ to $+0.5$. This means that the status of the walk during steps (a), (b), and (c) is verified.

Next, we need to analyze the portion of the path in which the height is positive -- for this, see the right-hand side of \cref{fig:schnyderlemma5}. By Schnyder's rule and the rules of our loop traversal, before we cross the outgoing green edge from $t$, we first traverse the blue subtree that descends from $t$ (whose vertices are marked in purple), and because our walk height is positive, it is now affected by green and blue edges. But in this portion of the loop, each blue and green edge that we traverse for the first time comes from such a purple vertex. Again by \cref{ordering}, any purple vertex $x$ has an outgoing green edge to some vertex $y < x$, meaning that the first time we visit that edge is when we are circling around vertex $x$. But we will do so only after we traverse the outgoing blue edge from $x$, meaning the number of traversed blue edges in this part is always at least the number of traversed green ones. Thus this part of the path corresponds to step (d), a Dyck path of \texttt{b} and \texttt{g} steps (plus some \texttt{r} steps which keep the height constant); when we complete this step, the path is once again at height $0.5$. Finally, crossing the outgoing green edge from $t$ yields step (e), and thus the sample path labeled $s$ will next be at height $0$ exactly at the midpoint of that \texttt{g} step, which is exactly when the sample path labeled $t$ is started. This is exactly what we wish to prove.

\medskip

The proof for the pre-red process $pZ^r_M$ goes similarly, except we now take the Schnyder wood loop in the forward order and return to marking down second visits to edges. By the rules governing the red process, the only way a path hits another starting point is if the former is on the $x$-axis when the latter begins. Furthermore, by Schnyder's rule, a \texttt{b} step cannot follow a \texttt{r} step. This means that once a path arrives on the $x$-axis, it does not leave until hitting the next vertex or having the path end. Thus, the sample path labeled $s$ will hit vertex $t$ if it takes the following sequence of steps:

\begin{itemize}
    \item[(a)] the initial \texttt{g} step, during which the path labeled $s$ is created and moves to position $0.5$,
    \item[(b)] a possibly empty string of steps with an equal number of \texttt{g}s and \texttt{b}s, during which the path stays positive and ends up at $0.5$,
    \item[(c)] an intermediate \texttt{b} step, in which the path moves to $-0.5$,
    \item[(d)] a possibly empty string of steps with an equal number of \texttt{b}s and \texttt{r}s, during which the path stays negative and ends up at $-0.5$,
    \item[(e)] at least one (but possibly more) \texttt{r} steps,
    \item[(f)] a final \texttt{g} step, during which the path intersects the sample path labeled $t$.
\end{itemize}

Again by \cref{ordering}, we know that $s < t$, that step (a) occurs when we cross the outgoing green edge from $s$ (because the forward loop visits each green edge for the second time at the source), and that step (f) occurs when we cross the outgoing green edge from $t$. We wish to check that steps (b) through (e) also occur in the order specified above. First, we claim that step (c) occurs when the Schnyder wood loop crosses the outgoing blue edge from $s$, as shown in the left-hand side of \cref{fig:schnyderlemma6}. 

\begin{figure}[ht]
\begin{center}
\begin{tikzpicture}
\draw[yellow!50!black, thick, dashed, <-] (0.6, -0.8) -- (0.3, -0.2);
\node[yellow!50!black] at (1.2, -0.5) {step (c)};
\draw[yellow!50!black, thick, dashed, <-] (0, 0)+(135:0.4cm) arc[start angle=135, end angle=210, radius=0.4cm];
\node[yellow!50!black] at (-1.1, -0.2) {step (a)};
\draw[green!50!black, thick, ->] (0, 0) -- (-1.5, 0.5);
\draw[red, thick, ->] (0, 0) -- (1.5, 0);
\draw[red, thick] (1.5, 0) -- (3, 0);
\draw[blue, thick, ->] (0.5, 1) -- (0.25, 0.5);
\draw[blue, thick] (0.25, 0.5) -- (0, 0);
\draw[blue, thick, ->] (0, 2) -- (0.25, 1.5);
\draw[blue, thick] (0.25, 1.5) -- (0.5, 1);
\draw[blue, thick, ->] (1, 2) -- (0.75, 1.5);
\draw[blue, thick] (0.75, 1.5) -- (0.5, 1);
\draw[blue, thick, ->] (-0.5, 1) -- (-0.25, 0.5);
\draw[blue, thick] (-0.25, 0.5) -- (0, 0);
\draw[blue, fill=blue] (0, 0) circle(0.1cm);
\draw[blue, thick, ->] (0, 0) -- (0.5, -1);
\node at (-0.2, -0.2) {$s$};
\draw[blue, fill=blue] (3, 0) circle(0.1cm);
\draw[purple!70!blue, fill=purple!70!blue] (0.5, 1) circle(0.1cm);
\draw[purple!70!blue, fill=purple!70!blue] (-0.5, 1) circle(0.1cm);
\draw[purple!70!blue, fill=purple!70!blue] (0, 2) circle(0.1cm);
\draw[purple!70!blue, fill=purple!70!blue] (1, 2) circle(0.1cm);
\node at (3.2, -0.2) {$t$};

\begin{scope}[shift={(7, -1)}]
\draw[yellow!50!black, thick, dashed, <-] (-1.4, 3.2) -- (-1.7, 3.8);
\node[yellow!50!black] at (-0.8, 3.5) {step (c)};
\draw[yellow!50!black, thick, dashed, ->] (2, 4)+(210:0.4cm) arc[start angle=210, end angle=135, radius=0.4cm];
\node[yellow!50!black] at (1, 4.3) {step (e)};
\draw[yellow!50!black, thick, dashed, ->] (2, 4)+(45:0.4cm) arc[start angle=45, end angle=-30, radius=0.4cm];
\node[yellow!50!black] at (3, 4.3) {step (f)};
\draw[green!50!black, thick, ->] (2, 4) -- (3.5, 4);
\draw[red, thick, ->] (2, 5) -- (2, 4.5);
\draw[red, thick] (2, 4.5) -- (2, 4);
\draw[blue, thick, ->] (2, 4) -- (1.5, 3);
\draw[blue, thick] (1.5, 3) -- (1, 2);
\draw[blue, thick, ->] (1, 2) -- (0.5, 1);
\draw[blue, thick] (0.5, 1) -- (0, 0);
\draw[blue, thick, ->] (-2, 4) -- (-1.75, 3.5);
\draw[blue, thick] (-1.75, 3.5) -- (-1.5, 3);
\draw[blue, thick, ->] (-1.5, 3) -- (-1, 2);
\draw[blue, thick] (-1, 2) -- (-0.5, 1);
\draw[blue, thick, ->] (-0.5, 1) -- (-0.25, 0.5);
\draw[blue, thick] (-0.25, 0.5) -- (0, 0);
\draw[blue, thick, ->] (0, 3) -- (0, 2.5);
\draw[blue, thick] (0, 2.5) -- (0, 2);
\draw[blue, thick, ->] (1, 3) -- (1, 2.5);
\draw[blue, thick] (1, 2.5) -- (1, 2);
\draw[blue, thick, ->] (0, 2) -- (0, 1);
\draw[blue, thick] (0, 1) -- (0, 0);
\draw[blue, thick, ->] (-1, 3) -- (-0.5, 2.5);
\draw[blue, thick] (-0.5, 2.5) -- (0, 2);
\draw[blue, fill=blue] (0, 0) circle(0.1cm);
\draw[blue, fill=blue] (1, 2) circle(0.1cm);
\draw[blue, fill=blue] (2, 4) circle(0.1cm);
\draw[blue, fill=blue] (-2, 4) circle(0.1cm);
\draw[purple!70!blue, fill=purple!70!blue] (-0.5, 1) circle(0.1cm);
\draw[purple!70!blue, fill=purple!70!blue] (-1.5, 3) circle(0.1cm);
\draw[purple!70!blue, fill=purple!70!blue] (1, 3) circle(0.1cm);
\draw[red, thick, ->] (-2, 4) -- (0, 4);
\draw[red, thick] (0, 4) -- (2, 4);
\draw[purple!70!blue, fill=purple!70!blue] (0, 2) circle(0.1cm);
\draw[purple!70!blue, fill=purple!70!blue] (0, 3) circle(0.1cm);
\draw[purple!70!blue, fill=purple!70!blue] (-1, 3) circle(0.1cm);
\node at (2.2, 3.8) {$t$};
\node at (-2.2, 4.2) {$s$};
\node at (0, -0.4) {$r$};
\end{scope}
\end{tikzpicture}
\end{center}
\caption{Key steps in the forward traversal of the Schnyder wood loop and the steps they correspond to for a particular sample path of $pZ^r_M$. The diagrams here are similar to those of \cref{fig:schnyderlemma5} but with the loop in the opposite direction.}\label{fig:schnyderlemma6}
\end{figure}
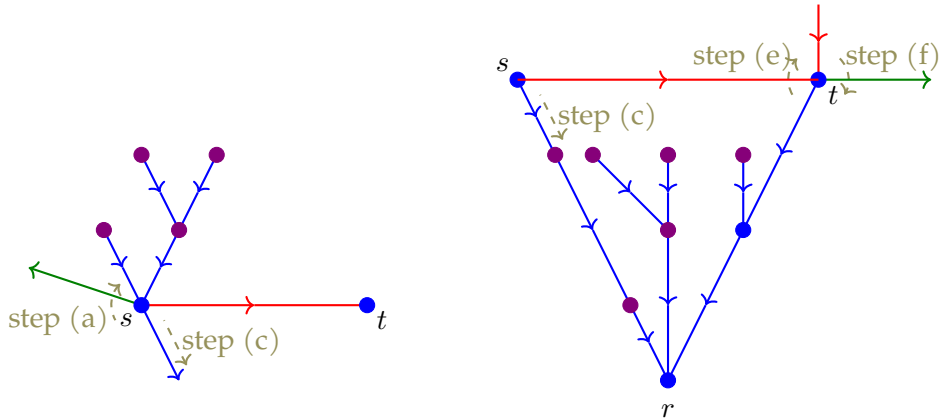

Indeed, for each purple vertex pictured, its outgoing green edge will be crossed (for the second time, while circling around the vertex) before its outgoing blue edge is visited for the second time. Additionally, the green and blue edges originating from purple vertices are the only blue and green edges second-crossed before the edge marked ``step (c).'' So step (b) does indeed conclude after traversing the entire descendant tree of $t$, yielding a Dyck path of \texttt{g}s and \texttt{b}s. Immediately after those steps, the Schnyder wood loop performs step (c), the second edge visit of the outgoing blue edge from $s$; this makes the height of the sample path go from height $0.5$ to height $-0.5$.

Next, to study steps (d), (e), and (f), consider the right-hand side of \cref{fig:schnyderlemma6}. We want to show that what immediately follows step (c) is a Dyck path of blue and red steps. Again, the purple vertices (defined in the same way as in the $pZ^g_M$ case) are the only ones creating red and blue edges traversed within the triangle, and each outgoing blue edge at a vertex is traversed (for the second time) before its corresponding red edge. Thus step (d) -- the edges traversed before crossing the red edge from $s$ to $t$ -- form a Dyck path of \texttt{b} and \texttt{r} steps with some extra green steps as well. After crossing step (e), the red edge from $s$ to $t$, Schnyder's rule implies that we can only encounter other incoming red edges before the loop performs step (f), crossing the outgoing green edge from $t$. So the sample path labeled $s$ will indeed intersect the $x$-axis and stay there until the path labeled $t$ begins, again proving the desired claim.
\end{proof}

The next lemma performs a similar analysis in the case where edges are not just between internal vertices. With this next result, we will have proved that all green and red edges of $M$ correspond appropriately to intersections of sample paths with other starting points.

\begin{lemma}\label{lemmaroot}
Let $M$ be a Schnyder wood  triangulation of size $n$, and let $1 \le s \le n$. If the outgoing green (resp.\ red) edge from vertex $s$ is directed at the green (resp.\ red) root, then in $pZ^g_M$ (resp.\ $pZ^r_M$), the sample path labeled $s$ does not intersect any other starting points.
\end{lemma}

\begin{proof}
First we do the green case -- suppose there is an edge from $s$ to the green root. We wish to prove that (here recall that we traverse the Schnyder wood loop in reverse and keep track of first visits) after crossing the green outward edge from $s$, we always encounter at least as many red edges as blue edges, which will prove that this sample path always stays below the $x$-axis throughout the entire interval on which it is defined. We do this by considering the region traced out by that green edge, as well as the blue path from $s$ to the blue root and the edge between the blue and green root. This is shown in \cref{fig:schnyderlemma7}.
\begin{figure}[ht]
\begin{center}
\begin{tikzpicture}
\draw[white, fill = gray!20!white] (0, 0) -- (0, 4) -- (3, 3) -- (3, 2) -- (1, 1);
\draw[green!50!black, thick, ->] (3, 3) -- (1.5, 3.5);
\draw[green!50!black, thick] (1.5, 3.5) -- (0, 4);
\draw[blue, fill=blue] (3, 3) circle(0.1cm);
\draw[blue, thick, ->] (3, 3) -- (3, 2.5);
\draw[blue, thick] (3, 2.5) -- (3, 2);
\draw[blue, thick, ->] (3, 2) -- (2, 1.5);
\draw[blue, thick] (2, 1.5) -- (1, 1);
\draw[blue, thick, ->] (2, 2) -- (1.5, 1.5);
\draw[blue, thick] (1.5, 1.5) -- (1, 1);
\draw[blue, thick, ->] (1, 1.5) -- (0.5, 0.75);
\draw[blue, thick] (0.5, 0.75) -- (0, 0);
\draw[blue, thick, ->] (1, 2.5) -- (1, 2);
\draw[blue, thick] (1, 2) -- (1, 1.5);
\draw[blue, thick, ->] (1, 1) -- (0.5, 0.5);
\draw[blue, thick] (0.5, 0.5) -- (0, 0);
\draw[blue, fill=blue] (3, 2) circle(0.1cm);
\draw[blue, fill=blue] (1, 1) circle(0.1cm);
\draw[purple!70!blue, fill=purple!70!blue] (1, 1.5) circle(0.1cm);
\draw[purple!70!blue, fill=purple!70!blue] (1, 2.5) circle(0.1cm);
\draw[purple!70!blue, fill=purple!70!blue] (2, 2) circle(0.1cm);
\draw[thick] (0, 0) -- (0, 4);
\node at (3.2, 3.2) {$s$};
\node at (-0.5, -0.35) {blue root};
\node at (-0.5, 4.35) {green root};
\draw[green!50!black, fill=green!50!black] (-0.15, 4.15) |- (0.15, 3.85) |- (-0.15, 4.15);
\draw[blue, fill=blue] (-0.15, 0.15) |- (0.15, -0.15) |- (-0.15, 0.15);
\end{tikzpicture}
\end{center}
\caption{A sample configuration containing a green edge directed to the green root.}\label{fig:schnyderlemma7}
\end{figure}
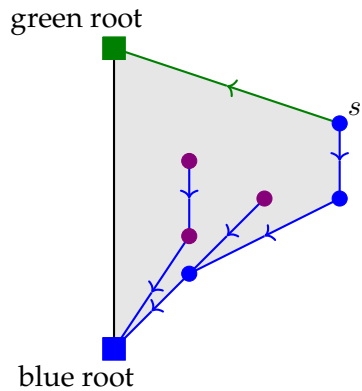

Each blue edge on the blue path from $s$ to the root has already been traversed before we cross the outgoing green edge from $s$, so the only remaining blue edges that are traversed originate from the purple vertices inside the region. But any purple vertex $x$ also has an outgoing red edge to a vertex $y > x$ (again by \cref{ordering}) in this region, so the loop will hit the red edge from $x$ (while traveling around $y$) before it hits the blue edge from $x$. Thus the number of red edges traversed is always at least the number of blue edges while we are within this region. Once the path crosses this green edge a second time, the remaining edges traced out are all green edges around the green root, so our path labeled $s$ does indeed always stay below the $x$-axis.

\medskip

The argument for the red case is similar (now traversing the Schnyder wood loop in the forward direction) but with one small subtlety. Consider the diagram in \cref{fig:schnyderlemma8} (letting $t$ denote the parent of $s$ in the blue tree). Recall from the proof of \cref{lemmareg} that when the Schnyder wood loop traverses the blue edge from $s$ to $t$ on the right side (labeled ``step (c)''), the sample path labeled $s$ crosses below the $x$-axis to a height of $-0.5$ (and does not intersect a starting point when it does so). We wish to prove that the path then subsequently stays below the $x$-axis until after the last green edge is traversed, so there are no more starting points to intersect.

\begin{figure}[ht]
\begin{center}
\begin{tikzpicture}
\draw[white, fill = gray!20!white] (0, 0) -- (1, 1) -- (1, 3) -- (4, 0) -- (0, 0);
\draw[red, thick, ->] (1, 3) -- (2.5, 1.5);
\draw[red, thick] (2.5, 1.5) -- (4, 0);
\draw[green!50!black, thick, ->] (1, 3) -- (0, 3);
\draw[blue, thick, ->] (1, 4) -- (1, 3.5);
\draw[blue, thick] (1, 3.5) -- (1, 3);
\draw[blue, fill=blue] (1, 3) circle(0.1cm);
\draw[blue, thick, ->] (1, 3) -- (1, 2.5);
\draw[blue, thick] (1, 2.5) -- (1, 2);
\draw[blue, thick, ->] (1, 2) -- (1, 1.5);
\draw[blue, thick] (1, 1.5) -- (1, 1);
\draw[blue, thick, ->] (2, 1) -- (1.5, 1);
\draw[blue, thick] (1.5, 1) -- (1, 1);
\draw[blue, thick, ->] (1, 1) -- (0.5, 0.5);
\draw[blue, thick] (0.5, 0.5) -- (0, 0);
\draw[blue, thick, ->] (1.5, 0.5) -- (0.75, 0.25);
\draw[blue, thick] (0.75, 0.25) -- (0, 0);
\draw[purple!70!blue, fill=purple!70!blue] (1, 2) circle(0.1cm);
\draw[purple!70!blue, fill=purple!70!blue] (2, 1) circle(0.1cm);
\draw[purple!70!blue, fill=purple!70!blue] (1, 1) circle(0.1cm);
\draw[purple!70!blue, fill=purple!70!blue] (1.5, 0.5) circle(0.1cm);
\draw[purple!70!blue, fill=purple!70!blue] (2, 1) circle(0.1cm);
\draw[yellow!50!black, dashed, thick, ->] (1.15, 2.9) -- (1.15, 2.1);
\node[yellow!50!black] at (1.9, 2.9) {step (c)};
\draw[thick] (0, 0) -- (4, 0);
\node at (0.7, 3.2) {$s$};
\node at (0.7, 1.8) {$t$};
\draw[yellow!50!black, thick, dashed, ->] (4, 0)+(165:0.4cm) arc[start angle=165, end angle=90, radius=0.4cm];
\node[yellow!50!black] at (4.1, 0.7) {step (e)};
\node at (-0.5, -0.3) {blue root};
\node at (4.5, -0.3) {red root};
\draw[red, fill=red] (3.85, 0.15) |- (4.15, -0.15) |- (3.85, 0.15);
\draw[blue, fill=blue] (-0.15, 0.15) |- (0.15, -0.15) |- (-0.15, 0.15);
\end{tikzpicture}
\end{center}
\caption{A sample configuration containing a red edge directed to the red root.}\label{fig:schnyderlemma8}
\end{figure}

To do this, we must prove that at any point after step (c), the number of blue edges visited within this shaded region is always greater than the number of red edges. Indeed, all red edges we visit before traversing the red edge from $s$ to the root originate from some purple vertex $x$ in \cref{fig:schnyderlemma8} (that is, a vertex in the enclosed region other than $s$ or the blue or red root), but again every such red edge's second visit is preceded by the outgoing blue edge's second-visit from $x$. Thus right before step (e), the sample path is still below the $x$-axis. And from step (e) onward, the only edges left to cross are the red edges directed into the root, none of which correspond to starting a new sample path. Thus again no intersection with another starting point occurs, completing the proof.
\end{proof}

The final lemma required is to show that the ordering of children agrees between the pre-coalescent-walk processes and the trees in the Schnyder wood:

\begin{lemma}\label{lemmaorder}
Let $M$ be a Schnyder wood triangulation, and let $v$ be any vertex of $M$. Then the order of incoming green and red edges into $v$ in the Schnyder wood triangulation agrees with the order of the sample paths in the pre-green and pre-red processes. Specifically, if two children $a$ and $b$ of a vertex $v$ are such that the branch of $a$ is visited before the branch of $b$ by the clockwise exploration of the green (resp.\ red) tree of $M$, then in $pZ^g_M$ (resp.\ $pZ^r_M$), the height of the path labeled $a$ is smaller (resp.\ larger) than that of the path labeled $b$ immediately before they coalesce or reach the end of the interval.
\end{lemma}

Mirroring the examples listed below \cref{mainresult1}, vertices $9, 4, 6$ are the children of vertex $10$ in the red tree of the example Schnyder wood triangulation in \cref{fig:sampleschnyder}, and this is also the order from top to bottom in which the paths coalescing into label $10$ appear in $pZ^r_M$ on the right-hand side of \cref{fig:sampleschnyder3}. Meanwhile, vertices $10, 6$, and $1$ are the children of the green root in the green tree of the example Schnyder wood triangulation in \cref{fig:sampleschnyder}, and this is the order from bottom to top in which those correspondingly labeled paths appear in $pZ^g_M$ on the left-hand side of \cref{fig:sampleschnyder3}. Indeed, the path labeled by 10 coalesces into the path labeled by 6 from below, and the path labeled by 6 reaches the end below the (single-point) path labeled by 1.

\begin{proof}
We again do the green and red cases separately.

For the green tree, suppose $a$ and $b$ both have outgoing green edges pointing to the same vertex $v$, and suppose that (without loss of generality) $a$ is traversed first in the clockwise exploration of $M$'s green tree. Put another way, when sweeping clockwise from $v$'s outgoing red edge to $v$'s outgoing blue edge (or if $v$ is the root, from the green-red edge to the green-blue edge), $a$ shows up first. Returning to the proofs of \cref{lemmareg} and \cref{lemmaroot} (taking $s$ to be $a$), this means vertex $b$ must be one of the other blue or purple vertices shown in either the left-hand side of \cref{fig:schnyderlemma5} or in \cref{fig:schnyderlemma7}. Furthermore, all vertices within this region are numbered smaller than $a$, so we must have $a > b$. (This fact is true for the green case only -- it's only true because vertices in the blue tree are always visited before their descendants in the depth-first traversal. Thus, this step of the proof will be different for the red case.)

Thus, when the sample path labeled $a$ reaches the step corresponding to the path labeled $b$, it has visited at least as many red edges as blue edges (again, because the blue outgoing edge from a vertex is always visited after the red outgoing edge from that same vertex) and thus is at a height of $-0.5$ or lower when the path labeled $b$ starts. This means the branch of the tree for $a$ is below the branch for $b$, meaning that the path labeled $a$ will show up first in a bottom-to-top traversal of the green coalescent walks. 

\medskip

Now for the red tree, suppose both $a$ and $b$ have outgoing red edges pointing to $v$, and suppose $a$ is traversed first in a clockwise exploration of the triangulation's red tree. Put another way, when sweeping clockwise from $v$'s outgoing blue edge to $v$'s outgoing green edge (or if $v$ is the root, from the red-blue edge to the red-green edge), $a$ shows up first. Again considering the proofs of \cref{lemmareg} and \cref{lemmaroot} but now taking $s$ to be $b$, we see that $a$ will be within the enclosed region in the right-hand side of \cref{fig:schnyderlemma6} or in \cref{fig:schnyderlemma8}. However, we now have two cases to consider:

\begin{itemize}
    \item If $a$ is on the path from $b$ to the blue root, then $a < b$ and $b$ is in the descendant tree of $a$. Therefore, the sample path labeled $a$ crosses the outgoing green edge from $b$ during step (b), and it has visited more green edges than blue edges up to that point (the blue edge originating from $x$ is only visited if the green edge originating from $x$ has already been visited). Thus, the path started at $a$ will be at a positive height when the path started at $b$ begins, meaning that $a$ is traversed first in a top-to-bottom traversal of the red process.
    \item If $a$ is not on this path, then $a > b$. When the path labeled $b$ crosses the outgoing blue edge from $b$, it moves from height $0.5$ to height $-0.5$. In this case, the green outgoing edge from $a$ is inside the enclosed region. When the path at $b$ crosses it, it has not yet crossed the outgoing red edge from $b$ to $v$, and for every red edge it does second-cross (from some other purple vertex in the diagram), it crosses the corresponding blue one first. (In other words, we cross during step (d).) Therefore the sample path labeled $b$ will be at a negative height when the path labeled $a$ begins, meaning that again $a$ is traversed first in a top-to-bottom traversal.
\end{itemize}
In either case, the orderings are consistent with the Schnyder wood triangulation, as desired.
\end{proof}

We now prove the main result about our pre-coalescent-walk processes:

\begin{proof}[Proof of \cref{mainresult2}]
Fix any such Schnyder wood triangulation $M$, and let $Z$ denote either $pZ^g_M$ or $pZ^r_M$. To prove that we have a valid total ordering on the starting points in the sense of \cref{totalorder}, it suffices\footnote{Indeed (as has already been pointed out), the same proof strategy as in \cite[Proposition 2.9]{borga2022scaling} shows that both $\le^{\text{up}}$ and $\le^{\text{down}}$ in fact yield total orders on the starting points.} to prove that $Z$ is a coalescent-walk process in the sense of \cref{coalescentdefn} (with the obvious adaptation to half-integer times). That is, we need to check that for any starting points $j_1, j_2$ and any $s \ge j_1, j_2$, if $Z_s^{(j_1)} \ge Z_s^{(j_2)}$, then $Z_{s'}^{(j_1)} \ge Z_{s'}^{(j_2)}$ for all half-integers $s' \ge s$. We do this by induction, looking at the half-integer increments in the definitions of the pre-green and pre-red processes (and thus we only need to study $s' = s + \frac{1}{2}$). Since the evolution of the sample paths at the next time-step depends only on the walk $W_M$ and the current value of the path $Z_s^{(j)}$, once we have $Z_s^{(j_1)} = Z_s^{(j_2)}$, we will have $Z_{s'}^{(j_1)} = Z_{s'}^{(j_2)}$ for all future time. Furthermore, in both cases, the sample paths are always half-integer-valued at all points at which they are defined, and in any half-integer increment the value of $Z_{s'}^{(j_1)} - Z_{s'}^{(j_2)}$ changes by at most $\frac{1}{2}$. Indeed, the only possible increments for $(-X, Y)$ are $(0, 1)$, $(1, 1)$, and $(1, 0)$, so for the pre-green process (recall the update rule in \cref{pregreen}) either the two paths have the same increment, or the upper one changes by $\frac{y}{2}$ and the lower by $-\frac{x}{2}$. And even the additional rule in \cref{prered} will preserve the ordering, since over any such increment all other paths not on the $x$-axis will only move by at most $\frac{1}{2}$. Thus if $Z_{s'}^{(j_1)} - Z_{s'}^{(j_2)}$ begins nonnegative, it will be nonnegative for all time, and this claim is proved.

The remainder of the proof follows from our lemmas. Indeed, recall from the remarks after \cref{coalescentpermutationexample} that $\sigma^{\text{up}}(pZ^g_M)$ can be interpreted as traversing the forest of trees, ordered from bottom to top, and recording the index at which each of the starting points is visited. \cref{lemmareg} and \cref{lemmaroot} imply that the descendants in this description are identical to the descendants in the green tree of $M$. Finally, \cref{lemmaorder} shows that the descendants are in fact ordered from bottom to top in clockwise order, so this permutation in fact also records the indices for the corresponding vertices in the green tree clockwise traversal, meaning it is $\sigma_M^g$. An identical argument but going top to bottom applies for the red tree, so $\sigma^{\text{down}}(pZ^r_M)$ is exactly $\sigma_M^r$. This completes the proof.
\end{proof}

From here, concluding the proof of \cref{mainresult1} involves showing that the coalescent-walk processes $Z^g_M$ and $Z^r_M$ share similarities with the processes $pZ^g_M$ and $pZ^r_M$, in the sense that the ordering of the starting points under \cref{totalorder} remains the same. 

\begin{proof}[Proof of \cref{mainresult1}]
We first show that we have a valid coalescent-walk process; let $Z$ denote either $Z^g_M$ or $Z^r_M$ depending on the case being considered. Much like in the proof above, by induction we only need to verify that for any two starting points $j_1, j_2$ (now integers), if $Z_s^{(j_1)} \ge Z_s^{(j_2)}$, then $Z_{s+1}^{(j_1)} \ge Z_{s+1}^{(j_2)}$. For the green process, this can be verified via the following casework:
\begin{itemize}
    \item If $Z_s^{(j_1)} = Z_s^{(j_2)}$ or if $Z_s^{(j_1)}, Z_s^{(j_2)}$ are both positive or both negative, then the increments $Z_{s+1}^{(j_1)} - Z_s^{(j_1)}$ and $Z_{s+1}^{(j_2)} - Z_s^{(j_2)}$ are equal and thus the inequality still holds.
    \item If $Z_s^{(j_1)} > 0$ and $Z_s^{(j_2)} < 0$, then regardless of the value of $W_{s+1} - W_s$ we always have $Z_{s+1}^{(j_1)} \ge 0$ and $Z_s^{(j_2)} \le 0$, so again the inequality still holds.
    \item If $Z_s^{(j_1)} > 0$ and $Z_s^{(j_2)} = 0$, then either we have a $(-1, 1)$ step and thus both sample paths increment by $1$, or we have a $(k, -1)$ step and thus $Z_{s+1}^{(j_1)} \ge 0$ and $Z_s^{(j_2)} < 0$. 
    \item Finally, if $Z_s^{(j_1)} = 0$ and $Z_s^{(j_2)} < 0$, then either we have a $(-1, 1)$ step and both paths increment by $1$, or we have a $(k, -1)$ step and thus $Z_{s+1}^{(j_1)} = -k-1$ and $Z_{s+1}^{(j_2)} = -k+Z_s^{(j_2)} \le -k-1$ (since all sample paths are integer-valued at integer times).
\end{itemize}
The casework for the red process is similar:
\begin{itemize}
    \item If $Z_s^{(j_1)} = Z_s^{(j_2)}$ or if $Z_s^{(j_1)}, Z_s^{(j_2)}$ are both nonnegative, then the increments $Z_{s+1}^{(j_1)} - Z_s^{(j_1)}$ and $Z_{s+1}^{(j_2)} - Z_s^{(j_2)}$ are equal.
    \item If $Z_s^{(j_1)}, Z_s^{(j_2)}$ are both negative, then either we have a $(1, -1)$ step and thus both paths increment by $-1$, or we have a $(k, -1)$ step and both paths increment by $k$ up to a maximum of height $0$ -- this preserves the inequality.
    \item If $Z_s^{(j_1)} > 0$ and $Z_s^{(j_2)} < 0$, then regardless of the value of $W_{s+1} - W_s$ we always have $Z_{s+1}^{(j_1)} \ge 0$ and $Z_s^{(j_2)} \le 0$.
    \item If $Z_s^{(j_1)} = 0$ and $Z_s^{(j_2)} < 0$, then either we have a $(1, -1)$ step and both paths increment by $-1$ or a $(k, -1)$ step and thus $Z_{s+1}^{(j_1)} = 1$ and $Z_{s+1}^{(j_2)} \le 0$.
\end{itemize}

Thus in both cases the total orders on our starting points are valid, and it remains to check that they agree with the ones formed from the pre-green and pre-red processes. In each case, we will do so by showing that the sample paths of the corresponding starting points are closely related. We recommend consulting \cref{fig:sampleschnyder2} and \cref{fig:sampleschnyder3} side-by-side for the subsequent arguments.

First, we consider the red process. By definition, the driving walk $W_M$ for $Z^r_M$ is obtained by moving the first character of $s_M$ to the end, then grouping together each \texttt{rr}$\cdots$\texttt{rrg} segment into a single $(-k, 1)$ step. Since the driving walk $pW_M$ for $pZ^r_M$ just uses the characters of $s_M$ with no modifications, we will verify that the sample paths of $Z^r_M$ essentially correspond to shifting the sample paths of $pZ^r_M$ down by $0.5$ units and grouping increments together. For this, we associate to each $(-k, 1)$ step its corresponding set of characters \texttt{rr}$\cdots$\texttt{rrg} in $s_M$, and we associate to each $(1, -1)$ step its corresponding \texttt{b}. 

The sample path of $pZ^r_M$ labeled $i$ (for any $1 \le i \le n$) is initialized in the middle of the $i$--th (leftmost) \texttt{g} step and immediately has an increment of $+0.5$ during the latter half of that step. Meanwhile, the sample path of $Z^r_M$ labeled $1$ is initialized at starting time $0$ (because $W_M$ always ends with a $(-k, 1)$ increment and thus the special case of \cref{cwrdefn} always holds), while any other path labeled $i$ is initialized at the right endpoint of the $(i-1)$--th $(-k, 1)$ increment of $W_M$. In either case, under the correspondence in the paragraph above, this starting time corresponds to being initialized at the right endpoint of the $i$--th \texttt{g} step of $pZ^r_M$ (because the first \texttt{g} step was moved to the end while constructing $W_M$, but the starting point does not move along with it). Thus, we can begin tracking how the two paths labeled $i$ evolve and check that they are driven similarly by the corresponding increments of $pW_M$ and $W_M$.
\begin{itemize}
\item When initialized, the path in $Z^r_M$ is at height $0$; at the corresponding time, the path in $pZ^r_M$ is at height $0.5$ because of the latter half of the increment for that \texttt{g} step. Now, track the two paths while they stay nonnegative. Each $(1, -1)$ step in $W_M$ corresponds to a $(1, -1)$ step in $pW_M$ (both causing the path to decrement by $1$ over the next unit interval), and for any $(-k, 1)$ step in $W_M$, the path in $Z^r_M$ moves up by $1$, while the path in $pZ^r_M$ does not change height for $k$ steps and then moves up by $1$. Therefore neither path will touch the $x$-axis again unless the other one also does at a corresponding $(1, -1)$ step, and before that occurs, the height of the path in $pZ^r_M$ will always be $0.5$ larger than the one in $Z^r_M$ at any corresponding time.
\item If the path in $pZ^r_M$ crosses the $x$-axis for the first time, this must correspond to a $(1, -1)$ step in the driving walk $pW_M$. This step will cause the height of the path in $pZ^r_M$ to go from $0.5$ to $-0.5$ and the one in $Z^r_M$ to go from $0$ to $-1$. Now by similar logic as before, now tracking the negative excursions, the two paths will follow the same increments until a $(-k, 1)$ step occurs in $W_M$ when the height of the path in $Z^r_M$ is within the interval $[-k, -1]$, causing it to increment to $0$. This occurs if the height of the path in $pZ^r_M$ before the corresponding \texttt{rr}$\cdots$\texttt{rrg} segment was within the interval $[0.5-k, -0.5]$; for any such height, this sequence of moves in $pZ^r_M$ brings the path to the $x$-axis and finally moves it to $0.5$ in the second half of the \texttt{g} step. So we are back to the case above, still with the path in $pZ^r_M$ at a height $0.5$ larger than that of $Z^r_M$ at corresponding times.
\item Finally, we check that the path labeled $i_1$ intersects the path labeled $i_2$ in $pZ^r_M$ if and only if the path labeled $i_1$ intersects the path labeled $i_2$ in $Z^r_M$. Indeed, this occurs in $pZ^r_M$ if the path labeled $i_1$ is on the $x$-axis in the \emph{middle} of the $i_2$--th \texttt{g} step, but correspondingly in $Z^r_M$ the path labeled $i_1$ is at height $0$ at the \emph{end} of the corresponding $(-k, 1)$ step, which is where the path labeled $i_2$ is initialized. So the intersections will again agree.
\end{itemize}

Therefore, if we have any two starting points labeled $i_1 < i_2$ (say at locations $j_1 < j_2$) in $Z^r_M$ and wish to know the relative order of $\sigma_M^r(i_1)$ and $\sigma_M^r(i_2)$, we must check whether $(Z^r_M)_{j_2}^{(j_1)} > 0$ (in other words, whether $j_1 \le^{\text{down}} j_2$). But we have just verified that paths of $pZ^r_M$ maintain the same intersections and relative height orderings as those of $Z^r_M$ (since they always differ by a height of $0.5$ at corresponding times). Thus this occurs if and only if (say the paths labeled $i_1 < i_2$ in $pZ^r_M$ start at locations $k_1 < k_2$) $k_1 \le^{\text{down}} k_2$ in $pZ^r_M$. So the order on the labels, and thus the red Schnyder wood permutation, is indeed consistent between the two processes. (Moving the initial \texttt{g} to the end in the construction of $W_M$ does not affect this argument, since we have already accounted for the adjustment via our rules for starting point labels, and all paths of $Z^r_M$ start before the final $(-k, 1)$ increment.)

\medskip 

Finally, we check the green process. We now have to consider our driving walk $W_M$ in reverse, and our labels now count from $n$ to $1$ instead of $1$ to $n$ from left to right. Again, we track how the two paths labeled $i$ evolve relative to each other (with $s_M$ now considered in reverse).

The sample path of $pZ^g_M$ labeled $i$ is always initialized in the middle of the $i$--th rightmost \texttt{g} step and immediately has an increment of $-0.5$ during the latter half of that step. The reversed steps are grouped as \texttt{grr}$\cdots$\texttt{rr} segments, and so (if there are $k$ \texttt{r}s) the path will be at height $-k - 0.5$ at the end of this segment. Meanwhile, the sample path of $Z^g_M$ labeled $1$ is initialized at starting time $2n$ (since again the special case of \cref{cwrdefn2} always holds), while any other path labeled $i$ starts off at the left endpoint of the $(i-1)$--th rightmost $(k, -1)$ increment of $W_M'$ and then immediately has an increment of $-k - 1$. Thus at the end of this grouped segment, the sample path in $pZ^g_M$ will be $0.5$ units higher than the corresponding path in $Z^g_M$ (except for the starting point labeled $1$, since there are no steps for it to take in $Z^g_M$).

The argument now proceeds very similarly to the one for the red process -- we will again show that the two paths evolve similarly. During the initial negative excursion, the increments of each step in $Z^g_M$ are identical to those in the corresponding step(s) of $pZ^g_M$, and we will only have a crossing if during a $(-1, 1)$ step, the path in $Z^g_M$ goes from height $-1$ to $0$ and the path in $pZ^g_M$ goes from height $-0.5$ to $0.5$. The increments are then again identical while both paths stay nonnegative until the path in $Z^g_M$ potentially crosses the $x$-axis again. If this occurs, it must correspond to a \texttt{g} step in $pZ^g_M$ and a $(-k, 1)$ step in $Z^g_M$, such that the height of the path in $pZ^g_M$ goes from $0.5$ to $-0.5$ and the one in $Z^g_M$ goes from $0$ to $-1$. But indeed a new path is created in the \emph{middle} of the step in $pZ^g_M$ and at the \emph{beginning} of the corresponding step in $Z^g_M$, so the intersections with starting points again match up. From here, we are back in a negative excursion and (by the logic of the previous paragraph) still have corresponding paths offset by $0.5$ at corresponding times.

So just like before, if we have any two starting points labeled $i_1 < i_2$ (say at locations $j_1 > j_2$) in $Z^g_M$ and wish to know the relative order of $\sigma_M^g(i_1)$ and $\sigma_M^g(i_2)$, we must check whether $(Z^g_M)_{j_1}^{(j_2)} < 0$ (in other words, whether $j_2 \le^{\text{up}} j_1$). We have just verified that paths of $pZ^g_M$ maintain the same intersections and relative height orderings as those of $Z^g_M$, so this occurs if and only if (say the walks labeled $i_1 < i_2$ in $pZ^g_M$ start at locations $k_1 > k_2$) $k_2 \le^{\text{up}} k_1$ in $pZ^g_M$. (Similarly, moving the final \texttt{g} to the beginning in the construction of the reversed walk $W_M'$ does not affect this argument, since we have already accounted for the adjustment via our rules for starting point labels. However, this time we check separately that the relations with $i_1 = 1$ also agree, since the sample paths of label $1$ do not reach any ``corresponding times'' before the end of the interval. The logic for intersecting starting point $1$ works identically as for any other point. And checking relative orderings, we have $k_2 \le^{\text{up}} k_1$ in $pZ^g_M$ if and only if at the left endpoint of the last \texttt{g} step, the sample path with label $i_2$ is at a height below $0.5$. But under the correspondence, this occurs if and only if the sample path with label $i_2$ in $Z^g_M$ is at a height below $0$ at the end of the interval, which is exactly the rule for $j_2 \le^{\text{up}} j_1$.) So the order on the labels, and thus the green Schnyder wood permutation, is indeed also consistent between the two processes, completing the proof.
\end{proof}

\bibliographystyle{hmralphaabbrv}
\bibliography{biblio}

\end{document}